\documentclass[12pt]{report}

\usepackage[noguesstabularheaders]{latexml}
\usepackage{array}
\usepackage{graphicx}
\usepackage{amssymb}
\usepackage{amsmath}

\iflatexml
  \usepackage{makecell}
  \lxRequireResource[type=text/css,content={
    .ltx_tocentry_appendix > a .ltx_tag_ref:before { content: "Appendix "; }
    .ltx_tabular .ltx_tabular { width: auto; }
  }]{}
\else
  \usepackage{uqthesis}
  \usepackage[a-2b]{pdfx}
\fi
\usepackage{hyperref}

\newtheorem{theorem}{Theorem}
\newtheorem{algorithm}{Algorithm}
\newtheorem{lemma}{Lemma}

\newtheorem{conjecture}{Conjecture}
\newtheorem{question}{Question}

\newtheorem{preproof}{{\bf Proof.}}

\newenvironment{proof}[1]{\begin{preproof}{\rm
               #1}\hfill{\rule[-0.5mm]{2mm}{2mm}}}{\end{preproof}}

\def\lcs#1{{\rm lcs}$(#1)$}
\def\scs#1{{\rm scs}$(#1)$}

\begin{document}

\title{{\bf Critical Sets in Latin Squares\\
        and\\
        Associated Structures}} 
\author{ Richard Winston Bean, B.Sc. (Hons) (UQ) \\[1.0cm]
 {\em A thesis submitted to}\\
 {\em the Department of Mathematics at}\\
 {\em The University of Queensland} \\
 {\em in fulfilment of the requirements for}\\
 {\em the degree of Doctor of Philosophy.}
\vspace{20mm}\\
The University of Queensland  \\
St. Lucia, Queensland, Australia
\vspace{8mm}\\
April, 2001}
\date{}
\maketitle

{\bf {\huge Statement of Originality}} \\

I declare that the work presented in this thesis is, to the best of my
knowledge and belief, original and my own work, except as acknowledged
in the text, and that this material has not been submitted, either in
whole or in part, for a degree at this or any other university.

Chapter~\ref{ch4} is joint work with Ebadollah Mahmoodian, submitted for publication
~\cite{me3}.

Parts of Chapter~\ref{ch5} are based on discussions with Ian Wanless.  This is
explained more fully in the text.

Chapter~\ref{ch6} is joint work with Diane Donovan and is published
in ~\cite{me1}.  

Chapter~\ref{ch7} is joint work with Diane Donovan, Abdollah Khodkar, and Anne
Street, and is published in ~\cite{me2}.

Chapter~\ref{ch8} is joint work with Peter Adams and Abdollah Khodkar, submitted for publication ~\cite{abk} and to appear in ~\cite{disj}.

\vskip 50mm

Richard Winston Bean
\newpage
{\bf {\huge Acknowledgements}}

Father and Mother for their love and support;

Diane Donovan for her ideas and for reading through countless drafts
of various chapters of my thesis;

Elizabeth Billington for some proof-reading;

Ebadollah Mahmoodian for always being interested in my crazy ideas 
and showing me related problems I would not have otherwise come across;

Peter Adams and Abdollah Khodkar for collaborating on the final paper;

Donald Arseneau and Neil Williams for technical help with \LaTeX;

Anne Street for lending me books no-one else could afford;

John Nelder, Peter Owens, David Bedford, Vassili Mavron, Tom McDonough,
Donald Preece, Ian Anderson, Donald Keedwell, Ian Wanless, and Anthony Hilton
for discussions all around the United Kingdom in July and August, 2000;

Ken Gray for miscellaneous thesis help;

Brendan McKay for {\it nauty} help, main class data and ideas for
intercalate-rich Latin squares;
 
Chris Rodger for helpful suggestions at conferences;

Jennie Seberry for her enthusiasm and many ideas;

Eric Mendelsohn for correspondence on greedy critical sets;

Nick Cavenagh for giving me a reason to finish more quickly;

John van Rees and Mohammad Mahdian for some more proof-reading;

Qscgz, Terry Ritter, Kathy Heinrich, Wal Wallis, John Bate, Ljiljana
Brankovi{\'c}, Adelle Howse, and Julie Lawrence for miscellaneous discussions.

\tableofcontents
\listoffigures
\listoftables
\chapter*{Abstract}
A critical set in a Latin square of order $n$ is a set of entries in an
$n \times n$ array which can be embedded in precisely one Latin square
of order $n$, with the property that if any entry of the critical set
is deleted, the remaining set can be embedded in more than one Latin
square of order $n$.

The number of critical sets grows super-exponentially as the order of the
Latin square increases.  It is difficult to find patterns in Latin squares
of small order (order 5 or less) which can be generalised in the process
of creating new theorems.  Thus, I have written many algorithms to find
critical sets with various properties in Latin squares of order greater
than 5, and to deal with other related structures.  
Some algorithms used in the body of the thesis are presented in Chapter~\ref{ch3}; 
results which arise from the computational studies and observations
of the patterns and subsequent results are presented in Chapters~\ref{ch4}, \ref{ch5}, \ref{ch6}, \ref{ch7} and \ref{ch8}.

The cardinality of the largest critical set in any Latin square of
order $n$ is denoted by \lcs{n}.  In 1978 Curran and van Rees proved
that \lcs{n} $\leq n^2 - n$.  In Chapter~\ref{ch4}, it is shown that \lcs{n}
$\leq n^2-3n+3$.

Chapter~\ref{ch5} provides new bounds on the maximum number of intercalates in
 Latin squares of orders $2^\alpha m$
($m$ odd, $\alpha \geq 2$) and $2^\alpha m + 1$ ($m$ odd, $\alpha \geq
2$ and $\alpha \neq 3$), and a new lower bound on \lcs{4m}.
It also discusses critical sets in intercalate-rich
Latin squares of orders 11 and 14.

In Chapter~\ref{ch6} a construction is given which verifies the existence of a critical 
set of size $\displaystyle{\frac{n^2}{4}} + 1$ when $n$ is even and $n \geq 6$.
The construction is based on the
discovery of a critical set of size 17 for a Latin square of order 8.

In Chapter~\ref{ch7} the representation of Steiner trades of volume less than
or equal to nine is examined.  Computational results are used to identify
those trades for which the associated partial
Latin square can be decomposed into six disjoint Latin interchanges.

Chapter~\ref{ch8} focusses on critical sets in Latin squares of order at most six and 
extensive computational routines are used to identify all the
critical sets of different sizes in these Latin squares.

\chapter{Introduction}\label{ch1}
This thesis examines the combinatorial structure of the ``Latin square''
and related ideas.  A ``Latin square'' can be thought of as a set of
ordered triples having certain properties.  The first known written
reference on this combinatorial structure was in
1723 \cite{dk}.  A Latin square of order $n$ is most commonly described
as an $n \times n$ array of symbols from a set $N$ of cardinality $n$
such that each symbol from the set $N$ occurs once in each row and
column.  One of the earliest problems relating to Latin squares, the
``Thirty-six Officers Problem'', was stated by Euler in 1779 \cite{dk}.
In this thesis, the topic under examination is the concept of subsets
of a Latin square which contain just enough information to generate the
complete Latin square.  These subsets are known as ``critical sets''.

The name ``critical set'' and the concept were invented by a statistician,
John Nelder, in 1977, and his ideas were first published in a note
\cite{Nel1}.  This note posed the problem of giving a formula for the
size of the largest and smallest critical sets for a Latin square of a
given order.

The initial theory of critical sets was expanded by authors such
as Curran, van Rees, Smetaniuk, C. Colbourn, M. Colbourn, and
Stinson \cite{MR80j:05022,MR81d:05015,MR84g:05036,MR85k:68035}
between 1978 and 1983.  After eight years of silence,
the topic was re-examined in a paper by Cooper, Donovan and
Seberry in 1991 \cite{MR92i:05049}.  Since then, the topic has
been prolifically covered by many authors; for instance, Donovan
\cite{MR1758263,MR99k:05038,MR99i:05039,MR99a:05018,MR98a:05030,MR97g:05032,
MR95k:05030,MR95j:05043}, Keedwell \cite{
MR99j:05036,MR97i:05017,MR97f:05031,
MR96a:05027,MR1605114}, and Mahmoodian
\cite{MR2000b:05023,MR99i:05084,MR99b:05060,MR98e:05019,MR98b:05044}.

More recently, critical sets have been put forward as a possible
secret-sharing scheme in Street \cite{MR93c:05026}, Cooper, Donovan
and Seberry \cite{MR95j:05043} and Seberry and Street \cite{MR1760181}.  Latin squares
have been used in cryptographic contexts in papers such as \cite{desv}
and \cite{MR98i:94027} and critical sets would be a useful way of reducing
the storage space required for the Latin squares.

Chapter~\ref{ch2} provides definitions which will be used in the thesis and gives
the appropriate background information which has been presented in
these papers.

In order to discover critical sets with unusual properties, it is crucial
to write efficient algorithms, to use fast computers and to experiment
with unorthodox approaches.  As the computational aspects of my work
underlie the rest of the thesis, Chapter~\ref{ch3} is devoted to the presentation
of some of the algorithms used in the rest of the thesis.  Chapters~\ref{ch4}
to \ref{ch8} then present results arising from the use of these algorithms.

Critical sets are complex structures and we are only just beginning
to understand them.  Data generated by comprehensive computer
searches for critical sets with particular properties has helped us
to generate much of the knowledge we now have about these structures.
However, research has shown that computer analysis of critical sets,
defining sets and premature partial Latin squares is computationally
expensive (see Colbourn \cite{MR85d:05055} and Colbourn, Colbourn and
Stinson \cite{MR85k:68035}).  Thus, it has been useful to write fast and
efficient algorithms in order to generate critical sets in Latin squares
of non-trivial orders.  These algorithms are documented in Chapter~\ref{ch3},
and have aided the discovery of patterns in such Latin squares.  For example,
the development of the main theorem in Chapter~\ref{ch6} required the generation of a critical set
of order 8 and size 17.  Results for specific orders of Latin squares
have guided the development of general results and conjectures which
are given in Chapters~\ref{ch4} to \ref{ch8}.
%

The cardinality of the largest critical set in any Latin square of
order $n$ is denoted by \lcs{n}.  In 1978 Curran and van Rees proved
that \lcs{n} $\leq n^2 - n$.  In Chapter~\ref{ch4}, it is shown that \lcs{n}
$\leq n^2-3n+3$.  This is joint work with Ebadollah Mahmoodian, and has
been submitted for publication (see \cite{me3}).

In Chapter~\ref{ch4}, we also show that the constructions for the largest known
critical sets are closely related to constructions
for Latin squares containing the largest known number of intercalates for a
given order.  This connection is expanded upon in Chapter~\ref{ch5}.  Chapter~\ref{ch5}
is based on discussions with Ian Wanless, and gives new bounds for the
maximum number of intercalates in Latin squares of orders $2^\alpha m$
($m$ odd, $\alpha \geq 2$) and $2^\alpha m + 1$ ($m$ odd, $\alpha \geq 2$
and $\alpha \neq 3$), and a new lower bound on \lcs{4m}.
We also discuss critical sets in intercalate-rich
Latin squares of orders 11 and 14.

Later papers such as \cite{MR1758263} introduce the idea of verifying
the existence of certain possible sizes of critical sets, instead
of just looking for the upper and lower bounds.  In the cited paper,
Donovan and Howse proved that for all $n$ there exist critical sets of
order $n$ and size $s$, where $\lfloor \displaystyle{\frac{n^{2}}{4}}
\rfloor \leq s \leq \displaystyle{\frac{n^{2}-n}{2}}$ with the exception
of the case $s = \displaystyle{\frac{n^{2}}{4}} + 1$ when $n$ is even.
In Chapter~\ref{ch6} a construction is presented for this exception, where $n
\geq 6$.  It is based on the discovery of a critical set of size 17
for a Latin square of order 8.  This verifies that there does exist
a critical set of order $n$ and size $\displaystyle{\frac{n^{2}}{4}}$
+ 1 when $n$ is even and $n \geq 6$.  This chapter is joint work with Diane Donovan,
and is published in \cite{me1}.

There is a connection between critical sets in Latin squares,
defining sets in block designs (an analogous idea --- see for example
\cite{MR93c:05026}) and premature partial Latin squares (see Brankovi{\'c},
Hor{\'a}k, Miller and Rosa {\cite{mirka1}}).  However, in the past, critical
sets, defining sets and premature partial Latin squares have been studied
in isolation and, in many cases,  using different techniques.  But as
articles \cite{dks1} and \cite{dks2} showed, there is much to be gained
by studying these configurations in unison. A crucial element in the
identification of defining sets or critical sets is the determination of
interchangeable sets within the design or Latin square.  In designs
these interchangeable sets are known as trades and in Latin squares
as ``Latin interchanges'' (also known as ``critical partial Latin squares''
\cite{MR96a:05027,MR98b:05019}).  So in Chapter~\ref{ch7}, we focus on the
connection between Latin interchanges and trades in designs, and develop
new results which help us classify these structures.

Latin interchanges are particularly important when searching for critical sets
with given properties such as a fixed size, or symmetrical properties,
and are used to establish that certain subsets of Latin squares are critical.
The use of interchanges was important in proving the existence of a
critical set of order $n$ ($n$ even) and size $\displaystyle{\frac{n^2}{4}} + 1$, and
also in the enumeration of critical sets of order at most six (Chapter~\ref{ch8}).

The representation of Steiner trades of volume less than or equal to
nine, provided in Khosrovshahi and Maimani \cite{MR2000g:05030}, is
examined and those for which the associated partial Latin square can be
decomposed into six disjoint Latin interchanges are identified.  This is
joint work with Diane Donovan, Abdollah Khodkar, and Anne Street and has
been published in \cite{me2}.  This research has led to a study of the
inherent nature of these configurations in order to obtain information
for refining  searches and associated algorithms.

Chapter~\ref{ch8} focusses on critical sets in Latin squares of order at most six and all the
critical sets of different sizes in these Latin squares 
are enumerated.  We comment on properties of the numbers of critical
sets found, particularly for the case of order 6 Latin squares, and
establish that \lcs{6} $= 18$.  This chapter is joint work with Peter
Adams and Abdollah Khodkar (see \cite{abk}).

The conclusion, with suggestions for further research, forms Chapter~\ref{ch9}.

Three appendices are provided, giving results which are referred to in
Chapters~\ref{ch4}, \ref{ch6} and \ref{ch8}.
%
%
%
%
%
%
%
\makeatletter
\newcommand{\ecell}{\rule{0.5em}{0pt}}
\newlength{\lsq@cell}
\iflatexml
  \newenvironment{latinsq}[2][0]{%
    \settowidth{\lsq@cell}{#1}%
    \advance\lsq@cell by 1.1em
    \ifdim\lsq@cell<1.8em \lsq@cell 1.8em\fi
    \def\lsq@strut{\rule[-0.3\lsq@cell]{0pt}{\lsq@cell}}%
    \begin{tabular}{|*{#2}{>{\lsq@strut}c|}}%
  }{\end{tabular}}
\else
  \newenvironment{latinsq}[2][0]{%
    \settowidth{\lsq@cell}{#1}%
    \advance\lsq@cell by 1.1em
    \ifdim\lsq@cell<1.8em \lsq@cell 1.8em\fi
    \setlength{\tabcolsep}{0pt}%
    \def\lsq@strut{\rule[-0.3\lsq@cell]{0pt}{\lsq@cell}}%
    \begin{tabular}{|*{#2}{>{\lsq@strut\centering\arraybackslash}p{\lsq@cell}|}}%
  }{\end{tabular}}
\fi
\makeatother
\chapter{Definitions}\label{ch2}
This chapter gives all the basic definitions required in the body of this
thesis, and relevant references are given.

\section{Latin Squares}
An $n \times n$ {\it Latin square} is an $n \times n$ array of symbols
(or elements) chosen from a set $N$ of size $n$ such that each symbol
occurs exactly once in each row and exactly once in each column.  In this thesis,
we take $N =\{ 0, \dots, n-1 \}$ or $N = \{ 1, \dots, n \}$.   The context
in which the numbers occur will always make clear which set is in use.
The positive integer $n$ is known as the {\it order} of the Latin square.

A Latin square can be represented as a set of 3-tuples, or ordered
triples.  These triples will be referred to as {\it entries}.
The first element in a 3-tuple $(i,j;k)$ refers to the row number, $i$, the second to
the column number, $j$, and the third to the symbol, $k$, of $N$ contained
in the cell at the intersection of the row $i$ with the column $j$.  Throughout this thesis
it will be assumed that where an entry in an $n \times n$ Latin square
based on the set of symbols $N = \{ 0, \dots, n-1 \}$
is referred to as a 3-tuple, 
the third element of the tuple has an implicit ``(mod~
$n)$'' after it.  That is, the third element 
falls into the range $0, \dots, n-1$.  One example of an $n \times n$ Latin square is $BC_n =
\{ (i, j; i + j ) \mid 0 \leq i,j \leq n-1 \}$.  This Latin
square is known as the {\it back-circulant} Latin square of order $n$.
It is equivalent to the group table for the group of integers, ${\mathbb Z}_n$.
The Latin square $BC_6$ is depicted overleaf.

\begin{center}
\begin{tabular}{c}
\begin{latinsq}{6}
\hline 0&1&2&3&4&5 \\
\hline 1&2&3&4&5&0 \\
\hline 2&3&4&5&0&1 \\
\hline 3&4&5&0&1&2 \\
\hline 4&5&0&1&2&3 \\
\hline 5&0&1&2&3&4 \\
\hline 
\end{latinsq}\\[3pt]
 $BC_6$
\end{tabular}
\end{center}

Another representation of Latin squares is used in two Russian
sources, an encyclopedia of mathematics \cite{enc} and Sachkov \cite{sachkov}.  
Sachkov calls two mappings (also known
as functions) $\psi : X \rightarrow Y$ and $\phi : X \rightarrow Y$ 
{\it discordant} if for all $x \in X, \psi(x) \neq \phi(x)$.
A mapping $\psi : X \rightarrow Y$ is called a {\it substitution} if
$X = Y$ and $\psi$ is bijective.  Then an $n \times n$
Latin square $L$ is a sequence of 
$n$ mutually discordant substitutions $\phi_{i} = \{ (x,y) \mid x,y \in N \wedge \phi_{i}(x) = y \}$ on a symbol set $N$ of size $n$,
written $L = [ \phi_1, \dots, \phi_n]_n$.  An example of a $3 \times 3$
Latin square in this form is $L = [ \phi_{1}, \phi_{2}, \phi_{3} ]_3$, where $\phi_{1}
= \{ (1,1),(2,2),(3,3) \}$, $\phi_{2} = \{ (1,2),(2,3),(3,1) \}$,  and
$\phi_{3} = \{ (1,3),(2,1),(3,2) \}$.  (In this case, the substitutions are
represented as ordered pairs from $X \times Y$.)


In a similar manner to the 3-tuple defined above, the notation $(i,j)$ 
with reference to a Latin square denotes the {\it cell} or {\it position} which is the intersection of row $i$ and
column $j$ in the $n \times n$ array.  The symbol occurring in a certain position $(i,j)$ in a Latin square $L$
may be written as $L_{ij}$.

A Latin square $L$ is called {\it symmetric} if for all entries $(x, y; z)$ in $L$,
the entry $(y, x; z)$ is also in $L$. 

Similarly, a Latin square $L$ is called {\it totally symmetric}, first defined in \cite{MR80a:05032},
if for all entries $(x, y; z) \in L$, $\{ (y, x; z), (x, z; y), 
(y, z; x), (z, x; y), (z, y; x) \} \subseteq L$. 

A {\it transversal} ~$T$ in an $n \times n$ Latin square $L$ is a
set of $n$ entries from $L$, $\{(r_1,c_1;e_1),\dots,$ $(r_n,c_n;e_n)\}$
such that all rows, columns, and symbols are represented exactly once;
that is, $\{r_1,\dots,r_n$\}, $\{c_1,\dots,c_n\}$, and $\{e_1,\dots,e_n\}$
are each sets of size $n$.

Given a transversal $T = \{ (r_1,c_1;e_1), \dots, (r_n,c_n;e_n) \}$
in an $n \times n$ Latin square $L$, we {\it prolong} $L$ along $T$ to
obtain $L'$.  That is, we form a new Latin square $L'$ of order $n+1$ from
$L$, using the transversal $T$, as follows.  Let $L' = \{ (r_i,c_i;n+1)
\mid 1 \leq i \leq n \} \cup \{ (n+1,n+1;n+1) \} \cup \{ (r_i, n+1;
e_i), (n+1, c_i; e_i) \mid 1 \leq i \leq n \} \cup (L \setminus T)$.
This technique, called {\it prolongation}, will be used in Chapters~\ref{ch3} and \ref{ch5}.
It is also referred to as {\it stripping the transversal} in Lindner
and Rodger \cite{dt}.

The following definition of the {\it direct product} of two Latin squares
is taken from Bedford and Whitehouse \cite{bandw}.

Let $M$ and $L$ be Latin squares of order $m$ and $n$ respectively
with symbols from the sets $ \{ 0, 1, \dots, m-1 \}$ and
$ \{ 0, 1, \dots, n-1 \}$ respectively.  Define $L^r$ to be the array
obtained from $L$ by adding $rn$ to each symbol of $L$, for 
$r = 0, 1, \dots, m-1$.  The direct product of $M$ with $L$ is the
Latin square of order $mn$ constructed by replacing each symbol $r$ in $M$
by the array $L^r$.  This is denoted by $M \times L$.

In Chapters~\ref{ch3} and \ref{ch5}, we shall concisely denote the direct product of $BC_n$
with itself $m$ times as ${\mathbb Z}_n^m$.

\section{Partial Latin Squares}
A {\it partial Latin square} is an $n \times n$ array such that each
symbol from a set $N$ of size $n$ occurs at most once in each row and at
most once in each column.  The number of non-empty positions of the array
is called the {\it size} (or {\it volume}) of the partial Latin square.
The {\it shape} of the partial Latin square is the set of non-empty
positions.  Expressed in set-theoretic terms, if $P$ is a partial Latin
square represented as a set of ordered triples, the size of $P$ is $|P|$ and the shape of $P$ is $S(P) = \{
(i,j) \mid (i,j;k) \in P \}$.  An $n \times n$ partial Latin square containing $n^2$
entries is called a {\it complete Latin square} or just a Latin square.
If a partial Latin square $P$ is a subset of exactly one Latin square $L$
it is said that $P$ is {\it uniquely completable}, or {\it UC} for short.  A {\it completion}
of $P$ is a Latin square $L$ which is a superset of $P$.

For a partial Latin square $P$ in a Latin square $L$ with symbol set $N$, we define the following sets for each
row $i \in N$, column $j \in N$ and symbol $k \in N$. For fixed $i$, let $R_i(P) = \{ k \mid (i, j; k)
\in P \}$; for fixed $j$, let $C_j(P) = \{ k \mid (i, j; k) \in P \}$; and for fixed $k$, let $E_k(P) = \{ (i,j)
\mid (i, j; k) \in P \}$.  So $R_i(P)$ ($C_j(P)$) is the set of symbols which
appear in row $i$ (column $j$) of $P$ and $E_k(P)$ is the set of positions in $P$ where the
symbol $k$ appears.  Then, for each position $(i,j)$, $1 \leq i,j \leq
n$, we define $x_{i,j}(P) = | R_i(P) \cup C_j(P) |$.  The concepts of $R_i(P),
C_j(P), E_k(P)$ and $x_{i,j}(P)$ will help in explaining the ideas behind Chapter~\ref{ch4}, where we use these concepts to
tighten the bound on the largest size of a critical set in a Latin square.

An $m \times m$ {\it subsquare} of a Latin square $L$ with symbol set $N$ is a set $S$ of $m^2$
entries in $L$ such that the sets of first, second and third elements in 
the ordered triples in $S$ contain $m$ different rows,
$m$ different columns and $m$ different symbols respectively.  In formal terms, $|S| = m^2$ and
for all $i, j, k \in N, 
|R_i(S)| = m$ or $|R_i(S)| = 0, |C_j(S)| = m$ or $|C_j(S)| = 0,$ and $|E_k(S)| = m$ or $|E_k(S)| = 0$.

\section{Critical Sets}
A proper subset $P$ of a Latin square $L$ is called a {\it critical set} if
\begin{enumerate}
\item $P$ is uniquely completable, and
\item the omission of any entry in $P$ 
destroys the unique completion property \cite{Nel1}.
\end{enumerate}
For example, in Table~\ref{def-one} above, the partial Latin square $P$ is
a critical set for $BC_3$, since it has unique completion to $BC_3$,
but $P \setminus \{ (1,1;2) \}$ completes
to both $BC_3$ and $L_1$, and $P \setminus \{ (0,0;0) \}$
completes to both $BC_3$ and $L_2$.  $P \setminus \{ (1,1;2) \}$ and
$P \setminus \{ (0,0;0) \}$ each have precisely four completions.

\begin{table}
\begin{center}
\caption{Critical sets and Latin squares of order 3}
\label{def-one}
\begin{tabular}{cccc}
\begin{latinsq}{3}
\hline 0 & \ecell & \ecell \\
\hline  \ecell & 2 & \ecell \\
\hline  \ecell & \ecell & \ecell \\
\hline 
\end{latinsq} &
\begin{latinsq}{3}
\hline 0 & 1 & 2 \\
\hline 1& 2 & 0\\
\hline 2&0 &1 \\
\hline 
\end{latinsq} &
\begin{latinsq}{3}
\hline 0 & 2 & 1 \\
\hline 2 & 1 & 0 \\
\hline 1 & 0 & 2 \\
\hline 
\end{latinsq} &
\begin{latinsq}{3}
\hline 1 & 0 & 2 \\
\hline 0 & 2 & 1 \\
\hline 2 & 1 & 0 \\
\hline 
\end{latinsq} \\[3pt]
 $P$ & $BC_3$ & $L_1$ & $L_2$
\end{tabular}
\end{center}
\end{table}

All of the following definitions, related to ``weak'' or ``strong'' critical
sets of various kinds, will be used in Chapter~\ref{ch8}, where we enumerate and
classify all critical sets of order at most six.
The next two definitions are taken from Bate and van Rees \cite{MR2000g:05034}.

A {\it strong critical set} $C$ for a Latin square $L$ with symbol set $N$ is a critical set
such that there is a sequence of $m = n^2-|C|$ partial Latin squares $C
= P_1 \subset P_2 \subset \dots \subset P_m \subset L$ where for any $i$,
$1 \leq i \leq m-1$, $P_{i} \cup \{ (r_i, c_i; e_i) \} = P_{i+1}$
and $P_i \cup \{(r_i,c_i;e)\}$ is not a partial Latin square for any $e \in N
\setminus \{ e_i \}$. 

A {\it semi-strong critical set} $C$ for a Latin square $L$ with symbol
set $N$ is a critical set such that there is a sequence of $m = n^2 -
|C|$ partial Latin squares $C = P_1 \subset P_2 \subset \dots \subset
P_m \subset L$ where for any $i$, $1 \leq i \leq m-1$, $P_{i} \cup \{
(r_i, c_i; e_i) \} = P_{i+1}$ and one of $P_i \cup \{(r_i,c_i;e)\}$ or $P_i
\cup \{(r,c_i;e_i)\}$ or $P_i \cup \{(r_i,c;e_i)\}$ is not a partial
Latin square for any $e \in N \setminus \{ e_i \}$, or is not a partial
Latin square for any $r \in N \setminus \{ r_i \}$, or is not a partial
Latin square for any $c \in N \setminus \{ c_i \}$ respectively.

A {\it weak critical set} is a critical set which is neither strong nor semi-strong.

In the process of completing the critical set $C$ to the 
Latin square $L$ of order $n$ which it characterizes, we say that the addition
of an entry $t=(r,c;s)$ (where $(r,c)$ is empty in $C$) is {\em forced} (see \cite {MR99j:05036})
in the process of completion of a set $T$ of entries
$(|T|<n^2,\;{\cal C}\subseteq T\subset L)$ to the complete set of
entries which represents $L$, if one of the following holds:
\begin{itemize}
\item [(i)] $\forall r' \neq r$, $\exists z \neq c$ such that 
$(r',z;s)\in T$ or $\exists z \neq s$ such that 
$(r',c;z)\in T$, or 
\item [(ii)] $\forall c' \neq c$, $\exists z \neq r$ such that 
$(z,c';s)\in T$ or $\exists z \neq s$ such that 
$(r,c';z)\in T$, or
\item [(iii)] $\forall s' \neq s$, $\exists z \neq r$ such that 
$(z,c;s')\in T$ or $\exists z \neq c$ such that 
$(r,z;s')\in T$.
\end{itemize}
A critical set is called {\em totally weak} if no entry is forced.

The following extension of the concept of the semi-strong critical set is
taken from Bedford and Whitehouse \cite{bandw}.  To give the definition
of a near-strong critical set, we need to give a definition of a {\it
conjugate} of a partial Latin square, which will be expanded upon
in Section~\ref{sec27}.

If $\{ a,b,c \} = \{ 1,2,3 \}$, then the $(a,b,c)-conjugate$ of $P$
is denoted and defined by $P_{(a,b,c)} = \{ (x_a, x_b; x_c) \mid (x_1, x_2; x_3) 
\in P \}$.  For $\theta \in S_3$, the symmetric group on $\{1,2,3\}$,
we define $\theta(x_1,x_2,x_3)=(x_{\theta(1)},x_{\theta(2)},x_{\theta(3)})$.

Let $P$ be a partial Latin square of order $n$ defined on a symbol set $N$.  
Then $A_P$ is an {\it array of alternatives} for $P$ if
\begin{enumerate}
\item $A_P$ is an $n \times n$ array ;
\item whenever the $(i,j)^{\rm {th}}$ cell of $P$ is filled, the  
$(i,j)^{\rm {th}}$ cell of $A_P$ is empty; and
\item whenever the $(i,j)^{\rm {th}}$ cell of $P$ is empty, the 
$(i,j)^{\rm {th}}$ cell of $A_P$ contains all the symbols of $N$ which do 
not appear in the $i^{\rm {th}}$ row or $j^{\rm {th}}$ column of $P$.
\end{enumerate}
We denote the set of symbols in cell $(i,j)$ of $A_P$ by ${A_P}(i,j)$. 
{\rm
Let $P$ be a partial Latin square. We shall say that the symbol  
$k'\in A_P(i,j)$ is {\it forced out} of $A_P$ if either:
\begin{itemize}
\item [(1)] there exists $r>0$ and $i_1,i_2,\ldots,i_r$ (all $\neq i$) 
with $k'\in A_P(i_1,j)\cup\ldots\cup A_P(i_r,j)$ and 
$|A_P(i_1,j)\cup\ldots\cup A_P(i_r,j)|=r$; or 
\item [(2)] $\theta(i,j,k')$ satisfies $1$ in $A_{P_{\theta(1,2,3)}}$ for
some $\theta\in S_3$.
\end{itemize}

The reduced array of alternatives,
$RA_P$, is the array obtained from $A_P$ by successively removing
symbols which are forced out until no more symbols can be forced out.
Then the addition
of an entry $(i,j;k)$ to $P$ is said to be {\it semi-forced} if either:

\begin{enumerate}
\item $k$ is the only symbol in $RA_P(i,j)$; or
\item $k$ occurs exactly once in either the $i^{\rm {th}}$ row or
$j^{\rm {th}}$ column of $RA_P$.
\end{enumerate}
}

Note that if a triple is forced it is also semi-forced.

For example, consider the partial Latin square $A$ given in 
Table~\ref{def-two} above.  We examine
the $(1,3,2)$-conjugate of $A$, $A_{(1,3,2)}$.  The symbols 1 and 6 occur in
some order at the positions $(2,3)$ and $(3,3)$ of $A_{(1,3,2)}$, and
the position $(6,3)$ of $A_{(1,3,2)}$ must contain either 4 or 6.
Thus, the position $(6,3)$ of $A_{(1,3,2)}$ is forced to contain 4.
This is because the entry $(6,3;6)$ is forced out of the array
of alternatives for $A_{(1,3,2)}$.

Therefore, we say that the addition of the triple $(6,4;3)$ to $A$ is
semi-forced.

\begin{table}
\begin{center}
\caption{Example of a semi-forced entry in a partial Latin square $A$}
\label{def-two}
\begin{tabular}{ccc}
\begin{latinsq}{6}
\hline   \ecell & \ecell & \ecell & 4 & \ecell & \ecell \\
\hline   \ecell & 1 & 4 & 5 & 6 & \ecell \\
\hline   \ecell & 4 & 2 & 6  & \ecell & \ecell \\
\hline 4  & \ecell & \ecell & \ecell & 3 & \ecell \\
\hline   \ecell & \ecell & \ecell & \ecell & 2 & 4 \\
\hline 6 & 5 & 1 & \ecell & 4 & \ecell \\
\hline
\end{latinsq} &

\begin{latinsq}{6}
\hline  \ecell & \ecell & \ecell & 4 & \ecell & \ecell \\
\hline 2 & \ecell & \ecell & 3 & 4 & 5 \\
\hline   \ecell &3  & \ecell &2  & \ecell & 4 \\
\hline   \ecell & \ecell &5  &1  & \ecell & \ecell \\
\hline   \ecell &5  & \ecell &6  & \ecell & \ecell \\
\hline 3 & \ecell & \ecell & 5 & 2 & 1 \\
\hline
\end{latinsq} \\[3pt]
 $A$ & $A_{(1,3,2)}$
\end{tabular}
\end{center}
\end{table}

{\rm A UC set $U$ is {\it near-strong} UC to the Latin square $L$
if we can find a sequence of sets of triples 
$U=S_1\subset S_2\subset ... \subset S_f = L$ such that each triple
$t\in S_{v+1}\setminus S_v$ is semi-forced in $S_v$, where $1 \leq v \leq f-1$.
}

We call a UC set {\it Bedford-Whitehouse totally weak} if no
entry is semi-forced.  If a UC set is Bedford-Whitehouse totally weak,
this implies that it is also totally weak.

In Keedwell's terminology in \cite{MR99j:05036}, the phrase `strong
critical set' is equivalent to Bate and van Rees's semi-strong concept,
and a weak critical set is one which is not strong, which is equivalent
to the definition of Bate and van Rees.  The terminology of Keedwell
will be used in Chapter~\ref{ch8}.

A parallel concept to total symmetry for Latin squares exists for critical sets.
A critical set $C$ is called {\it totally symmetric} if
for all entries $(x, y; z) \in C$, \break $\{ (y, x; z), (x, z; y), (y, z; x), 
(z, x; y), (z, y; x) \} \subseteq C$ \cite{pre1}.

\section{Latin Interchanges}
Latin interchanges are subsets of Latin squares which are most often used 
in the process of determining whether a given subset of a Latin square
is a critical set.  Their use greatly speeds up this process, as testing
whether several Latin interchanges intersect a given set is a much faster
process than attempting to determine whether a given set has unique
completion.

A {\it Latin interchange} in an $n \times n$ Latin square $L_{1}$ is the set difference
between it and another $n \times n$ Latin square $L_{2}$; that is, $L_{1} \setminus L_{2}$.
This is the most concise definition to appear in the literature, and is used in \cite{njcint}.
A longer definition is given in papers such as \cite{MR96a:05027},
where Latin interchanges are known as critical partial Latin squares,
and \cite{MR98b:05019}.  The definition from \cite{MR98b:05019} follows.

Consider two partial Latin squares $L$ and $M$ of order $n$ with symbol
set $N$ which have the same size and shape.  These are said to be {\it
disjoint} if $L_{ij} \neq M_{ij}$ for all $i,j \in N$, and {\it mutually
balanced} if, for each column $c$ of $L$, the set of symbols in column $c$
of $L$ is equal to the set of symbols in column $c$ of $M$, and for each
row $r$ of $L$, the set of symbols in row $r$ of $L$ is equal to the
set of symbols in row $r$ of $M$.  Formally, $L$ and $M$ are mutually
balanced if for all $r,c \in N$, $R_r(L) = R_r(M)$ and $C_c(L) = C_c(M)$.

It is said that $M$ is a {\it disjoint mate} of $L$ if $L$ and $M$ are disjoint and 
mutually balanced.
Then a {\it Latin interchange} is a partial Latin square $P$
such that there exists a disjoint mate, $P'$ of $P$.
At times it will be useful to emphasise the connection between a Latin
interchange and its disjoint mate.  This will be particularly true
in Chapter~\ref{ch7}.  So in Chapter~\ref{ch7} we shall refer to a Latin interchange as a pair
of partial Latin squares $(P, P')$ and it will be assumed that $P$
and $P'$ are the same size and shape, and are disjoint and mutually balanced.

An example of a Latin interchange $I$ of order 3 and size 7 and its disjoint mate $I'$ is given below.
\begin{center}
\begin{tabular}{cc}
\begin{latinsq}{3}
\hline
 \ecell &2& 3  \\ \hline
1& \ecell & 2  \\ \hline
2&3& 1      \\ \hline
\end{latinsq}
&
\begin{latinsq}{3}
\hline
 \ecell &3&2 \\ \hline
2& \ecell &1 \\ \hline
1&2&3 \\ \hline
\end{latinsq} \\[3pt]
 $I$ & $I'$
\end{tabular}
\end{center}
An {\it intercalate} is a Latin interchange of size 4 \cite{MR1:199a}.
It is also a $2 \times 2$ subsquare.

In {\cite{MR96a:05027}}, Keedwell introduced the definition of the ``type'' of 
a Latin interchange. 
The {\em type} of a Latin interchange $S$ in an $n \times n$ Latin square is given by the following vector:
\begin{eqnarray*}
\left(
\begin{array}{c}
|C_1(S)| + |C_2(S)| + \dots + |C_n(S)| \cr
|R_1(S)| + |R_2(S)| + \dots + |R_n(S)| \cr
|E_1(S)| + |E_2(S)| + \dots + |E_n(S)| \cr
\end{array}
\right).
\end{eqnarray*}
Note that if any of the values $|C_i(S)|$, $|R_i(S)|$ or 
$|E_i(S)|$ in the above vector are zero, then for brevity they are omitted.
The type of the Latin interchange, $I$, given in the above example is
\begin{eqnarray*}
\left(
\begin{array}{c}
 2+2+3\cr
 2+2+3\\
 2+3+2
\end{array}
\right).
\end{eqnarray*}

There is a relationship between critical sets and Latin interchanges.
This relationship can be expressed in the following lemma.

\begin{lemma}
\label{lem11}
A partial Latin square $C\subset L$, of size $s$ and 
order $n$, is a critical set for a Latin square $L$ if and only if both the 
following hold:
\begin{enumerate}
\item $C$ contains an entry of every Latin interchange that occurs in $L$;
\item for each $(i,\, j;\, k)\in C$, there exists a Latin interchange
$I$ in $L$ such that $I \cap C = \{(i,\, j;\, k)\}.$  
\end{enumerate}
\end{lemma}

\section{Designs, Defining Sets, and Trades}
The following definitions will be used in Chapter~\ref{ch7} of this thesis,
where a relationship between trades in Steiner triple systems and Latin
interchanges is introduced.  The concept of the defining set is analogous
to that of the uniquely completable set and it is interesting and useful
to look at connections between the two concepts, as in Chapter~\ref{ch7}.

Let $V=\{1,\dots, v\}$  and let ${\mathcal B}$  be a  collection of 
$3$-subsets chosen from $V$ in such a way that each pair of $V$ 
occurs in at most  one of the $3$-subsets. Then $(V,{\mathcal B})$ is 
said to be a  {\em partial Steiner triple system}  and is sometimes 
referred to as a 2-$(v,3)$ partial Steiner system. The $3$-subsets 
are called {\em blocks} or {\em triples} and the {\em replication number} 
for a given element $e\in V$ is the number of triples in $\mathcal{B}$ which contain $e$.
If  $|{\mathcal B}|=v(v-1)/6$ then each of the pairs of $V$ is 
contained in precisely one triple of ${\mathcal B}$  and in this 
case $(V,{\mathcal B})$ is said to be a {\em Steiner triple system of 
order} $v$.   We denote a Steiner triple system of order $v$ as STS($v$).

Take two such partial Steiner triple systems with triples $T$ and $T'$. If $|T|=|T'|$ and each of the pairs of elements of $V$ contained in the triples of $T$ are also contained in the triples of $T'$, then $T$ and $T'$ are said to be {\em mutually balanced}. If $T$ and $T'$ are mutually balanced and have no common triples,  they form a 2-$(v,3)$ {\em Steiner trade} usually denoted by ${\mathcal T}=(T,T')$. The {\em volume(T)} of the trade is $|T|$
and the {\em foundation} of ${\mathcal T}$ is 
$F({\mathcal T})=\{x\mid x \mbox{ is contained in a triple of } \mathcal{T}\}$. 

For example, consider the trade
${\mathcal T}=(T,T')$ 
where $T=\{1 2 3, 1 5 6, 4 3 5, 4 2 6\}$ 
and $T'=\{1 2 6, 1 3 5, 4 2 3, 4 5 6\}$.
${\mathcal T}$ has volume $|T|=4$ and foundation 
$F({\mathcal T})=\{1,2,3,4,5,6\}$.

Let ${\mathcal T}=(T,T')$ be a trade. We say ${\mathcal T}$ is a
{\it minimal} trade if there is no set $B$ satisfying $\emptyset\neq B\subset 
T$ and no set $B'$ satisfying $\emptyset\neq B'\subset T'$ such that
$(T\setminus B, T'\setminus B')$ is a trade.

Let $(V,{\mathcal B})$ be a partial Steiner triple system of order
$v$.  We define the corresponding {\em partial Steiner Latin square of
order $v$} to be the $v \times v$ array $I$ with entry $k$ in cell $(i,j)$ if and only
if $\{i,j,k\}\in {\mathcal B}$. We emphasise that  for  each triple
$\{x,y,z\}\in T$, $I$ contains six entries $(x,y;z)$, $(x,z;y)$,
$(y,x;z)$, $(y,z;x)$, $(z,y;x)$, $(z,x;y)$ \cite{gw:css}.  In Chapter~\ref{ch7}
we shall often shorten a triple $(x,y,z)$ to $xyz$ where the context
makes it clear that $xyz$ is a triple.

The pair of partial Latin squares $I$ and $I'$ given below correspond to the trade
${\mathcal T}=(T,T')$ given above.  The partial Latin square $I$ corresponds to $T$ and the partial Latin square $I'$ corresponds
to $T'$.

\begin{center}
\begin{tabular}{cc}

\begin{latinsq}{6}
\hline   \ecell & 3 & 2 & \ecell & 6 & 5 \\
\hline 3 & \ecell & 1 & 6 & \ecell & 4 \\
\hline 2 & 1 & \ecell & 5 & 4 & \ecell \\
\hline   \ecell & 6 & 5 & \ecell & 3 & 2 \\
\hline 6 & \ecell & 4 & 3 & \ecell & 1 \\
\hline 5 & 4 & \ecell & 2 & 1 & \ecell \\
\hline
\end{latinsq}
&
\begin{latinsq}{6}
\hline   \ecell & 6 & 5 & \ecell & 3 & 2 \\
\hline 6 & \ecell & 4 & 3 & \ecell & 1 \\
\hline 5 & 4 & \ecell & 2 & 1 & \ecell \\
\hline   \ecell & 3 & 2 & \ecell & 6 & 5 \\
\hline 3 & \ecell & 1 & 6 & \ecell & 4 \\
\hline 2 & 1 & \ecell & 5 & 4 & \ecell \\
\hline
\end{latinsq} \\[3pt]
 $I$ & $I'$
\end{tabular}
\end{center}

\section{The Spectrum of Critical Set Sizes}
In Nelder's 1977 note on critical sets \cite{Nel1}, he defined
the concept of critical sets and then proposed one problem,
that of finding a formula for the size of the largest
and smallest critical sets in $n \times n$ Latin squares.
He suggested that solutions should be sought first for prime $n$,
then for $n$ a prime power, and then for general $n$.  The best known 
bounds for these functions are outlined below.

For an $n \times n$ Latin
square, the size of the smallest and largest possible critical sets,
respectively, are denoted \scs{n} and \lcs{n} \cite{Nel1}.

The best known bounds on \scs{n} are:
\begin{itemize}
\item \scs{n} $\geq \lfloor \displaystyle\frac{7n-3}{6} \rfloor$
for $n > 20$ (\cite{MR1468170}).
\item \scs{n} $\geq n + 1$, for $n \geq 5$ (\cite{MR96h:05030} and  \cite{MR95h:05032} independently).
\end{itemize}

The best known bounds on \lcs{n} are:
\begin{itemize}
\item \lcs{n} $\geq \displaystyle\frac{n^2-n}{2}$, conjectured in \cite{nel2} and proved in \cite{MR97g:05032}.
\item \lcs{2^m} $\geq 4^m - 3^m$,  \cite{MR84g:05036}.
\item \lcs{2^m - 1} $\geq 4^n - 3^n - 2^{n+1} + 3$, \cite{MR96h:05030}.
\item \lcs{2m} $\geq \displaystyle\frac{5m^2-3m}{2}$, \cite{MR99a:05018}.
\end{itemize}
In Chapter~\ref{ch4} of this thesis, we shall show that \lcs{n} $\leq n^2 -
3n + 3$, and in Chapter~\ref{ch5}, we shall show that \lcs{4m}
$\geq \displaystyle{\frac{23m^2-9m}{2}}$.

\section{Classifying Latin Squares}\label{sec27}
Two Latin squares $L$ and $M$ are said to be $isotopic$ if the rows, columns,
or symbols of $L$ can be permuted to transform $L$ to $M$.

Formally, let $L = \{ (i_1, j_1; k_1) \mid i_1, j_1, k_1 \in N \}$ and 
$M = \{ (i_2, j_2; k_2) \mid i_2, j_2, k_2 \in N \}$ be two Latin squares
of order $n$.  Then $L$ is said to be $isotopic$ to $M$ if there
exist permutations $\alpha, \beta$ and $\gamma$ of $N$, such that 
$M = \{ (i_1 \alpha, j_1 \beta; k_1 \gamma) \mid (i_1, j_1; k_1) \in L \}$.
In this case $M$ is said to be an {\it isotope} of $L$ and the triple
$(\alpha, \beta, \gamma)$ is said to be an {\it isotopism} (see \cite{howse:thesis}).
Two Latin squares $L$ and $M$ are said to be $conjugate$ if rows, columns or symbols
in $L$ can be interchanged, so that $L$ is transformed to $M$.

Let $L$ be an $n \times n$ Latin square.
Then there are six Latin squares conjugate to $L$, or six {\it conjugates}:

\begin{itemize}
\item[] $L$;
\item[] $L^{*} = \{ (j, i; k) \mid (i, j; k) \in L \}$;
\item[] $^{-1}L = \{ (k, j; i) \mid (i, j; k) \in L \}$;
\item[] $L^{-1} = \{ (i, k; j) \mid (i, j; k) \in L \}$;
\item[] $^{-1}(L^{-1}) = \{ (j, k; i) \mid (i, j; k) \in L \}$; and
\item[] $(^{-1}L)^{-1} = \{ (k, i; j) \mid (i, j; k) \in L \}$. \cite{dk}
\end{itemize}

Two Latin squares $L$ and $M$ are said to be in the same $isotopy~class$ if $L$ is isotopic to $M$, and 
in the same $main~class$ if $L$ is isotopic to a conjugate of $M$.

These same concepts of isotopy classes and main classes can also
be applied to partial Latin squares.

The following table, Table~\ref{mainiso}, shows the number of main and isotopy
classes for Latin squares of order $1\leq n\leq 8$ (see D{\'e}nes and
Keedwell~\cite{dk}). 

\begin{table}
\begin{center}
\caption{Number of main and isotopy classes for Latin squares of small order}
\label{mainiso}
\begin{tabular}{|c|c|c|c|c|c|c|c|c|} \hline
order $n$ & $1$ & $2$ & $3$ & $4$ & $5$ & $6$ & $7$ & $8$ \\ \hline
Main classes & $1$ & $1$ & $1$ & $2$ & $2$ & $12$ & $147$ & $283\,657$ \\ \hline
Isotopic classes & $1$ & $1$ & $1$ & $2$ & $2$ & $22$ & $564$ & $1\,676\,267$ \\ \hline
\end{tabular} \vspace {3mm} \\
\end{center}
\end{table}

It is apparent from the table that even the number of main classes
grows super-exponentially with the order of the Latin square.  Thus, enumerating
all the critical sets by main classes would be currently impossible for Latin
squares of order greater than 7.  Erroneous values for the number 
of isotopy classes of orders 7 and 8 have been given in the standard references, \cite{dk} 
and \cite{crc}.




We say that an $n \times n$ partial Latin square is {\it reduced} or in {\it
reduced form} if it contains the symbols $1, \dots, n$ in this order
in the first row and in the first column.

\section{Intercalates in Latin Squares}
The maximum number of intercalates in an $n \times n$ Latin square
is denoted $I(n)$, \cite{MR84g:05034}.  Formally, where $L$ is 
an $n \times n$ Latin square, denote the
number of intercalates in $L$ by $I(L)$.  Let
\begin{eqnarray*}
A & = & \{ \{ (r_1,c_1;e_1),(r_1,c_2;e_2),(r_2,c_1;e_2),(r_2,c_2;e_1) \} \mid \\
&& \{ (r_1,c_1;e_1),(r_1,c_2;e_2),(r_2,c_1;e_2),(r_2,c_2;e_1) \} \subseteq L \\
&& \wedge (r_1 \neq r_2) \wedge (c_1 \neq c_2) \wedge (e_1 \neq e_2) \}, {\rm~and} \\
I(L) & = & |A|.
\end{eqnarray*}
Then $I(n)$ is the maximum value of $I(L)$ where $L$ ranges over
all $n \times n$ Latin squares.  Thus, $I(n) \geq I(L)$ for
any $n \times n$ Latin square, $L$.

Since an intercalate is the smallest possible Latin interchange, it has
been useful to investigate the number of intercalates in a Latin square.
This information was used in the search for a critical set of order 8
and size 17 in Chapter~\ref{ch6} (as outlined in Chapter~\ref{ch3}), and is used
in the conclusion to Chapter~\ref{ch4}, when commenting on the conjectured link
between Latin squares with $I(n)$ intercalates and critical sets of
order $n$ and size \lcs{n}.

Some exact, upper and lower bounds of $I(n)$ are known for specific values 
of $n$.

The following results are a summary of the theorems in Heinrich and Wallis, 
\cite{MR84g:05034}.  These results will be used in Chapter~\ref{ch5} when
discussing new bounds on $I(n)$.

\begin{itemize}
\item When $n$ is even, $I(n) \leq \displaystyle\frac{n^2(n-1)}{4}$ with equality if and only if $n = 2^m$; 
\item when $n$ is odd, $I(n) \leq \displaystyle\frac{n(n-1)(n-3)}{4}$ with equality if and only if $n = 2^m - 1$; 
\item when $m$ is odd, $I(2m) \geq m^3$;
\item when $m$ is odd and $\alpha \geq 1$, $I(2^\alpha m) \geq \displaystyle\frac{(2^\alpha m)^2 (2^\alpha m + 2^\alpha - 2)}{8}$;
\item when $m$ is odd and $\alpha \geq 2$, $I(2^\alpha m + 1) \geq 2^\alpha m ( 2^\alpha m(2^\alpha m + 2^\alpha - 10)/8 + m + 1 ) + 2^{\alpha-1} m(m-1)$;
\item when $(m,6) = 1$, $I(2m+1) \geq \displaystyle\frac{m(2m-3)(m-1)}{2}$;
\item when $(m,6) = 1$, $I(2^\alpha m+1) \geq (2^\alpha m) ( (2^\alpha m) (2^\alpha m+2^\alpha - 2)-10m+6 ) /8$.
\end{itemize}

The next two results are from Kotzig and Zaks \cite{MR85h:05026}:
\begin{itemize}
\item when $k \geq 1$, $I(4k+1) \leq 2k(8k^2-4k-1)$;
\item when $k \geq 1$, $I(4k+2) \leq (2k+1)(8k^2+1)$.
\end{itemize}

In Chapter~\ref{ch5}, we shall prove that 
$I(2^\alpha m) \geq (2^\alpha m)^2
(3.2^\alpha m + 2^\alpha - 4)/16$, for $\alpha \geq 2$ and $m$ odd,
and that $I(2^\alpha m + 1) \geq 2^\alpha m (
2^\alpha m (3.2^\alpha m+2^\alpha - 20)/16 + m + 1 ) + 2^{\alpha-1}
m(m-1)$, for $\alpha = 2$ or $\alpha \geq 4$ and $m$ odd.




\chapter{Algorithms}\label{ch3}
%
%
%

Most of the results in this thesis were obtained via a combination
of theoretical analysis and computational methods.  In this chapter,
we discuss some of the algorithms which were used.  In particular,
we describe algorithms for the discovery of critical sets and 
Latin interchanges, and the completion of partial Latin squares. 

The discovery of patterns which will lead to new theorems such as those
which are presented in this thesis requires the study of Latin squares
of non-trivial order, that is, of order greater than 5.  Since critical
sets are complex structures, and as the number of Latin squares increases
super-exponentially with the order, it was necessary to develop fast
algorithms to generate critical sets with certain desired properties,
and as a consequence of this, to develop fast completion algorithms and
new algorithms for finding Latin interchanges.

The limitations of applying general principles to Latin squares of small
order are clearly seen in the slow progress which has been made in
improving the general bound on \scs{n}, the size of the
smallest critical set in a Latin square.  In 1978, Curran and van Rees
\cite{MR80j:05022} showed that \scs{n} $\geq n-1$ and by 1994, Cooper,
McDonough and Mavron had proven \scs{n} $\geq n+1$ for $n \geq 
5$~\cite{MR95h:05032}.  Perhaps the examination of small critical sets
for non-trivial orders of Latin squares will produce further breakthroughs
for this bound, just as the examination of large critical sets of non-trivial 
orders did for \lcs{n} in Chapter~\ref{ch4}.

\section{Algorithms for finding critical sets}
In the search for a critical set of a particular size $m$ for a Latin
square of known order, one obvious approach is to find all
subsets of size $m$ in the Latin square, and test each of these
for unique completion.  For each subset $U$ which passes this test,
all the proper subsets of $U$ of size $|U|-1$
are tested for unique completion.  If no such subset has unique
completion, $U$ is a critical set.  In ~\cite{MR85k:68035},
Colbourn, Colbourn and Stinson proved that, in general, the problem of deciding
whether a partial Latin square $P$ has unique completion
is NP-complete, even given a Latin square completing $P$.  Thus,
it is desirable to avoid the process of exhaustively testing for unique
completion by eliminating many subsets which are candidates through the
use of algorithms which run in polynomial time.

For example, the basis of the main theorem in Chapter~\ref{ch6} was the discovery
of a critical set of order 8 and size 17.  Previously, many researchers
have attempted to find a critical set of this size with no success.
To search exhaustively for any such critical set in each of the 283\,657
main classes of $8 \times 8$ Latin squares would have required the testing
of $\displaystyle{\binom{64}{17}}$, or more than $10^{15}$, subsets in each Latin square.
Thus, this algorithm is inefficient for finding all critical sets of a
given size, and in order to find very large critical sets, a completely
different approach is required.   Consequently, two other more efficient
algorithms (Algorithms~\ref{311} and \ref{genint}) for finding a critical set of size
$m$ in a known Latin square of order $n \times n$ have been developed.
These two algorithms are important as they form the basis for all other
searches.

The first (Algorithm~\ref{311}) involves the same exhaustive search
through all 
$\displaystyle{\binom{n^2}{m}}$ subsets of size $m$ of the Latin square.  The process
is speeded by calculating in advance some Latin interchanges in the
Latin square and determining whether the subset $U$ intersects all these Latin
interchanges.  If a Latin interchange is found which does not intersect
$U$, then $U$ cannot be a critical set.  This approach has the advantage of
avoiding the time-intensive process of attempting to determine whether
the set $U$ has unique completion.  This algorithm is used in Chapter~\ref{ch8}
to determine all the critical sets in the main classes of Latin squares
of order at most six.
\newpage
\begin{algorithm}{Finding a critical set}
\label{311}
\begin{itemize}
\item Input an $n \times n$ Latin square $L$.
\item Input a set $\cal{I}$ of Latin interchanges in $L$.
\item Generate all size $m$ subsets of the Latin square $L$.  Place these subsets into a set $\cal{U}$.
\item For each subset $U$ in $\cal{U}$,
\begin{itemize}
\item Test whether there exists a Latin interchange $I$ in $\cal{I}$ such that $I \cap U = \phi$.
\item If such a Latin interchange exists, proceed to the next subset in $\cal{U}$.  Otherwise:
\item Test whether $U$ has unique completion.  If not, proceed to the next subset.  Otherwise:
\item Test whether any subset of $U$ of size $|U|-1$ has unique completion.  If so, then proceed to the next subset.
\item Otherwise, output $U$, a critical set, and proceed to the next subset.
\end{itemize}
\end{itemize}
\end{algorithm}

A further refinement of this method involves the decomposition of the
Latin square into disjoint Latin interchanges of small size, and ensuring in the generation step that at least one
entry of each of these Latin interchanges is included in the subset $U$.
Also, where the Latin interchanges are subsquares,
the intersection of any critical set with the subsquare must be a uniquely
completable set in the subsquare; otherwise, the subsquare 
has more than one completion.
This last refinement was particularly useful in Chapter~\ref{ch8}, where some of
the $6 \times 6$ Latin squares could be partitioned into
$3 \times 3$ subsquares.  We give an algorithm which assists with
splitting a Latin square into disjoint Latin interchanges.

\begin{algorithm}{Locating disjoint Latin interchanges in a Latin square}
\label{312}
\begin{itemize}
\item Input an $n \times n$ Latin square $L$.
\item Read in the array $\cal{I}$ of $s$ Latin interchanges in $L$ of small size.
\item Create an $n \times n$ array $T$ such that $T[i][j]$ contains the
number of Latin interchanges of $\cal{I}$ in which the cell $(i,j)$ is non-empty.
\item Initialize an empty $n \times n$ partial Latin square $P$.
\item When $T[i][j] = 1$, add the relevant Latin interchange to $P$, as it
must occur in any decomposition of $L$ into disjoint Latin interchanges.
\item Call the function choose(0).
\end{itemize}

The function choose(pos):

\begin{itemize}
\item If $P = L$, then $L$ can be decomposed into disjoint Latin interchanges.
\item For each Latin interchange $I$ from $\cal{I}$[pos] to $\cal{I}$[s],
if $I$ is disjoint to $P$, then add $I$ to $P$ and call choose(pos+1).
\end{itemize}
\end{algorithm}

The use of bitmaps to check whether the Latin interchange intersects the
proposed critical set speeds the search considerably.  Instead of
using a {\it for loop}, the set and the Latin interchange are represented
as bitmaps and a logical OR used to test whether the Latin interchange
intersects the set.

In Chapter~\ref{ch8}, when searching the $6 \times 6$ Latin squares for critical
sets of size greater than 18, Algorithm~\ref{311} was further speeded by
ensuring that in each partial Latin square examined, no row or column
was full and no symbol occurred six times.  Such partial Latin squares
cannot be critical sets as any entry may be removed from the relevant row,
column or set of symbols while maintaining the unique completion property.
As the partial Latin squares being tested become larger, this algorithm
runs comparatively more and more quickly; for example, for partial Latin
squares of size 21 it is several times faster than any other method.

\section{Unique completion and Latin interchanges}
Two key steps in the Algorithm~\ref{311} are discussed separately.
The first is the step which checks for unique completion, and the
second, which must be completed prior to running the algorithm,
is determining the set of Latin interchanges.

First, we give an algorithm which recursively fills all the blank cells
in a partial Latin square.  It was used as part of Algorithm~\ref{311}
to test whether a set is uniquely completable, and is thus used
in Chapters~\ref{ch4}, \ref{ch5}, \ref{ch6}, \ref{ch8}, and to determine the results of Appendices~\ref{app1} and \ref{app2}.

\begin{algorithm}{Checking for unique completion}
\label{323}
\begin{itemize}
\item Input the partial Latin square $P$.
\item Copy $P$ to $M$.
\newline 
\noindent Label 1
\item Determine an empty position in $M$, $(r,c)$.
\item Determine the set $\cal{E}$ of all the symbols that it is possible to place in $(r,c)$.
\newline
For each symbol $e \in \cal{E}$ in turn:
\begin{itemize}
\item Place the symbol $e$ in $M$ in position $(r,c)$.
\item If $M$ is now complete, output $M$; otherwise recursively jump to Label 1.
\end{itemize}
\end{itemize}
\end{algorithm}

The empty position $(r,c)$ in $M$ may be determined by two different means.  The
first is to simply proceed through $M$ column by column and row
by row, beginning at position $(0,0)$.  The second is to search for
the position in the Latin square in which the least number of alternatives
is possible.  That is, for each position in the Latin square, we
count the number of symbols which, if added to the partial Latin square
in that position, would result in an array which would not be a partial
Latin square.  The position in which this number is greatest is
chosen.  Through extensive testing, it was found that each alternative  is suitable for
different goals.  When all completions need to be generated, the first
approach is better.  However, given a strong critical set which is
being tested for unique completion, the second approach will determine
more quickly if only one completion is possible.  The speed of
the algorithms is also affected by the density of the partial Latin square,
that is, the ratio of the number of entries to the number of cells.

The key part of Algorithm~\ref{311} is finding the Latin interchanges,
and so the next group of algorithms is for finding Latin interchanges
of various sizes greater than or equal to 4.

If the Latin square contains a large number of intercalates (Latin
interchanges of size 4) we can considerably reduce the search space required
to find critical sets.
The reason for this is that each critical set in the Latin square must
contain at least one of the four entries from the intercalate.  Further
improvements are possible if the Latin square can be decomposed into
disjoint intercalates.  For example, suppose we are searching for a
critical set of size 9 in a $6 \times 6$ Latin square where the square
can be decomposed into 9 disjoint intercalates.  Then 
there are $4^9 = 262\,144$ cases to examine, compared to $\displaystyle{\binom{6^2}{9}} =
94\,143\,280$ for the exhaustive search through all the subsets of size 9.

Since the existence of intercalates can reduce the search size
dramatically, we give two algorithms for finding intercalates. 
There is an obvious $O(n^4)$
algorithm and a less obvious $O(n^3)$ algorithm for finding these.
These algorithms are given below.

\begin{algorithm}{$O(n^4)$ algorithm for finding intercalates}
\label{324}
\begin{itemize}
\item Input an $n \times n$ Latin square $L$.  

\item Generate the $\displaystyle{\binom{n}{2}}$ pairs of row numbers $r_{0}$ and $r_{1}$, with
$1 \leq r_{0} < r_{1} \leq n$.

\item Generate the $\displaystyle{\binom{n}{2}}$ pairs of column numbers $c_{0}$ and $c_{1}$, with
$1 \leq c_{0} < c_{1} \leq n$.

\item If $L_{r_{0}c_{0}} = L_{r_{1}c_{1}}$ and 
$L_{r_{1}c_{0}} = L_{r_{0}c_{1}}$ for any
of the generated pairs of $r_{0}, r_{1}, c_{0},$ and $c_{1}$, then an intercalate exists
in these four positions.
\end{itemize}
\end{algorithm}

\begin{algorithm}{$O(n^3)$ algorithm for finding intercalates}
\label{325}
\begin{itemize}
\item Input an $n \times n$ Latin square $L$.
\item For each symbol $e$, $1 \leq e \leq n$, and for each
column $c$, $1 \leq c \leq n$, determine in which row symbol $e$
occurs in column $c$.  Place this row number ($f$) in a two-dimensional $n \times n$ array
$d$ such that $d[c][e] = f$, where $L_{fc} = e$.

\item Consider each entry $(r,c;L_{rc})$ in the Latin square $L$
with $1 \leq r \leq n$, $1 \leq c \leq n$.  For each such entry,
consider all columns $b$ such that $c + 1 \leq b \leq n$.

\item Determine where the symbol $L_{rb}$ occurs in column $c$; 
that is, find $d[c][L_{rb}]$.  Call this row number $g$.

\item If $L_{rc} = L_{gb}$ then an intercalate exists in positions
corresponding to the intersection of rows
$r$ and $g$ with columns $c$ and $b$.
\end{itemize}
\end{algorithm}


We give an example of the application of this algorithm.  Take the
following Latin square $L = {\mathbb Z}_2^2$.

\begin{center}
\begin{tabular}{c}
\begin{latinsq}{4}
\hline 1&2&3&4 \\
\hline 2&1&4&3 \\
\hline 3&4&1&2 \\
\hline 4&3&2&1 \\
\hline 
\end{latinsq}\\[3pt]
 $L$
\end{tabular}
\end{center}

We calculate the array $d$.  We consider column $c$ = 1 first.  We
proceed through the symbols $e = 1$ up to $e = 4$.  Since symbol 1
occurs in column 1 at position (1,1), that is, in row 1, we assign the value 1 to $d[1][1]$.
Since symbol 2 occurs in column 1 at position (2,1), that is, in row 2, we assign the
value 2 to $d[1][2]$.  We proceed through all $n^2$ different entries
in $L$.

Next, we consider each entry in $L$.  Start at row $r = 1$ and column $c = 1$.
Then, consider column $b = 2$.

We determine where the symbol $L_{12} = 2$ occurs in column 1.  This
row number is contained in the array $d$ at $d[1][2]$.  The answer is $g = 2$.

Then, the last step says that if $L_{rc} = L_{gb}$, an intercalate
exists at the intersections of rows $r$ and $g$ with columns $b$ and $c$.
Since $L_{11} = L_{22} = 2$, we have an intercalate at positions $(1,1),
(1,2), (2,1)$ and $(2,2)$.  Also, the algorithm runs in $O(n^3)$ time, as filling the
array $d$ takes $n^2$ steps and the determination of the intercalates
requires $O(n)$ steps for each of the $n^2$ entries in the Latin square.

The complexity of the search for Latin interchanges which are
not intercalates is much higher.  Thus, when using Algorithm~\ref{311},
the best idea is to restrict the search space
as much as possible initially, by using smaller interchanges such
as intercalates.  For instance, suppose a Latin square
of order 6 can be decomposed into four $3 \times 3$ subsquares, and
we are searching for a critical set of size 9.  Then each $3 \times 3$
subsquare must contain a uniquely completable set for that subsquare.
So each subsquare must contain at least two entries.  In fact three of
the subsquares must contain two entries and the fourth must contain three.
There are nine UC sets of size two for any $3 \times 3$ subsquare and 75
UC sets of size three.  Thus, there are $9^3 \times 75 \times 4 = 218\,700$
possibilities which must be examined, compared to ${4^9 = 262\,144}$ when
using 9 disjoint intercalates.  For sizes greater than 10, however,
this method is slower than using the intercalates.  For various sizes
of subsets in $6 \times 6$ Latin squares, Table~\ref{split} lists
the number of cases which need to be considered with an exhaustive search,
and when the Latin square can be decomposed into nine disjoint $2 \times 2$ Latin 
subsquares or four disjoint $3 \times 3$ Latin subsquares.

\begin{table}
\caption{Number of subsets to be examined using various search methods in $6 \times 6$ Latin squares}
\label{split}
\begin{center}
\begin{tabular}{|c|c|c|c|}
\hline
Size of subset & Exhaustive & $2 \times 2$ & $3 \times 3$ \\
\hline
9  &    94\,143\,280 &       262\,144 &       218\,700 \\
10 &   254\,186\,856 &     3\,538\,944 &     3\,101\,166 \\
11 &   600\,805\,296 &    23\,592\,960 &    24\,740\,316 \\
12 & 1\,251\,677\,700 &   103\,219\,200 &   125\,331\,705 \\
13 & 2\,310\,789\,600 &   332\,365\,824 &   439\,425\,648 \\
14 & 3\,796\,297\,200 &   837\,697\,536 & 1\,149\,328\,764 \\
15 & 5\,567\,902\,560 & 1\,716\,436\,992 & 2\,366\,815\,464 \\
16 & 7\,307\,872\,110 & 2\,932\,162\,560 & 3\,982\,863\,312 \\
17 & 8\,597\,496\,600 & 4\,250\,133\,504 & 5\,620\,113\,720 \\
18 & 9\,075\,135\,300 & 5\,293\,364\,736 & 6\,771\,725\,820 \\
19 & 8\,597\,496\,600 & 5\,716\,214\,784 & 7\,057\,334\,304 \\
20 & 7\,307\,872\,110 & 5\,386\,735\,872 & 6\,419\,253\,726 \\
21 & 5\,567\,902\,560 & 4\,449\,137\,664 & 5\,127\,197\,616 \\
\hline
\end{tabular} \vspace{3mm} \\
\end{center}
\end{table}
\vspace{3mm}

In the search for a critical set of order 8 and size 17 (see
Chapter~\ref{ch6}), all $8 \times 8$ main classes
of Latin squares with $4 \times 4$ subsquares were generated, and then
all possible $4 \times 4$ UC sets were placed in the subsquares.  In a
similar effort, all $8 \times 8$ main classes of Latin squares with 16 disjoint
intercalates were found, and potential critical sets were generated
by taking one entry from each intercalate and then one entry 
from somewhere else in the complete Latin square.  These efforts showed
that no critical set of size 17 could exist in any of these main classes
of Latin squares.  Table~\ref{ebad} shows the possible numbers of 
intercalates in $8 \times 8$ Latin squares, and the number of main
classes which contain that number of intercalates.  In each pair
of columns, the first number (\#MC) is the number of main classes, with
the number of intercalates given in the second column (\#I).  

\begin{table}
\begin{center}
\caption{Number of intercalates in all main classes of $8 \times 8$ Latin squares}
\vspace{3mm}
\label{ebad}
\begin{tabular}{|c|c| |c|c| |c|c| |c|c| |c|c| |c|c|}
\hline
\#MC & \#I &  \#MC&\#I    & \#MC&\#I   &  \#MC&\#I  &   \#MC&\#I &   \#MC&\#I \\
\hline
3&0  &	    23206&11  &	6273&21	 &  211&31  &	1&41  &	    2&51   \\
14&1  &	    26212&12  &	5002&22	 &  255&32  &	24&42  &    9&52   \\
66&2  &	    26840&13  &	3094&23	 &  79&33  &	12&43  &    14&56   \\
265&3  &    26797&14  &	2609&24	 &  123&34  &	27&44  &    1&60   \\
758&4  &    24225&15  &	1532&25	 &  67&35  &	5&45  &	    12&64   \\
1830&5  &   21535&16  &	1265&26	 &  113&36  &	5&46  &	    1&68   \\
3893&6  &   18020&17  &	699&27  &   25&37  &	3&47  &	    1&72   \\
6587&7  &   14747&18  &	748&28  &   58&38  &	34&48  &    2&80   \\
10583&8	 &  11241&19  &	340&29  &   21&39  &	1&49  &	    1&88   \\
15073&9	 &  8905&20  &	350&30  &   75&40  &	2&50  &	    1&112  \\
19760&10  &     &     &    &    &     &    &     &    &      &      \\
\hline
\end{tabular} \vspace{3mm} \\
\end{center}
\end{table}

However, at times it is necessary to search for Latin interchanges
of size greater than four.  

An algorithm for finding Latin interchanges had been given previously by
Howse \cite{howse:thesis} which determined Latin interchanges of size up
to 11 in a given Latin square.  I independently developed a new algorithm
(Algorithm~\ref{genint})
which worked for Latin interchanges of any size and was used in Chapter~\ref{ch7} 
in the process of decomposing partial Latin squares.  In addition,
for the results of Chapters~\ref{ch6} and \ref{ch8}, it was necessary to determine how
and when the Latin interchanges should be used for maximum efficiency.
\newpage

\begin{algorithm}{Searching for Latin interchanges of general size $m$, $m > 4$}
\label{genint}
\begin{itemize}
\item Input the $n \times n$ Latin square $L$.

\item Generate all $\lfloor \displaystyle{\frac{m}{2}} \rfloor \times 
\lfloor \displaystyle{\frac{m}{2}} \rfloor$
subarrays of $L$.

\item For each subarray $S$, generate all subsets $U$ of size 
$m$.

\item Calculate all permutations of size $x$ of the symbols 
$\{1,2,\dots,\lfloor \displaystyle{\frac{m}{2}} \rfloor \}$ with no fixed points, where 
$2\leq x \leq \lfloor \displaystyle{\frac{m}{2}} \rfloor$.
\item Determine the size of each row and column in the subset $U$, and
the number of times each symbol occurs in the subset $U$. If each
of these numbers is greater than or equal to $2$, continue; else
move to the next subset.

\item Apply each of the permutations calculated above to each of 
the rows in each subset $U$.
If the columns are mutually balanced then a Latin interchange has been found.
\end{itemize}
\end{algorithm}

For Latin squares of order 10 or more, an optimization of this algorithm is possible.
For these sizes of Latin squares, it is sometimes faster to apply the
pre-calculated permutations to the columns rather than the rows.  For
example, if a Latin interchange consists of five columns of two entries each,
occurring in two rows of five entries, it is much faster to examine all
permutations of the columns than to examine all permutations of the rows.
The faster method requires calculating the number of permutations required
using both methods and determining whether there are more permutations
which can be applied to the rows or the columns.  This is followed by
the test for mutually balanced columns or rows, accordingly.

\section{Critical sets with a given property}

Sometimes we are specifically interested in determining critical sets
with a given property.

The next algorithm (Algorithm~\ref{337}) involves beginning with a complete Latin square 
and `intelligently' removing entries
while maintaining the critical set property.  This method is used
for finding examples of critical sets of a given size and is thus not generally
suitable for determining {\em all} critical sets of a given size.  It was used
to extend the spectrum of known critical set sizes in Latin squares of orders
9 and 10 which are shown in Appendix~\ref{app1}.

This algorithm has been especially successful in
demonstrating the existence of many critical sets of order 7 and size 25.
The results of Chapter~\ref{ch4} and especially the conjectures and questions
in its conclusion are the result of many computer searches for large
critical sets of order greater than 5.

\begin{algorithm}{Finding critical sets with a given property}
\label{337}
\begin{itemize}
\item Input an $n \times n$ Latin square $L$.
\item Input a property $\cal{R}$.
\item Copy $L$ into $P$.
\item Determine if there exists an entry $(x,y;z)$ in $P$ such that it both meets property $R$ and has the property that $L \setminus \{(x,y;z)\}$ has unique completion.
\item If there are no such entries, output $P$ and stop, otherwise remove $(x,y;z)$ from $P$ and repeat the last step.
\end{itemize}
\end{algorithm}

The entry removed can
be the first one reached (beginning at the top left of the Latin
square and moving down and to the right) the removal of which does not destroy
the critical set property.  Alternatively, it can be the entry where
the value of $x_{i,j}(P)$ is lowest or highest.  
(Recall that $x_{i,j}(P)$ is the number of different symbols in row $i$
and column $j$ in a partial Latin square $P$.)  Of course, where the value of $x_{i,j}(P)$ is $n$,
the entry at position $(i,j)$ can always be removed and the
result will still be a uniquely completable set.  This led to the
idea of attempting to remove the entry at position $(i,j)$ where
the value of $x_{i,j}(P)$ is the highest.  This proved to be effective
in practice, as did, paradoxically, removing the entry at $(i,j)$
where the value of $x_{i,j}(P)$ was lowest.  This idea led to the
generation of the critical sets given in Appendix~\ref{app1}, which improved
on the examples of Curran and van Rees \cite{MR80j:05022}, being the largest known
critical sets for the given orders.


As will be seen in Chapter~\ref{ch4}, removing a row, column and symbol from the
complete square before beginning the search is a good idea when searching
for large critical sets.  We comment on the results of this idea further
for the $7 \times 7$ Latin squares.

For each of the 147 main classes of $7 \times 7$ Latin squares, we 
considered all $7^3$ triples of (row, column, symbol) and removed
all entries in the row and column and all occurrences of the symbol.
This left partial Latin squares of size $7^2-3 \times 7+2$ or
$7^2-3 \times 7+3$, that is, 30 or 31.  All subsets of size 25 in
these partial Latin squares were tested to see if any were critical
sets.

In the examination of the Latin square corresponding to the Steiner
triple system of order 7, 11\,592 critical sets of size 25 were found.
Also, critical sets of size 25 were found in 113 out of a total of 147
main classes of $7 \times 7$ Latin squares.  Further experimentation
by removing all $3 \times 7^2$ pairs of (row,column), (row,symbol),
and (column,symbol) led to the discovery of a total of 29\,484 critical
sets in the Latin square corresponding to STS(7).

Given the fact that only one critical set of size 25 was known \cite{kho} prior
to the discovery of this technique, it is obvious that the
discoveries of Chapter~\ref{ch4} have proved very useful in locating large
critical sets in Latin squares of order greater than 6.  For example,
the size of the largest known critical set in a Latin square of order 9
was increased from 39 to 44, and in a Latin square of order 10, from 55
to 57.

Using all six conjugates of the Latin square being examined provides
a larger search space.  Allowing for slight random variations (that is,
using a different property at random steps in the search) on the highest
and lowest $x_{i,j}(P)$, or picking a position at random where $x_{i,j}(P)$
is highest or lowest, extends the space still more.  The reason for
this is that the output will consist of a greater variety of critical sets
when some random variations are allowed, more so than when following a
fixed set of steps.

We may be searching for a critical set of largest possible size,
that is $lcs(n)$.  In this case, the best squares to
begin with seem to be $n \times n$ squares with $I(n)$ intercalates,
as proposed in Chapter~\ref{ch4}.  On the other hand, if critical sets of small
size are required, back circulant Latin squares or Latin squares with
fewer intercalates seem to be a better starting point.  The reason underlying
this is that any critical set in an intercalate-rich Latin square
must intersect large numbers of intercalates, and should therefore be
larger.  Also, as we have seen, it is easier to count 
intercalates than Latin interchanges of any other size, and thus
it is easier to locate Latin squares with many intercalates, as opposed
to Latin squares with many interchanges of size greater than 4.  Conversely,
intercalate-poor Latin squares are good places to look for small critical
sets.  Also, it has been found that the smallest critical sets occur
only in the back-circulant Latin squares for orders from 1 to 7.~\cite{AK2}

With this in mind we give Algorithm~\ref{many} for finding Latin
squares which have many intercalates.  First, we need to define
an algorithm which calculates all transversals in a given Latin square.

\section{Discovering transversals in a Latin square}
This algorithm will be used as part of the next algorithm to
prolong given Latin squares.

\begin{algorithm}{Finding all transversals in a Latin square $L$}
\begin{itemize}
\item Input an $n \times n$ Latin square $L$.
\item Initialise the size $n$ arrays $cols$ and $syms$.
\item Call findtransversal($L$,$0$,$cols$,$syms$).
\end{itemize}

The function findtransversal($L$,$r$,$cols$,$syms$):
\begin{itemize}
\item If $r = n$, we have a transversal in $L$, in the entries 
$\{ (i,cols[i];syms[i]) \mid 0 \leq i \leq n-1 \}$.  Output it and continue.
\item Otherwise, for each column $j$, $0 \leq j \leq n-1$,  set $flag = 0$;
\item For each row $i$, $0 \leq i \leq n-1$,
\item[-] If $syms[i] = L_{rj}$ or $cols[i] = j$, set $flag = 1$.
\item If $flag=0$, set $cols[r]=j$, $syms[r]=L_{rj}$, and call findtransversal($L$,$r+1$,$cols$,$syms$).
\end{itemize}
\end{algorithm}

\section{Algorithm for finding Latin squares with many intercalates}
\begin{algorithm}{Finding Latin squares with many intercalates}
\label{many}
\begin{itemize}
\item Generate as many different main classes of order $n$ Latin squares as possible.
\item Determine all the transversals in each of these Latin squares.
\item Prolong each of the Latin squares along all possible transversals as in ~\cite{MR2000g:05035},
to generate $(n+1) \times (n+1)$ Latin squares.
\end{itemize}
\end{algorithm}

A similar idea was used independently in Danziger and Mendelsohn
\cite{MR2000g:05035} and Heinrich and Wallis \cite{MR84g:05034}.



This leads to discovering the $n \times n$ Latin squares with the 
currently known maximum number of intercalates for $n = 9$ and $11$.
This $11 \times 11$ Latin square 
is presented with a corresponding large critical set in Chapter~\ref{ch5}.

There are several other ways to reduce the search space and
still determine intercalate-rich Latin squares.  All of
the following methods reduce the search space, which leads
to discovering intercalate-rich Latin squares more quickly.
Some of these ideas can also be combined.  

\begin{itemize}
\item Start with a reduced partial Latin square 
and find all completions.

\item Enforce symmetry in the completions.  That is, when adding an
entry $(x,y;z)$, add the entry $(y,x;z)$ also.

\item Enforce total symmetry in the completions, as defined in 
~\cite{MR80a:05032}.  That is, when adding an entry $(x,y;z)$,
add the entries $(y,x;z), (y,z;x), (x,z;y), (z,x;y),$ and $(z,y;x)$ also.

\item Begin with a partial Latin square containing just ones along the
main diagonal.  This was originally suggested in ~\cite{MR84g:05034}.
\end{itemize}

Finally, placing Latin subsquares in a partial Latin square and then 
using some combination of the above has also proved a useful approach.
For the subsquares, either group tables or the subsquare with the most 
intercalates may be used.  Some of these approaches led to the discovery
of a $10 \times 10$ Latin square with 117 intercalates, which is the
Latin square from which the largest known critical set of order 10
is drawn.  This critical set is shown in Appendix~\ref{app1}. Also, the results of
some of these ideas were used to construct the Latin squares from
which the critical sets of order 9 given in Appendix~\ref{app1} eere derived.

\section{Finding critical sets similar to a given critical set}
At times we may try to generate a critical set with a given property
by starting with a similar critical set and adapting it.  For instance,
the critical set of order 8 and size 17 in Chapter~\ref{ch6} was discovered by
starting with a critical set of order 8 and size 16, and looking at all
possible ways of removing two entries and then adding three.

%

The following algorithm takes a critical set $C$ and attempts to create
new critical sets which vary in size from $C$ by a small number of entries.
The idea is to input two numbers $x$ and $y$, and look at all
possible ways to remove $x$ entries and add $y$ other entries.  Each
resulting partial Latin square $P$ is tested to see whether it is a critical set.

\begin{algorithm}{Finding critical sets close to a given critical set}
\begin{itemize}
\item Input a critical set $C$ of size $m$ for an $n \times n$ Latin square $L$.
\item Input $x$, the number of entries to be removed, and $y$, the
number of entries to be added.
\item Generate all $\displaystyle{\binom{m}{x}}$ subsets of $C$ which are of size $x$ and
place them in an array $C_x$.
\item For each $x$-sized subset $X$ in $C_x$, remove $X$ from $C$, creating $E_X$.
\item Generate all $\displaystyle{\binom{m}{y}}$ subsets of $C$ which are of size $y$ and place them in an array $X_y$.
\item For each $y$-sized subset $Y$ in $X_y$ such that $X \cap Y = \emptyset$, add $Y$ to $E_X$.
\item Determine whether $E_X$ is a critical set.
\end{itemize}
\end{algorithm}

\section{A suggestion of Brendan McKay}
In a search for the maximum number of intercalates in a $9 \times 9$ Latin
square, McKay~\cite{mck} proposed beginning with three intercalates
in the first two rows and extending the square one row at a time,
maximizing the number of intercalates at each stage.  I wrote a program
to determine whether this idea was effective.  This only led to a Latin square
of order 9 containing 49 intercalates.  However, Owens and Preece in
~\cite{MR96g:05029} had already discovered $I(9) \geq 72$.

\section{Parallel Algorithms}
Many of the algorithms presented in this chapter can be parallelised;
that is, a single problem may be split up and the sub-problems run
on different computers.  We give two examples of this.

Splitting up a problem proved very useful in the case of finding large
critical sets in the main class of $6 \times 6$ Latin squares with
no intercalates.  This Latin square can be partitioned into $3 \times 3$
subsquares.  Thus, in the search for a critical set of size 17, 
we split the search up so that one computer was attempting to put 5, 4, 3
and 5 entries in each subsquare respectively while another computer
was attempting to put 4, 7, 4 and 2 entries into each subsquare.
In total there were 204 cases, which were split across five computers
running the Linux operating system.  This enabled us to calculate
some of the results in Chapter~\ref{ch8}  more quickly than any other method.
Using one computer would have been very slow, and the basic Algorithm~\ref{311}
would not have worked well for a Latin square with no intercalates.

In the case of the search for a critical set of order 8 and size 17
in Latin squares with precisely sixteen intercalates (related to Chapter~\ref{ch6})
a count was maintained of the number of subsets examined.  This simple
approach enabled each search to be split across eight nodes of an SGI
Power Challenge computer, which reduced the execution time considerably.



\section{Finding a small set of Latin interchanges satisfying a property}

There are 150 Latin interchanges of size 8 in the back-circulant Latin
square of order 5.  As it has never been proven that the minimal critical
set in a back-circulant Latin square of order $n$ has size 
$\lfloor \displaystyle{\frac{n^2}{4}} \rfloor$, we decided to take a close look
at the subsets of size $\lfloor \displaystyle{\frac{5^2}{4}} \rfloor - 1 = 5$ in $BC_5$.

In $BC_3$, the minimal size of a critical set is 2.  The size of the
smallest Latin interchange in $BC_3$ is 6, and there are nine such Latin interchanges.
We need only three of these interchanges to prove that every subset
of size $\lfloor \displaystyle{\frac{3^2}{4}} \rfloor - 1 = 1$ in $BC_3$ is not a critical
set.  That is, if $X_3 = \{ I_1, I_2, I_3 \}$ where $I_1 = \{ (0,0;0),
(0,1;1), (0,2;2), (1,0;1),\break (1,1;2), (1,2;0) \}$, $I_2 = \{ (0,0;0),
(0,1;1), (0,2;2), (2,0;2), (2,1;0), (2,2;1) \}$, and $I_3 \break = \{ (1,0;1), 
(1,1;2), (1,2;0), (2,0;2), (2,1;0), (2,2;1) \}$, then for every subset
$U \in BC_3$ such that $|U| = 1$, there exists $V \in X_3$ such that 
$V \cap U = \emptyset$.

This algorithm attempts to find a set $X_5$ (called $seq$ in the algorithm)
containing Latin interchanges of size 8 from $BC_5$ such that for every
subset $U \in BC_5$ where $|U| = 5$, there exists $V \in X_5$ such that
$V \cap U = \emptyset$.  The smallest such set of Latin interchanges found has been 
of size 41.

\begin{algorithm}{Find a subset of 150 interchanges satisfying the above property}
\begin{itemize}
\item Input $\cal{I}$, the 150 Latin interchanges of size 8 in $BC_5$.
\item Turn each Latin interchange into a bitmap (25 bits).
\item Initialise the array of Latin interchanges $seq$ of size 150 and set $len = 0$.
\item Call callseq($seq$,$len$).
\end{itemize}

The function callseq($seq$,$len$):
\begin{itemize}
\item For each $I \in \cal{I}$ such that $I \not\in$ seq:
\item[-] Let $count$ = the number of subsets of size 5
in $BC_5$, represented as bitmaps, that do not intersect any of the
interchanges in $seq \cup I$.
\item If $count = 0$ for any $I$, print seq and continue.
\item Otherwise, for the $I \in \cal{I}$ where count is a minimum, let $seq = seq \cup I$, $len = len + 1$, and callseq($seq$,$len$).
\end{itemize}
\end{algorithm}

Variations on this algorithm, for instance by replacing the test
for where count is a minimum with a test which ensures that the
Latin interchanges are evenly distributed over the Latin square, have
been attempted without much success.

\section{Near-strong critical sets}
This algorithm takes a critical set $C$ and tests if it is
near-strong.  It is an extremely complex
algorithm and the order of the heavily nested for loops is critical to
the correct operation of the algorithm. 

The basic idea is to simulate the union of sets of symbols by setting bits in
a binary string to represent the presence of symbols in their union, and then
counting them or testing for the presence of a particular symbol at the
end of a loop.  

For each empty position $(i,j)$ under consideration in
the array of alternatives for the partial Latin square $P$, $A_P$, 
we need to determine whether a symbol
$k \in A_P(i,j)$ is forced out.  We do this by determining if there exists a $g$, $1
\leq g \leq n$ such that 
\begin{itemize}
\item [(1)] there exist distinct $i_1, \dots, i_g$ (all $\neq i$) with $k' \in
(i_1,j)_{A_P} \cup \dots \cup (i_g,j)_{A_P}$ and $|(i_1,j)_{A_P} \cup
\dots \cup (i_g,j)_{A_P}| = g$, or there exist distinct $j_1, \dots,
j_g$ (all $\neq j$) with $k' \in (i,j_1)_{A_P} \cup \dots \cup
(i,j_g)_{A_P}$ and $|(i,j_1)_{A_P} \cup \dots \cup (i,j_g)_{A_P}| = g$; or
\item [(2)] $\theta(i,j,k')$ satisfies $1$ in $A_{P_{\theta(1,2,3)}}$ for
$\theta = (2\;3)$ or $\theta = (1\;3)$.
\end{itemize}
This is equivalent to the definition of a symbol ``forced out'' 
of an array of alternatives given in Chapter~\ref{ch2}.

Obviously the definition of a near-strong critical set relies heavily
on the use of unions of sets and so the use of binary strings will be
important in the following algorithm.

Also, the algorithm for picking $g$ objects from $n$ objects is taken
from a program on the World Wide Web by Rhoads \cite{prog}, which is
based on code from Reingold, Nievergelt, and Deo \cite{MR57:11164}.

\begin{algorithm}{Testing whether a critical set $C$ is near-strong}
\begin{itemize}
\item Input the critical set $C$ based on the symbol set $N=\{0,\dots,n-1\}$
\item Copy $C$ to $P$.
\item Repeat the following until no more entries can be added to $P$:
\item[-] Call the function force($P$) to generate the $n \times n$ array of binary strings $A$ corresponding to the reduced array of alternatives for $P$, $RA_P$.
\item[-] If any binary string $A[i][j]$ has exactly one bit $x$ set, add the entry $(i,j;x)$ to $P$.
\item[-] If in row $i$ or column $j$ of $A$, the binary string $A[i][j]$ has bit $x$ set and no other binary string in row $i$ or column $j$ of $A$, respectively, has bit $x$ set, add the entry $(i,j;x)$ to $P$.
\item[-] Call the function force($P_{(1,3,2)}$) to generate the $n \times n$ array of binary strings $A$ corresponding to the reduced array of alternatives for $P_{(1,3,2)}$, $RA_{P_{(1,3,2)}}$.
\item[-] If any binary string $A[i][j]$ has exactly one bit $x$ set, add the entry $(i,x;j)$ to $P$.
\item[-] If in row $i$ or column $j$ of $A$, the binary string $A[i][j]$ has bit $x$ set and no other binary string in row $i$ or column $j$ of $A$, respectively, has bit $x$ set, add the entry $(i,x;j)$ to $P$.
\item[-] Call the function force($P_{(3,2,1)}$) to generate the $n \times n$ array of binary strings $A$ corresponding to the reduced array of alternatives for $P_{(3,2,1)}$, $RA_{P_{(3,2,1)}}$.
\item[-] If any binary string $A[i][j]$ has exactly one bit $x$ set, add the entry $(x,j;i)$ to $P$.
\item[-] If in row $i$ or column $j$ of $A$, the binary string $A[i][j]$ has bit $x$ set and no other binary string in row $i$ or column $j$ of $A$, respectively, has bit $x$ set, add the entry $(x,j;i)$ to $P$.
\item Finally, if $P$ is a complete Latin square, then $C$ is near-strong.
\end{itemize}

The function force($P$):
\begin{itemize}
\item For every empty position $(i,j)$ in $P$:
\item[-] If $|R_i \cup C_j| = n-1$, add the entry $(i,j;N \setminus (R_i 
\cup C_j))$ to $P$, and continue the completion;
\item[-] Otherwise, generate the array of alternatives for $P$, represented by an $n \times n$ array of binary strings, $A$, with bit $x$ of the binary string $A[i][j]$ set if and only if $x \in R_i \cup C_j$.
\item Repeat the following until the array of alternatives $A$ is unchanged; that is, when $A$ corresponds to the reduced array of alternatives for $P$.
\item[-] For every empty position $(i,j)$ in $P$:
\item[-] For every symbol $k'$ which is a possibility at $(i,j)$: 
\item[-] For $g$ from 1 to $n$:
\item[-] Pick $g$ numbers $c[1]$, \dots, $c[g]$ from the numbers $0$ to $n-1$.
\item[-] Where $c[x] \neq j$ and $P_{c[x]j}$ is non-empty for all $1 \leq x \leq g$, calculate $u$, the binary string with bit $y$ set if and only if the symbol $y$ is contained in $P_{c[x]j}$ for some $1 \leq x \leq g$.
\item[-] If bit $k'$ of $u$ is set and the number of bits in $u$ equals $g$, set bit $k'$ of $A[i][j]$ to 0.
\item[-] Where $c[x] \neq i$ and $P_{ic[x]}$ is non-empty for all $1 \leq x \leq g$, calculate $u$, the binary string with bit $y$ set if and only if the symbol $y$ is contained in $P_{ic[x]}$ for some $1 \leq x \leq g$.
\item[-] If bit $k'$ of $u$ is set and the number of bits in $u$ equals $g$, set bit $k'$ of $A[i][j]$ to 0.
\item Return the array $A$, corresponding to the reduced array of alternatives for $P$.
\end{itemize}
\end{algorithm}
\chapter{Largest critical sets in a Latin square}\label{ch4}
%
%
\def\n#1{\vbox to 3mm{\vspace{1mm}\vfill \hbox to 2.0mm{\hfill
             $#1$\hfill} \vfill }}
\def\arraystretch{1.0}                 
\pagestyle{plain}
%
%
%
%
In this chapter, we use the concept of $x_{i,j}(P)$, the number of different
symbols in the intersection of row $i$ and column $j$ in a partial Latin square $P$,
to prove a new upper bound on \lcs{n}; that is, \lcs{n} $\leq n^2-3n+3$.
Recall that \lcs{n} is the size of the largest critical set in
an $n \times n$ Latin square.
\section{The value of \lcs{n}  for small $n$}      
In Table~\ref{lcstable}, the known values of \lcs{n} are listed for small
values of $n$.  The extra columns are to compare different bounds
discussed subsequently in Section~\ref{sec43}.

\begin{table}
\caption{The sizes of the largest known critical sets of small order, with conjectured bounds}
\label{lcstable}
\begin{center}
$\begin{array}{c|cccc}
n & \lcs{$n$}  & n^{2}-3n+3 & \lfloor n^{2}-n^{3/2} \rfloor &
\lfloor (1 - (\frac{3}{4})^{log_{2}n}) n^2 \rfloor \\
\hline
 1  &       0  &  1  &  0  &  0  \\
 2  &       1  &  1  &  1  &  1  \\
 3  &       3  &  3  &  3  &  3  \\
 4  &       7  &  7  &  8  &  7  \\
 5  &      11  & 13  & 13  & 12  \\
 6  &      18  & 21  & 21  & 18  \\
 7  & \geq 25  & 31  & 30  & 27  \\
 8  & \geq 37  & 43  & 41  & 37  \\
 9  & \geq 44  & 57  & 54  & 48  \\
10  & \geq 57  & 73  & 68  & 61  \\
\end{array}$ \vspace{3mm} \\
\end{center}
\end{table}

All bounds on \lcs{n} given in column 2,
expect for $n = 5, 7, 9,$ and $10$,
are taken from \cite{MR99k:05038}.  The current bounds for $n = 5$ and 7
were given by A.~Khodkar \cite{kho}.
In Appendix~\ref{app1}, we give some examples for the largest
known critical sets for $n = 5, 7, 9$, and $10$.
The bound for $n = 6$ is given in Chapter~\ref{ch8}, which is based on \cite{abk}.

\section{Non-critical sets}    
The following lemma is our main tool for improving the upper bound
on \lcs{n}.
\begin{lemma}
\label{3.1}
Let $C$ be a critical set for a Latin square $L$ and assume that
there exists $i$
such that $|R_{i}(C)| = n - 1$.
Then the missing symbol in row $i$ does not occur anywhere in $C$,
and the column corresponding to the missing symbol is empty.
That is, if ${(i, j; k)} \in L\setminus C$, then
$|C_{j}(C)| =|E_{k}(C)| =  0$.
\end{lemma}
\begin{proof}{
Without loss of generality, let $i=1$ and assume that
$C$ contains the entries $\{ (1,x;x) |\ 1 \leq x \leq n - 1\}$
and that position $(1,n)$ is empty.

By Lemma~\ref{lem11} part (2), for each $x$  ($1 \leq x \leq n-1 $)
there exists a Latin interchange $I_x \subseteq L$ such that
$I_x \cap C = \{ (1,x;x)\}$.
Since there is only one empty position in the first row, it follows that
$\{ (1,x;x), (1,n;n) \} \subseteq I_x$.  Now the Latin interchange $I_x$
has a disjoint mate, say $I'_x$.  In this case since
$ (1,x;n)  \in I'_x$,  for some $r$, $ (r,x; n)  \in I_x$, and since
$|I_x \cap C| = 1$,  $ (r,x;n) \in L\setminus C$.
So symbol $n$ does not occur in column $x$ of $C$.
Since $x$ ranges over all columns from 1 to $n-1$,
symbol $n$ does not occur in $C$ at all.
Therefore $|E_{n}(C)| = 0$.

Also we have $(1,n;x)  \in I'_x$. Thus for some $s$, $ (s,n;x) \in I_x$.
Similarly we have $ (s,n;x)  \not\in C$; therefore no symbol apart from
$n$ may occur in column $n$ in $C$, and we have said that symbol $n$ does
not occur in column $n$ either. Therefore column $n$ is empty.
So $|C_{n}(C)| = 0$.
}\end{proof}

We can generalize Lemma~\ref{3.1} to the following.

\begin{lemma}
\label{3.2}
Let $C$ be a critical set for a Latin square $L$ and assume that
there exists $i$, such that
$|R_i(C)| = n-m$, where
$\{(i, c_1; e_1), (i, c_2; e_2)\dots, (i, c_m; e_m) \}
\subseteq L\setminus C$ \  and
$\{(i,c_{m+1};e_{m+1}), \dots, (i, c_n; e_n) \}\subseteq  C$.
Then we have
\begin{itemize}
\item[$(1)$]
In each of the columns $c_{m+1},c_{m+2}, \dots, c_n$ in $C$, at least
one of the symbols $e_1, e_2, \dots, e_m$ is missing.
That is, for each $x \in \{c_{m+1}, c_{m+2}, \dots, c_n \}$,
there exists a symbol $y \in \{e_1, e_2, \dots, e_m \}$, and a row
$r \in \{1,2, 3, \dots, n \}\setminus \{i\}$
such that $(r, x; y)  \in L\setminus C$.
\item[$(2)$]
For each symbol $e \in \{e_{m+1},e_{m+2}, ..., e_n\}$, we have
a column $c \in \{c_1, c_2, \dots, c_m \}$, from which
this symbol is missing.
\end{itemize}
\end{lemma}
\begin{proof}{(1)
Without loss of generality we may assume that $i=1$ and $c_j = e_j = j$, for
$j = 1,2, \dots, n$.
For each $x \in \{m+1, m+2, \dots, n \}$,
there exists a Latin interchange $I_x$ such that
$I_x\subseteq L$ and $I_x \cap C = \{ (1, x;x) \}$.
So if $I'_x$ is the disjoint mate of $I_x$ then there exists
$y \in \{1, 2, \dots, m \}$ such that $(1,x;y) \in I'_x$, implying that
there exists $r \in \{ 2, \dots, n \}$ such that $(r,x;y) \in I_x$.
Since $|I_x \cap C| = 1$, $(r,x;y) \in L\setminus C$.

\noindent
(2) Similarly
for each $e \in \{m+1, m+2, \dots, n \}$,
there exists a Latin interchange $I_e$ such that
$I_e\subseteq L$ and $I_e \cap C = \{ (1, e;e) \}$.
So if $I'_e$ is the disjoint mate of $I_e$ then there exists
$c \in \{1, 2, \dots, m \}$ such that $(1,c;e) \in I'_e$, implying that
there exists $s \in \{ 2, \dots, n \}$ such that $(s,c;e) \in I_e$.
Since $|I_e \cap C| = 1$, $(s,c;e) \in L\setminus C$.
}\end{proof}
\begin{theorem}
\label{th3.1}
If $C$ is a uniquely completable partial Latin square of order $n$
completing to the Latin square $L$
with $| C | > n^2-3n+3$, then $C$ is not a critical set.
\end{theorem}
\begin{proof}{
We prove this result by contradiction. Suppose $C$ is a critical set.
Since a critical set in a Latin square of order $n$ cannot have $n$
triples whose $i$th components are the same ($1 \leq i \leq 3$)
(see for example \cite{MR80j:05022}),
we can assume that any row or column contains at most $n-1$ symbols
and any symbol occurs at most $n-1$ times.

We have three cases to consider.

\noindent
{\bf Case 1 \ }
There exists a row $i$ such that $|R_{i}(C)| = n-1$. Assume that
$(i,j;k) \in L\setminus C.$
Then by Lemma~\ref{3.1}, $|C_{j}(C)| = |E_{k}(C)| =  0$.
Now if there exists $j'$ ($j' \neq j$) such that $|C_{j'}(C)| = n-1$
and $(i', j'; k') \in L\setminus C$,
then we have $|R_{i'}(C)| = 0$. These together imply that
$|C| \le n^2 - (2n-1)-(n-2)=n^2 - 3n + 3$.  Otherwise
for all $l$, $1 \leq l \leq n$, $|C_{l}(C)| \leq n-2$ and thus
$|C|\leq  n(n -2) -(n-2) =  n^2 - 3n + 2$.

\noindent
{\bf Case 2 \ }
For all $i$ ($1\le i \le n$) we have $|R_{i}(C)| \leq n-3$. Then
$|C| \leq n(n - 3) = n^2 - 3n$.

\noindent
{\bf Case 3 \ } For all $i$ ($1\le i \le n$) we have $|R_i(C)| \leq n-2$ and
there exists a row $r$ such that $|R_{r}(C)| = n-2$.
And by contrast for all $j$ ($1\le j \le n$) we have $|C_j(C)| \leq n-2$.
Assume that $R_r(C) = \{e_{3},e_{4}, \dots, e_{n}\}$,
and $\{(r, c_{1}; e_{1}), (r, c_{2}; e_{2})\}\subset L\setminus C $.
Then by Lemma~\ref{3.2}
each of the symbols $e_{3}, e_4, \dots, e_{n}$
occurs at most once in columns $c_{1}$ and $c_{2}$.   This means
$|C_{c_{1}}(C)| + |C_{c_{2}}(C)| \leq n$. Thus
$|C|\leq  n(n -2) -(n -4) =  n^2 - 3n + 4$.
We shall show that $|C| =  n^2 - 3n + 4$ is also impossible. Proof of this fact
is somewhat involved and we need to introduce more notation.

First note that if we consider the conjugate of the Latin square $L$
we may assume that for all $k$ ($1\le k \le n$) we have $|E_k(C)| \leq n-2$.
Let  $f_k = n-2 - |E_k(C)|$. We have $f_k \ge 0$, for all $k$ ($1\le k \le n$)
and
\begin{displaymath}
\sum_{k=1}^{n}
f_k = n(n-2)- |C|= n-4.
\end{displaymath}
For each position $(i,j)$, $1 \leq i,j \leq n$, we have $x_{i,j}(C) = |R_i(C) \cup C_j(C)|$.
We have
\begin{displaymath}
\hspace*{-3.3cm}
(*)
\hspace*{3cm}
\sum_{1 \le i,j \le n}
x_{i,j}(C)= n^3 - \sum_{k=1}^{n}(n-|E_k(C)|)^2.
\end{displaymath}
In fact, for each position $(i,j)$, $1 \leq i,j \leq n$, we have $x_{i,j}(C) = n$, except when a
symbol $k$ is missing from {\it both}
row $i$ and column $j$ in $C$.
For each $k$ we have
exactly $(n-|E_k(C)|)^2$  such positions. They are the positions which are in
the $(n-|E_k(C)|) \times (n-|E_k(C)|)$ subsquare obtained from the $n \times n$
array by omitting all the rows and columns containing symbol $k$ in $C$.
Each such position causes a ``$-1$'' in the summation of the
left hand side of $(*)$.

Note that since $C$ is a critical set, for each  position
$(i,j) \in L\setminus C$, that is, for each position in $L$ in which $C$ is
empty, we have \ $x_{i,j}(C) \leq n-1$.  Recall that the shape of a partial
Latin square $P$ is $S(P) = \{ (i,j) \mid (i,j;k) \in P \}$.  Thus
\begin{displaymath}
\begin{array}{ccl}
\displaystyle{\frac{1}{|C|}}{\displaystyle{\sum_{(i,j)\in C} x_{i,j}(C)}} & = &
\displaystyle{\frac{1}{|C|}}\Big( (n^3 -{\displaystyle{ \sum_{k=1}^{n}(n-|E_k(C)|)^2)}} -
{\displaystyle{\sum_{(i,j)\in S(L\setminus C)} x_{i,j}(C)}} \Big)    \\
&\ge &
\displaystyle{\frac{1}{n^2-3n+4}}\Big( (n^3 -{\displaystyle{ \sum_{k=1}^{n}(f_k+2)^2)}} -
(3n-4)(n-1) \Big)  \\
& = &
\displaystyle{\frac{1}{n^2-3n+4}}{\displaystyle(n^3 - 3n^2-n +12 -{\sum_{k=1}^{n}f^2_k})}. \\
\end{array}
\end{displaymath}
Since \
${\displaystyle{\sum_{k=1}^{n}f^2_k}} \le
{\displaystyle{(\sum_{k=1}^{n}f_k})^2} = (n-4)^2$, \ it follows that \\
$$
\displaystyle{\frac{1}{|C|}}{\displaystyle{\sum_{(i,j)\in S(C)} x_{i,j}(C)}} \ge
\displaystyle{\frac{n^3 - 3n^2-n +12 -(n-4)^2}{n^2-3n+4}} = n-1.
$$

\noindent
This implies that, either
\begin{enumerate}
\item[(i)]
for some position $(i,j)\in S(C)$ we have $x_{i,j}(C) > n-1$; or
\item[(ii)]
for all $(i,j)\in S(C)$,  $x_{i,j}(C) = n-1$.
\end{enumerate}
The first case is contradictory with $C$ being a critical set.
In the second case if we remove an entry $(a,b;e) \in C$ and let $C' = C \setminus \{ (a,b,e) \}$, then we have
\begin{itemize}
\item
 $x_{a,b}(C') = n-2$ \ and \
 $x_{a,j}(C') , x_{i,b}(C') \le n-1$, for all \ $(a,j) \ {\rm and} \ (i,b) \in S(C')$; \
 and
 \item
  $x_{i,j}(C') = n-1$; otherwise.
\end{itemize}


But if case (ii)
holds, then all of the inequalities that we have above
must be equalities, and
this implies that for every $(i,j)\not\in S(C)$, we have $x_{i,j}(C) = n - 1$.
This follows because we have used the inequality  $x_{i,j}(C)\le  n-1$, 
where $(i,j) \in S(L \setminus C)$.
So  $C'$ can be completed to $L$, first by completing any
position not in the row $a$ or column $b$, then the positions of row $a$
and column $b$. This is a contradiction.

}\end{proof}
\section{Conjectures and Questions}\label{sec43}
%
%
The study of lower bounds on \lcs{n} has been examined by many researchers.
While the work presented in this thesis improves on this bound, it does not settle the open problem
of what the exact value of \lcs{n} is.
Here we list some conjectures and questions which arise from this research.
\begin{conjecture}      \ \
\lcs{n} $ \leq n^2 - n^{3/2}$.
\end{conjecture}
This is motivated by the proof of Theorem~\ref{th3.1}.
It is analogous to a similar conjecture
made by Brankovi{\'c}, Hor{\'a}k, Miller, and Rosa, in \cite{mirka1}, concerning
the size of the largest premature partial Latin square.
\begin{conjecture} \ \
\lcs{n} $\leq (1 - (\frac{3}{4})^{log_2{n}}) n^{2}$.
\end{conjecture}
This is true for the current known values of \lcs{n}. (That is, $1 \leq n 
\leq 17$.) It implies that \lcs{2^n} $= 4^{n} - 3^{n}$.

This conjecture is based on Stinson and van Rees's result in
\cite{MR84g:05036} that \lcs{2^n} $\geq 4^{n} - 3^{n}$.
We postulate that this is an equality.  Below in Questions~\ref{q432} and \ref{q433}, we ask how $I(n)$, the
maximum number of intercalates in an $n \times n$ Latin square, and
\lcs{n} are related.  This conjecture assumes that as the value of $I(n)$
reaches a theoretical maximum when $n$ is a power of 2, so too does the value of \lcs{n}.

\begin{question}\label{q431}
Where $C$ is a critical set of order $n$ and of size \lcs{n},
do there exist $i,j,k$, $1 \leq i,j,k \leq n$, such that
$|R_{i}(C)| = |C_{j}(C)| = |E_{k}(C)| = 0$?  That is, is there
always an empty row, an empty column, and a missing symbol in a critical
set of size \lcs{n}?
\end{question}
Evidence for the ``yes'' case in Question~\ref{q431} is that every
example in Stinson and van Rees ~\cite{MR84g:05036}, and in Donovan
~\cite{MR99k:05038} where critical sets of largest known size are given,
have this property.  Also, every critical set of largest size in Latin
squares of orders 1 to 6 has this property.  Additionally, all of the
largest known critical sets in Latin squares of orders from 8 and 9 have
this property.

All the constructions for large critical sets given in articles
such as~\cite{MR99a:05018},~\cite{MR96h:05030},~\cite{nel2} and~\cite{MR84g:05036}
have this property.  However, the example of a critical set of largest known size in a Latin square
of order 10, given in Appendix~\ref{app1}, does not have this property.  Also, as
we found in Chapter~\ref{ch3}, there are many critical sets of order 7 and size 25
from the Latin square corresponding to STS(7) which do not have this property.

\begin{question}\label{q432}
Where $C$ is a critical set for the Latin square $L$ of order $n$ and size
\lcs{n}, does $L$ have $I(n)$ intercalates?
\end{question}

\begin{question}\label{q433}
Where $L$ is a Latin square of order $n$ with $I(n)$ intercalates,
does $L$ contain a critical set $C$ of size \lcs{n}?
\end{question}

In what follows, we examine evidence for and against both of these 
questions.

Evidence for the ``yes'' case in both of these questions is that all of
the largest known critical sets in Latin squares of orders 1 to 6 and 8 have
the property that they occur in Latin squares with the largest known
number of intercalates.

Also, for each order of Latin square $n$, $1 \leq n \leq 6$, the largest
critical set of order $n$ occurs only in the Latin square with $I(n)$
intercalates.

The original construction for a critical set of size $\displaystyle{\frac{n^2-n}{2}}$,
given by Nelder \cite{nel2}, is in the back-circulant Latin square of
order $n$.  However, apart from this construction, all known constructions
for large critical sets complete to Latin squares which provide lower bounds
for $I(n)$ in ~\cite{MR84g:05034}, as shown in this list.

\begin{itemize}
\item \lcs{2^m} $\geq 4^m - 3^m$,  ~\cite{MR84g:05036}.  The completion of
this critical set is isomorphic to the Latin square in 
~\cite{MR84g:05034} with $I(2^m) = \displaystyle\frac{4^m(2^m-1)}{4}$ intercalates.
\item \lcs{2^m-1} $\geq 4^n - 3^n - 2^{n+1} + 3$, ~\cite{MR96h:05030}.  The
completion of this critical set is isomorphic to the Latin square in 
~\cite{MR84g:05034} with $I(2^m-1) = \displaystyle\frac{(2^m-1)(2^m-2)(2^m-4)}{4}$ intercalates.
\item \lcs{2m} $\geq \displaystyle\frac{5m^2-3m}{2}$, ~\cite{MR99a:05018}.
The completion of this critical set is isomorphic to the Latin square
in~\cite{MR84g:05034} with $m^3$ intercalates, which demonstrated that 
$I(2^m) \geq m^3$.  In the next chapter, the completion of this critical set 
is denoted by $L_2$ and the Latin square in \cite{MR84g:05034} is denoted by 
$L_1$, and we find that $L_1^{-1} = L_2$.

\end{itemize}

As the number of intercalates in a Latin square increases, any
critical set in such a Latin square must intersect an increasing number
of intercalates.  Therefore, it seems reasonable to assume that, in general,
such critical sets would grow in size.  A ``yes'' answer for Question~\ref{q433} would fit in
with this expectation.

We now examine the evidence for the ``no'' case.  The largest known
critical set of order 9 does not come from a Latin square with $I(9)$
intercalates, since all known examples of this size are derived
from Latin squares with 53 or 54 intercalates, and yet Heinrich and
Wallis~\cite{MR84g:05034} found I(9) $\geq$ 64 and more recently the author
of this thesis
found I(9) $\geq$ 72, which had been independently discovered by Owens
and Preece~\cite{MR96g:05029}.

Also, the largest known critical set of order 10 comes from a Latin
square with 117 intercalates, but Heinrich and Wallis~\cite{MR84g:05034}
found I(10) $\geq$ 125.  No critical set of size greater than 55 has
been found in the order 10 Latin square with 125 intercalates.

Additionally, there are at least 113 (out of 147) main classes of
$7 \times 7$ Latin squares which contain a critical set of size 25.


\section{Conclusion}
In this chapter we have developed a new upper bound on \lcs{n} which has improved
considerably the bound given by Curran and van Rees in ~\cite{MR80j:05022}.
We also speculated on the evidence given in a multitude of papers which
links constructions for large critical sets and the classic paper on
the maximum number of intercalates in a Latin square.  Such links have
not been made before.  Additional observations about the nature of
published large critical sets and a large amount of data about large
critical sets led to further conjectures and questions, for which there
is conflicting evidence.

In the last chapter of this thesis, we use the new upper bound on \lcs{n}
to calculate the value of \lcs{6} more quickly.  In the next chapter,
we provide more evidence for the close link between Latin squares with
large numbers of intercalates and large critical sets in these Latin 
squares.
\chapter{New constructions for intercalate-rich Latin squares and their large critical sets}\label{ch5}

There are two well-known papers giving bounds on the value of $I(n)$,
the maximum number of intercalates in an $n \times n$ Latin square.
The first is by Heinrich and Wallis \cite{MR84g:05034} (1980) and the
second is by Kotzig and Zaks \cite{MR85h:05026} (1983).  This chapter
gives new bounds on the values of $I(2^\alpha m)$ and $I(2^\alpha m +
1)$ when $\alpha \geq 2$ ($\alpha \neq 3$ in the $I(2^\alpha m + 1)$
case) and $m$ is odd, by constructing the relevant Latin squares of orders
$2^\alpha m$ and $2^\alpha m + 1$.

Heinrich and Wallis proved that $I(2^\alpha m) \geq (2^\alpha m)^2 
(2^\alpha m+2^\alpha - 2)/8$, for $\alpha \geq 1$ and $m$ odd.
We shall show that $I(2^\alpha m) \geq (2^\alpha m)^2
(3m.2^\alpha + 2^\alpha - 4)/16$, for $\alpha \geq 2$ and $m$ odd.
This is an improvement because the Latin square constructed in this chapter
contains an extra 
$(2^\alpha m)^2 (2^\alpha m - 2^\alpha)/16$ intercalates.

Also, by using the technique of prolonging a transversal, (defined
in Chapter~\ref{ch2}), in the
Latin square of order $2^\alpha m$, Heinrich and Wallis found
$I(2^\alpha m+1) \geq 2^\alpha m ( 2^\alpha m (2^\alpha m + 2^\alpha -
10) / 8 + m + 1 ) + 2^{\alpha-1} m(m-1)$, for $\alpha \geq 2$ and $m$ odd.
By prolonging a different transversal in our newly constructed square of order
$2^\alpha m$, we shall show that $I(2^\alpha m + 1) \geq 2^\alpha m (
2^\alpha m (3.2^\alpha m+2^\alpha - 20)/16 + m + 1 ) + 2^{\alpha-1}
m(m-1)$, for $\alpha = 2$ or $\alpha \geq 4$ and $m$ odd.
The Latin square constructed in this chapter also contains an extra $(2^\alpha m)^2 (2^\alpha m - 2^\alpha)/16$
intercalates.

Both of these bounds are greater than the Heinrich and Wallis bound, and 
represent significant improvements.

We noted in the previous chapter that all except one of the constructions
for the largest known critical sets complete to Latin squares which
are isomorphic to those given in Heinrich and Wallis's paper.  By
combining constructions for critical sets mentioned in 
Donovan and Cooper \cite{MR97g:05032}, together with
the above mentioned construction for a Latin square of order $2^\alpha m$,
we can find a new lower bound for \lcs{4m}, for $m$
any positive integer.  We also give a construction for an $11 \times 11$
Latin square with a record number of intercalates, and we comment
on critical sets from intercalate-rich $14 \times 14$ Latin squares.

The discoveries of this $11 \times 11$ Latin square with 172 intercalates
and a $12 \times 12$
Latin square with 324 intercalates were
a result of joint work with Ian Wanless.  The generalization of
this $12 \times 12$ Latin square to derive a new bound for $I(2^\alpha m)$,
the construction giving the new bound for $I(2^\alpha m+1)$, and the new
critical set construction giving a new bound for \lcs{4m}, are all
new and original work, by the author of this thesis.

\section{Background}\label{sec51}
The rest of this chapter will involve the concatenation of Latin squares
and partial Latin squares to form larger Latin and partial Latin squares,
in order to create new bounds on the maximum number of intercalates
in Latin squares of order $2^\alpha m$, $m$ odd and $\alpha \geq 2$, and the size of critical
sets in such Latin squares.  We define new notation to clarify this process.  

For a partial Latin square $P$ of order $m$, define
\begin{eqnarray*}
S(P,x,y,z) = \{ ( xm + i, ym + j; zm + P_{ij} ) \mid (i,j;k) \in P \}
\end{eqnarray*}

In this chapter, we number from zero to $n-1$ for the rows, columns and symbols in a Latin square of order $n$,
as it is more convenient for the frequent modulo $n$ arithmetic which is used.

Heinrich and Wallis gave the following construction for a Latin
square $L_1$ of order $2m$, $m$ odd, with at least $m^3$ intercalates.
For $L_1$ and the subsequent constructions we shall demonstrate
that there are exactly $m^3$ intercalates.

Let $A = \{ (i,j; (i-j) \pmod{m}) \mid 0 \leq i,j \leq m-1 \}$
and $B = \{ (i,j; (i+j) \pmod{m}) \mid 0 \leq i,j \leq m-1 \}$.
Then if we re-order the columns of $A$ in the order: column $0$, column
$m-1$, column $m-2$, \dots, column $1$, the result is $B$.  That is,
for all $0 \leq i,j \leq m-1$, the entry $(i,j; (i-j) \pmod{m})$
in $A$ gets mapped to $(i,(m-j) \pmod{m}; (i-j) \pmod{m})$
= $(i,k;(i+k) \pmod{m})$ in $B$, where $k = (m-j) \pmod{m}$.  Thus $A$ and $B$
are isotopic.

Recall from Chapter~\ref{ch2} that $M^n$ denotes the Latin square $M$ of order $m$
with $nm$ added to each of the symbols.  We denote the transpose of $M$
by $M^T$.  Then $L_1$ can be diagrammatically represented as follows.

\begin{center}
\begin{tabular}{c}
\begin{tabular}{|c|c|}
\hline
$A^0$ & $B^1$ \\
\hline
$B^1$ & $A^0$ \\
\hline
\end{tabular} \\[3pt]
 $L_1$
\end{tabular}
\end{center}
Thus $L_1 = S(A,0,0,0) \cup S(B,0,1,1) \cup S(B,1,0,1) \cup S(A,1,1,0)$.

We wish to prove that there exist $m^3$ intercalates in $L_1$. 

For any $i,j \in \{ 0,1, \dots, m-1 \}$, $(i, j; (i-j) \pmod{m}) \in S(A,0,0,0)$,
and for any $l \in \{ 0,1, \dots, m-1 \}$, $(i, l+m; (i+l) \pmod{m}+m) \in S(B,0,1,1)$.

In addition, $((i+l-j) \pmod{m}+m, j; (i+l) \pmod{m}+m) \in S(B,1,0,1)$. 

We need to show that 
\begin{eqnarray*}
&& ((i+l-j) \pmod{m}+m, l+m; ((i+l-j) - l)  \pmod{m} ) \\
&=& ((i+l-j) \pmod{m}+m, l+m; (i-j) \pmod{m} )
\end{eqnarray*}
which it obviously does.

So 
\begin{eqnarray*}
I_1 &=& \{ (i,j; (i-j)\pmod{m}), (i,l+m; (i+l) \pmod{m}+m), \\
&& ((i-j+l) \pmod{m}+m,j; (i+l) \pmod{m}+m),\\
&& ((i-j+l) \pmod{m}+m,l+m; (i-j)\pmod{m}) \mid \\
&& 0 \leq i,j,l \leq m-1 \}
\end{eqnarray*}
is an intercalate.

Since $m$ is odd, $B$ contains no intercalates \cite{njcint} and
since $A$ is isomorphic to $B$, $A$ contains no intercalates either.
Every pair of entries $(i,j; i-j)$ and $(i,l+m; (i+l) \pmod{m}+m)$ of $L_1$
where $0 \leq i,j,l \leq m-1$ is contained in an
intercalate.  Therefore there are exactly $m^3$ intercalates in $L_1$.  Similar
arguments will apply to $L_2, L_3, L_4, L_5$ and $L_6$ below.

We define $L_2, L_3, L_4, L_5$ and $L_6$, respectively, as follows:

\begin{center}
\begin{tabular}{@{}c@{\hspace{5pt}}c@{\hspace{5pt}}c@{\hspace{5pt}}c@{\hspace{5pt}}c@{}}

\begin{tabular}{|c|c|}
\hline
$A^0$ & $(A^1)^T$ \\
\hline
$A^1$ & $(A^0)^T$ \\
\hline
\end{tabular}
&
\begin{tabular}{|c|c|}
\hline
$(A^0)^T$ & $(A^1)^T$ \\
\hline
$A^1$ & $A^0$ \\
\hline
\end{tabular}
&

\begin{tabular}{|c|c|}
\hline
$(A^0)^T$ & $A^1$ \\
\hline
$(A^1)^T$ & $A^0$ \\
\hline
\end{tabular}
&

\begin{tabular}{|c|c|}
\hline
$A^0$ & $A^1$ \\
\hline
$(A^1)^T$ & $(A^0)^T$ \\
\hline
\end{tabular}
&

\begin{tabular}{|c|c|}
\hline
$(A^0)^T$ & $B^1$ \\
\hline
$B^1$ & $(A^0)^T$ \\
\hline
\end{tabular}\\[3pt]
 $L_2$ & $L_3$ & $L_4$ & $L_5$ & $L_6$
\end{tabular}
\end{center}

Thus 
\begin{eqnarray*}
L_2 &=& S(A,0,0,0) \cup S(A^T,0,1,1) \cup S(A,1,0,1) \cup S(A^T,1,1,0), \\
L_3 &=& S(A^T,0,0,0) \cup S(A^T,0,1,1) \cup S(A,1,0,1) \cup S(A,1,1,0), \\
L_4 &=& S(A^T,0,0,0) \cup S(A,0,1,1) \cup S(A^T,1,0,1) \cup S(A,1,1,0), \\
L_5 &=& S(A,0,0,0) \cup S(A,0,1,1) \cup S(A^T,1,0,1) \cup S(A^T,1,1,0) {\rm ~and} \\
L_6 &=& S(A^T,0,0,0) \cup S(B,0,1,1) \cup S(B,1,0,1) \cup S(A^T,1,1,0).
\end{eqnarray*}

There exist $m^3$ intercalates in each of $L_2, L_3, L_4, L_5$, and $L_6$.

To prove that $L_2$ contains $m^3$ intercalates we proceed as before.
For any $i,j \in \{ 0,1, \dots, m-1 \}$, $(i, j; (i-j) \pmod{m} ) \in S(A,0,0,0)$.

Then take any $l \in \{0,1, \dots, m-1 \}$, $(i, l+m; (l-i) \pmod{m} + m ) \in S(A^T,0,1,1)$ and
$((l-i+j) \pmod{m} + m, j; (l-i) \pmod{m}+m) \in S(A,1,0,1)$.

Now we need to check that
\begin{eqnarray*}
&&( (l-i+j) \pmod{m} + m, l+m; (l-(l-i+j)) \pmod{m} ) \\
&=& ((l-i+j) \pmod{m} + m, l+m; (i-j) \pmod{m}) 
\end{eqnarray*}
which it obviously does.

Thus there exist $m^3$ intercalates in $L_2$ of the form:
\begin{eqnarray*}
I_2 &=& \{(i,j;(i-j) \pmod{m}),(i,l+m;(l-i) \pmod{m}+m),\\
&&((l-i+j) \pmod{m}+m,j;(l-i) \pmod{m}+m), \\
&& ((l-i+j) \pmod{m}+m,l+m;(i-j) \pmod{m}) \mid \\
&& 0 \leq i,j,l \leq m-1 \}
\end{eqnarray*}


It follows, since $L_3 = L_2^T$ and $L_6 = L_1^T$, that $L_3$ and $L_6$
each contain $m^3$ intercalates.  Also, since each of the Latin squares
is made up of the union of four subsquares, as in the definition
above, it is obvious that transposing all four of the subsquares will
not affect the number of intercalates.  This kind of transposition
maps $L_2$ to $L_4$ and $L_3$ to $L_5$.  Therefore $L_5$ and $L_4$
each contain $m^3$ intercalates.

We recall from Chapter~\ref{ch2} that one of the six conjugates of a Latin square 
$L$ is $L^{-1} = \{ (i, k; j) \mid (i, j; k) \in L \}$.
We find that $L_1^{-1} = L_2$, and $L_3, L_4$ and $L_5$ must be in the
same main class as $L_2$ by the previous arguments.

\section{The $2^\alpha m \times 2^\alpha m$ construction}
We can combine $L_1, L_2, L_3$ and $L_6$ to reach a Latin square
$L'$ of order $4m$, $m$ odd, as follows.  We note that the underlying
structure of $L'$ corresponds to ${\mathbb Z}_2^2$, as displayed below.


\begin{center}
\begin{tabular}{c c}
\begin{tabular}{|c|c|c|c|}
\hline
$A^0$ & $(A^1)^T$ & $(A^2)^T$ & $B^3$ \\
\hline
$A^1$ & $(A^0)^T$ & $B^3$ & $(A^2)^T$ \\
\hline
$A^2$ & $B^3$ & $(A^0)^T$ & $(A^1)^T$ \\
\hline
$B^3$ & $A^2$ & $A^1$ & $A^0$ \\
\hline
\end{tabular}
&
\begin{latinsq}{4}
\hline 0 & 1 & 2 & 3 \\
\hline 1 & 0 & 3 & 2 \\
\hline 2 & 3 & 0 & 1 \\
\hline 3 & 2 & 1 & 0 \\
\hline \end{latinsq}
\\
$L'$ & $\mathbb{Z}_2^2$
\end{tabular}
\end{center}

Thus 
\begin{eqnarray*}
L' &=& S(A,0,0,0) \cup S(A^T,0,1,1) \cup S(A^T,0,2,2) \cup S(B,0,3,3) \cup \\
&&      S(A,1,0,1) \cup S(A^T,1,1,0) \cup S(B,1,2,3) \cup S(A^T,1,3,2) \cup \\
&&      S(A,2,0,2) \cup S(B,2,1,3) \cup S(A^T,2,2,0) \cup S(A^T,2,3,1) \cup \\
&&      S(B,3,0,3) \cup S(A,3,1,2) \cup S(A,3,2,1) \cup S(A,3,3,0).
\end{eqnarray*}

We now use $L'$ to verify that  
$I(2^\alpha m) \geq (2^\alpha m)^2 (3m 2^\alpha+2^\alpha-4)/16$, when $\alpha \geq 2$.
\begin{theorem}
For $\alpha \geq 2$,
$I(2^\alpha m) \geq (2^\alpha m)^2 (3m 2^\alpha+2^\alpha-4)/16$.
\end{theorem}

\begin{proof}
{
Consider ${\mathbb Z}_2^2$ displayed above; it contains 12 distinct intercalates.
Then if $\{(r_1,c_1;e_1),(r_1,c_2;e_2),(r_2,c_1;e_2),(r_2,c_2;e_1)\}$
is an intercalate in ${\mathbb Z}_2^2$, we have that
\begin{eqnarray*}
D &=& \{ (i,j;L'_{ij}) \mid ((c_1m \leq j \leq c_1m+m-1) \vee (c_2m \leq j \leq c_2m+m-1)) \wedge \\
&& ((r_1m \leq i \leq r_1m+m-1) \vee (r_2m \leq i \leq r_2m+m-1)) \}
\end{eqnarray*}
is a subsquare of $L'$ which is isomorphic to one of $L_1, L_2, L_3, L_4, L_5$
or $L_6$ and thus contains exactly $m^3$ intercalates.  Therefore $I(L') = 12m^3$, and thus $I(4m) \geq 12m^3$.

Heinrich and Wallis counted the number of intercalates in the direct
product of two Latin squares $M$ (of order $k$) and $N$ (of order $l$).
This count was used to create a new lower bound on $I(kl)$:
\begin{eqnarray*}
I(kl) \geq I(k) l^2 + I(l) k^2 + 4.I(k).I(l)
\end{eqnarray*}

If we use the Latin squares $M = {\mathbb Z}_2^{\alpha-2}$ and $N = L'$,
we may use this formula with $k = 2^{\alpha-2}$ and $l = 4m$
and the values $I(2^{\alpha-2}) = \displaystyle{\frac{(2^{\alpha-2})^2
(2^{\alpha-2}-1)}{4}}$ (known from Heinrich and Wallis), and $I(4m)
\geq 12m^3$.  Thus we deduce that:
\begin{eqnarray*}
I(2^\alpha m) &\geq& I(4m).(2^{\alpha-2})^2 + I(2^{\alpha-2}).(4m)^2 + 4.I(2^{\alpha-2}).I(4m) \\
&=& 12m^3.2^{2\alpha-4} + 16m^2.\displaystyle{\frac{(2^{\alpha-2})^2 (2^{\alpha-2}-1)}{4}} +  \\
&& 4.\displaystyle{\frac{(2^{\alpha-2})^2 (2^{\alpha-2}-1)}{4}}.12m^3 \\
&=& (2^\alpha m)^2 (3m.2^\alpha + 2^\alpha - 4)/16
\end{eqnarray*}
}
\end{proof}

We comment on an attempt to generalise this construction to $8m \times 8m$
Latin squares.  It was expected, since $I(4m) \geq I(4) m^3$, and since
this had been obtained by ``doubling'' previous constructions in
a clever way, that a construction could be obtained which would give
$I(8m) \geq I(8) m^3$, that is, $I(8m) \geq 112 m^3$.

There is a total of twelve $2m \times 2m$ Latin squares containing
$m \times m$ subsquares isomorphic to $A, A^T$ and $B$.  We shall
refer to these subsquares as $L_1, \dots, L_{12}$.  

All concatenations of four Latin squares from $L_1, \dots, L_{12}$ into
$4m \times 4m$ squares were examined by computer.
Thus, a total of $12^4$ $4m \times 4m$ squares were examined,
giving 96 subsquares similar to $L'$ which contain $12m^3$ intercalates.
(We number these $L'_1, \dots, L'_{96}$.)
Then all concatenations of four Latin squares from $L'_1, \dots, L'_{96}$ into $8m
\times 8m$ squares were examined.  This is a total of $96^4$ $8m \times
8m$ Latin squares.  Unfortunately, all of these Latin squares contained the
number of intercalates predicted by the above theorem.

\section{A critical set of order $4m$}
It will be useful to restate the following lemmas
from Donovan \cite{MR99a:05018}.

\begin{lemma}
\label{lem1}
If $L$ is a Latin square of order $n$, $S$ a subsquare in $L$ and $C$
a critical set in $L$, then $C \cap S$ must have a unique completion in $S$.
\end{lemma}

\begin{lemma}
\label{lem2}
Let $L$ be a Latin square with critical set $C$.  Let $(\alpha, \beta, \gamma)$
be an isotopism from the critical set $C$ onto $C'$.  Then $C'$ is a critical
set in the Latin square $L'$ isotopic to $L$.
\end{lemma}

\begin{lemma}
\label{lem3}
Let $L$ be a Latin square with critical set $C$ and let $C'$ be a 
conjugate of $C$.  Then $C'$ is a critical set in the corresponding
conjugate $L'$ of $L$.
\end{lemma}

We also restate Theorem 2 from Donovan and Cooper~\cite{MR97g:05032}.
\begin{theorem}
\label{thm1}
Let $L$ be the back-circulant Latin square of order $n$; then the set
\begin{eqnarray*}
C &=& \{ (i,j; i+j) \mid i = 0, \dots, n-2 {\rm~and~} j = 0, \dots, n-2-i \}
\end{eqnarray*}
is a critical set in $L$ of size $\displaystyle{\frac{n^2-n}{2}}$.
\end{theorem}

In the $4m \times 4m$ Latin square $L'$ given above we can find a critical set
$P$ of size $\displaystyle{\frac{23m^2-9m}{2}}$.
This construction will work for any integer $m$, not just the odd values.

We define some new notation to create critical sets 
in the Latin squares $A, A^T$ and $B$.
Let $L$ be an $m \times m$ Latin square, and let 
$G(L) = \{ (i,j;L_{ij}) \mid (0 \leq i,j \leq m-1) \wedge (m \leq i+j \leq 2m-2) \}$, 
$H(L) = \{ (i,j;L_{ij}) \mid 1 \leq j \leq i \leq m-1 \},$ and $J(L) = \{
(i,j;L_{ij}) \mid 1 \leq i \leq j \leq m-1 \}$.   
Then $|G(L)| = |H(L)| =  |J(L)| =
\displaystyle{\frac{m^2-m}{2}}$.  Note that $H(L^T) = (J(L))^T$.

Now let $P$ be the following partial Latin square:

\begin{center}
\begin{tabular}{|c|c|c|c|}
\hline
$H(A^0)$ & $H((A^1)^T)$ & $H((A^2)^T)$ & $G(B^3)$ \\
\hline
$J(A^1)$ & $(A^0)^T$ & $G(B^3)$ & $(A^2)^T$ \\
\hline
$J(A^2)$ & $G(B^3)$ & $(A^0)^T$ & $(A^1)^T$ \\
\hline
$G(B^3)$ & $A^2$ & $A^1$ & $A^0$ \\
\hline
\end{tabular}
\end{center}

Thus 
\begin{eqnarray*}
P &=& S(H(A),0,0,0) \cup S(H(A^T),0,1,1) \cup S(H(A^T),0,2,2) \cup \\
&& S(G(B),0,3,3) \cup \\
&&    S(J(A),1,0,1) \cup S(A^T,1,1,0) \cup S(G(B),1,2,3) \cup S(A^T,1,3,2) \cup \\
&&    S(J(A),2,0,2) \cup S(G(B),2,1,3) \cup S(A^T,2,2,0) \cup S(A^T,2,3,1) \cup \\
&&    S(G(B),3,0,3) \cup S(A,3,1,2) \cup S(A,3,2,1) \cup S(A,3,3,0).
\end{eqnarray*}

For example, when $m = 3$, $P$ is the following critical set:

\begin{center}
\begin{latinsq}[10]{12}
\hline  \ecell & \ecell & \ecell & \ecell & \ecell & \ecell & \ecell & \ecell & \ecell & \ecell & \ecell & \ecell \\ 
\hline  \ecell &  0 & \ecell & \ecell &  3 & \ecell & \ecell &  6 & \ecell & \ecell & \ecell &  9 \\ 
\hline  \ecell &  1 &  0 & \ecell &  5 &  3 & \ecell &  8 &  6 & \ecell &  9 & 10 \\ 
\hline  \ecell & \ecell & \ecell &  0 &  1 &  2 & \ecell & \ecell & \ecell &  6 &  7 &  8 \\ 
\hline  \ecell &  3 & 5 &  2 &  0 &  1 & \ecell & \ecell &  9 &  8 &  6 &  7 \\ 
\hline  \ecell & \ecell &  3 &  1 &  2 &  0 & \ecell &  9 & 10 &  7 &  8 &  6 \\ 
\hline  \ecell & \ecell & \ecell & \ecell & \ecell & \ecell &  0 &  1 &  2 &  3 &  4 &  5 \\ 
\hline  \ecell &  6 & 8 & \ecell & \ecell &  9 &  2 &  0 &  1 &  5 &  3 &  4 \\ 
\hline  \ecell & \ecell &  6 & \ecell &  9 & 10 &  1 &  2 &  0 &  4 &  5 &  3 \\ 
\hline  \ecell & \ecell & \ecell &  6 &  8 &  7 &  3 &  5 &  4 &  0 &  2 &  1 \\ 
\hline  \ecell & \ecell &  9 &  7 &  6 &  8 &  4 &  3 &  5 &  1 &  0 &  2 \\ 
\hline  \ecell &  9 & 10 &  8 &  7 &  6 &  5 &  4 &  3 &  2 &  1 &  0 \\ 
\hline
\end{latinsq}
\end{center}

\begin{theorem}
The partial Latin square $P$ is a critical set contained in $L'$, a Latin square of size $4m$, and
$|P| = \displaystyle{\frac{23m^2-9m}{2}}$.  Therefore \lcs{4m} 
$\geq \displaystyle{\frac{23m^2-9m}{2}}$.
\end{theorem}

\begin{proof}
{

%
%

Consider the following partial Latin squares $Q$ and $U$:

\begin{center}
\begin{tabular}{cc}

\begin{tabular}{|c|c|}
\hline
$H(A^0)$ & $H((A^1)^T)$ \\
\hline
$J(A^1)$ & $(A^0)^T$ \\
\hline
\end{tabular}
&

\begin{tabular}{|c|c|}
\hline
$H((A^0)^T)$ & $G(B^1)$ \\
\hline
$G(B^1)$ & $(A^0)^T$ \\
\hline
\end{tabular}
\\[3pt]
 $Q$ & $U$
\end{tabular}
\end{center}

Thus 
\begin{eqnarray*}
Q &=& S(H(A),0,0,0) \cup S(H(A^T),0,1,1) \cup \\
&&  S(J(A),1,0,1) \cup S(A^T,1,1,0), \\
U &=& S(H(A^T),0,0,0) \cup S(G(B),0,1,1) \cup \\
&& S(G(B),1,0,1) \cup S(A^T,1,1,0).
\end{eqnarray*}

Given these sets, we may consider $P$ as:
\begin{eqnarray*}
P &=& S(Q,0,0,0) \cup S(U,0,1,1) \cup S(U^T,1,0,1) \cup S(L_3,1,1,0).
\end{eqnarray*}

We shall begin by proving that $S(Q,0,0,0)$ and $S(U,0,1,1)$ are critical
sets in the Latin squares $S(L_2,0,0,0)$ and $S(L_6,0,1,1)$ respectively,
and so by Lemma~\ref{lem1} every entry in these subarrays and
the subarray $S(U^T,1,0,1)$ is necessary for the unique completion of $P$.

%
%
%

We prove that $Q$ is a critical set for $L_2$.  Consider the partial Latin subsquare
of $Q$, $S(H(A),0,0,0)$.  By Lemmas~\ref{lem1}, \ref{lem2} and \ref{lem3}
and Theorem~\ref{thm1}, $H(A)$ is a critical set for $A$.  Thus if
any entry $(x,y;z)$ is removed from $S(H(A),0,0,0)$, $S(H(A),0,0,0) \setminus \{ (x,y;z)
\}$ will no longer uniquely complete to $A$ and thus the partial Latin
square $Q$ will no longer have unique completion.  
Thus every entry
occurring in $S(H(A),0,0,0)$ is necessary in $Q$ for $Q$ to have unique
completion to $L_2$.  Similar arguments apply for the partial Latin
subsquares $S(J(A),1,0,1)$ and $S(H(A^T),0,1,1)$.

We consider the entries of $Q$ occurring in the Latin subsquare of
$Q$, $S(A^T,1,1,0)$.  For $m \leq i \leq 2m-1$, $m \leq j \leq 2m-1$ and
either $i = m$, or $j \neq m$ and $i \geq j$, there is an intercalate 
\begin{eqnarray*}
I &=& \{ (i,j; (j-i) \pmod{m}), (0,j;j), \\
&&( 0, (i-j) \pmod{m}; (j-i) \pmod{m}), (i, (i-j) \pmod{m}; j) \}
\end{eqnarray*}
such that $I \cap Q = \{( i,j;(j-i) \pmod{m}) \}$.  For $m \leq i \leq 2m-1$, $m
\leq j \leq 2m-1$ and either $j = m$, or $i \neq m$ and $i \leq j$,
there is an intercalate 
\begin{eqnarray*}
I &=& \{ (i,j; (j-i) \pmod{m} ), (i,0;i), \\
&&( (j-i) \pmod{m}, 0; (j-i) \pmod{m}), ((j-i) \pmod{m}, j;  i) \}
\end{eqnarray*}
such that $I \cap Q = \{ (i, j; (j-i) \pmod{m}) \}$.
Therefore every entry occurring in $S(A^T,1,1,0)$ is necessary in $Q$
for $Q$ to have unique completion.

We complete $Q$ to $L_2$ by noting that each row and column of $S(A^T,1,1,0)$ 
contains all of the symbols
$m,$ $\dots,$ $2m-1$ and thus both $S(J(A),1,0,1)$ and $S(H(A^T),0,1,1)$
are forced to use only the symbols $0, \dots, m-1$.  By
Lemmas~\ref{lem1}, \ref{lem2} and \ref{lem3} and Theorem~\ref{thm1}, $J(A)$ 
is a critical set for $A$, $H(A)$ is a critical set for $A$ and 
$H(A^T)$ is a critical set for $A^T$.  
Thus the completions in $Q$ of $S(H(A^T),0,1,1)$
and $S(J(A),1,0,1)$ are forced to be $S(A^T,0,1,1)$ and $S(A,1,0,1)$ respectively,
which forces the unique completion in $Q$ of $S(H(A),0,0,0)$ to $S(A,0,0,0)$.
Thus $Q$ has a unique completion to $L_2$, and is a critical set.
The fact that all the entries in $Q$ are necessary for
unique completion to $L_2$ will be essential
to the proof that $P$ is a critical set, because several subsquares in $P$
are isomorphic to $Q$.

The partial Latin square $U$ is proven to be a critical set for $L_6$
in a similar manner to $Q$.  Every entry in $S(H(A^T),0,0,0)$,
$S(G(B),0,1,1)$ and $S(G(B),1,0,1)$ is required for $U$ to have unique
completion to $L_6$.

We consider the entries of $U$ occurring in the Latin subsquare,
$S(A^T,1,1,0)$.  For $m \leq i \leq 2m-1$, $m \leq j \leq 2m -1$ and $i
\geq j$, there is an intercalate 
\begin{eqnarray*}
I &=& \{ (i, j; (j-i) \pmod{m}), \\
&& (i-j, j; i), (i, 0; i), (i-j, 0; (j-i) \pmod{m}) \}
\end{eqnarray*}
such that $I
\cap U = \{ (i, j; (j-i) \pmod{m}) \}$.  For $m \leq i \leq 2m-1$, $m \leq j \leq
2m-1$ and $i < j$, there is an intercalate
\begin{eqnarray*}
I &=& \{ (i, j; j-i), (i, 2m-1-i; 2m-1), \\
&& (2m-1-i, 2m-1-j; j-i), (2m-1-j, j; 2m-1) \}
\end{eqnarray*}
such that
$I \cap U = \{ (i, j; (j-i) \pmod{m}) \}$.  

The unique completion of $U$ is shown in a manner analogous to $Q$.
Thus $U$ is a critical set in the Latin square $L_6$.
Similarly, by Lemmas~\ref{lem1}, \ref{lem2}, and \ref{lem3}, $U^T$ must
be a critical set for $L_1$.

To complete the proof that $P$ is a critical set for $L'$, we must
also show that the set 
\begin{eqnarray*}
R &=& S(H(A),0,0,0) \cup S(G(B),0,1,1) \cup \\
&& S(G(B),1,0,1) \cup S(A,1,1,0)
\end{eqnarray*}
is a critical set in $L_1$.

The partial Latin square $R$ may be represented by the following diagram.
\begin{center}
\begin{tabular}{|c|c|}
\hline
$H(A^0)$ & $G(B^1)$ \\
\hline
$G(B^1)$ & $A^0$ \\
\hline
\end{tabular}
\end{center}
and we can see that if the mapping $i \rightarrow m-i$, $1 \leq i \leq m-1$,
is applied to the symbols of $R$, then we obtain $U$.  Hence $R$ and $U$
are isotopic.  Thus, by the arguments presented above, $R$ is a critical
set in $L_1$.

We now have enough information to prove that $P$ is a critical set for $L'$.

There are three distinct types of partial Latin subsquare of size $2m \times 2m$
in $P$, which correspond to $Q$, $R$, $U$ and $U^T$.

The partial Latin subsquares in $P$, 
\begin{eqnarray*}
&& S(H(A),0,0,0) \cup S(H(A^T),0,1,1) \cup S(J(A),1,0,1) \cup S(A^T,1,1,0) {\rm ~and} \\
&& S(H(A),0,0,0) \cup S(H(A^T),0,2,2) \cup S(J(A),2,0,2) \cup S(A^T,2,2,0),
\end{eqnarray*}
correspond to $Q$.

The partial Latin subsquare in $P$, 
\begin{eqnarray*}
S(H(A),0,0,0) \cup S(G(B),0,3,3) \cup S(G(B),3,0,3) \cup S(A,3,3,0),
\end{eqnarray*}
corresponds to $R$.

The partial Latin subsquares in $P$, 
\begin{eqnarray*}
&& S(H(A^T),0,2,2) \cup S(G(B),0,3,3) \cup S(G(B),1,2,3) \cup S(A^T,1,3,2) {\rm ~and} \\
&& S(H(A^T),0,1,1) \cup S(G(B),0,3,3) \cup S(G(B),2,1,3) \cup S(A^T,2,3,1),
\end{eqnarray*}
correspond to $U$.  

The partial Latin subsquares in $P$,
\begin{eqnarray*}
&& S(J(A),2,0,2) \cup S(G(B),2,1,3) \cup S(G(B),3,0,3) \cup S(A,3,1,2), \\
&& S(J(A),1,0,1) \cup S(G(B),1,2,3) \cup S(G(B),3,0,3) \cup S(A,3,2,1),
\end{eqnarray*}
correspond to $U^T$.

Now all of the subsquares listed above which correspond to $Q$, $R$,
$U$ and $U^T$ are made up of the union of four partial Latin squares.
If we consider the partial Latin squares that make up these unions,
we find that there is a total of sixteen different partial Latin squares.
These correspond to the sixteen partial Latin squares used in the
first definition of $P$.

We have shown that $Q$, $R$, $U$ and $U^T$ are critical sets for $L_2$,
$L_1$, $L_6$ and $L_1$ respectively.  Thus, if any entry from any of
the sixteen partial Latin squares is removed, one of the partial
Latin squares corresponding to $Q$, $R$, $U$ or $U^T$ above will
not have unique completion.  Therefore,
we have, by Lemmas~\ref{lem1} and \ref{lem2}, all of the entries
in $P$ are necessary for $P$ to have unique completion.  The reasoning
is the same as in the proof that $R$ and $U$ are critical sets. 
Any entry $(x,y;z)$ removed from $P$ in a partial Latin subsquare $X$ corresponding to $Q$, $R$,
$U$ or $U^T$ ensures that the partial Latin subsquare $X \setminus \{(x,y;z)\}$ no longer has unique completion.  Thus the
partial Latin square $P \setminus \{(x,y;z)\}$ also no longer has unique completion.

We complete $P$ by noting that both $S(A,3,3,0)$ $\cup$ $S(A,3,2,1)$
$\cup$ \break $S(A,3,1,2)$ and $S(A,3,3,0)$ $\cup$ $S(A,2,3,1)$ $\cup$ $S(A,1,3,2)$
use all of the symbols $0, \dots, 3m-1$ and thus both $S(G(B),0,3,3)$
and $S(G(B),3,0,3)$, respectively, are forced to use only the symbols
$3m, \dots, 4m-1$. As noted above, $G(B)$ is a critical set for $B$.
Thus $S(G(B),0,3,3)$ and $S(G(B),3,0,3)$ are
forced in $P$ to complete to $S(B,0,3,3)$ and $S(B,3,0,3)$ respectively.
Similar reasoning shows that $S(G(B),2,1,3)$ and $S(G(B),1,2,3)$ are
forced to complete in $P$ to $S(B,2,1,3)$ and $S(B,1,2,3)$ respectively.
Then the rest of the entries in the partial Latin square have forced completion
to $L'$.  Thus $P$ has a unique completion to $L'$, and is a critical set.

Also, we know that $|Q| = \displaystyle{\frac{5m^2-3m}{2}}$ and that
$|U| = 2 \times |G(B)| + |H(A)| + |A| = |Q|$.  Thus
$|P| = |Q| + |U| + |U^T| + |L_3| = \displaystyle{\frac{23m^2-9m}{2}}$.
} 
\end{proof}

This is a significant construction because it is the first which is not
made up of four back-circulant Latin squares combined, and it also
combines several other critical set constructions into one whole construction.
It also gives a better bound for \lcs{4m} than the previous bound given by Donovan \cite{MR99a:05018}
(\lcs{2m} $\geq \displaystyle{\frac{5m^2-3m}{2}}$).

\section{Prolonging the $2^\alpha m \times 2^\alpha m$ construction}

The technique of prolonging an $n \times n$ Latin square along a
transversal to reach an $(n+1) \times (n+1)$ Latin square was described
in Chapter~\ref{ch2}.  We consider prolonging the construction for the $4m \times
4m$ Latin square constructed above.  For $m$ odd, we can prolong $L'$,
given in the last section, along the transversal $T$.  We give $T$ below.
\begin{eqnarray*}
T & = & \{ (i, j; i-j) \mid (0 \leq i,j \leq m-1) \wedge ((i + j) \equiv 0\pmod{m}) \} \cup \\
 && \{ (i, j; i+j) \mid (m \leq i \leq 2m-1) \wedge (2m \leq j \leq 3m-1) \wedge \\
&& (i \equiv j\pmod{m} ) \} \cup \\
 && \{ (i, j; j-i) \mid (2m \leq i \leq 3m-1) \wedge (3m \leq j \leq 4m-1) \wedge \\
&& ((i + j) \equiv  0\pmod{m}) \} \cup \\
&& \{ (i,j ; i-j) \mid (3m \leq i \leq 4m-1) \wedge (m \leq j \leq 2m-1) \wedge \\
&& ((i+j) \equiv 0\pmod{m} ) \}.
\end{eqnarray*}

\begin{theorem}
For $\alpha = 2$ or $\alpha \geq 4$, 
\begin{eqnarray*}
I(2^\alpha m+ 1) &\geq& 2^\alpha m [ 2^\alpha m(3m.2^\alpha + 2^\alpha - 20)/16 + m + 1] + 2^{\alpha-1} m (m-1).
\end{eqnarray*} 
\end{theorem}

\begin{proof}
{
We begin with the direct product $E = D \times L'$, where $D =
{\mathbb Z}_2^{\alpha-2}$.  Since $L'$ has a transversal $T$, we can also find
a transversal in $E$ and prolong it.  This is possible since we know
from Heinrich and Wallis \cite{MR84g:05034} that $D$ also has a transversal (when $\alpha
\neq 3$).  Since $E$ consists of copies of $L'$, the
transversal in $E$ consists of copies of the transversal in $L'$
in the subsquares in $E$ corresponding to the transversal in $D$.

We also know the value of $I(E)$ from the previous theorem.

In the prolongation process, as in Heinrich and Wallis, 
at most $2^\alpha m (2^\alpha m - m)$ intercalates are destroyed
and at least $2^\alpha {m+1 \choose 2}$ intercalates
are recovered.  

We give the reasoning behind this.  For each entry $x$ in a row of
$E$, there are at most $(2^\alpha m - m)$ intercalates containing $x$,
since if we suppose that $x$ falls within a copy of $A$, $A^T$ or $B$, there is
an intercalate created with $x$ and any other entry in the same
row, but outside the copy of $A$, $A^T$ or $B$.  This accounts for the destroyed
intercalates.  However, in each of the $2^\alpha$ copies of $L'$ in $E$
which are affected by the transversal, the substitution of a new symbol
creates ${m+1 \choose 2}$ intercalates in each such copy.  To illustrate
this creation of intercalates, the diagram below gives a copy of $A$
where $m=5$ before and after the prolongation.  The copy of $A$ before
prolongation contains no intercalates, but after prolongation contains
${5+1 \choose 2}$ intercalates.  (The symbol $6$ represents the symbol
which is added after prolongation.  Note that in the actual prolongation of 
$L'$, the final row and column
of the prolongation of $A$ occur in the final row and column of $L'$.)

\begin{center}
\begin{tabular}{cc}
\begin{latinsq}{5}
\hline
1 & 5 & 4 & 3 & 2 \\
\hline
2 & 1 & 5 & 4 & 3 \\
\hline
3 & 2 & 1 & 5 & 4 \\
\hline
4 & 3 & 2 & 1 & 5 \\
\hline
5 & 4 & 3 & 2 & 1 \\
\hline
\end{latinsq}
&
\begin{latinsq}{6}
\hline
6 & 5 & 4 & 3 & 2 & 1 \\
\hline
2 & 1 & 5 & 4 & 6 & 3 \\
\hline
3 & 2 & 1 & 6 & 4 & 5 \\
\hline
4 & 3 & 6 & 1 & 5 & 2 \\
\hline
5 & 6 & 3 & 2 & 1 & 4 \\
\hline
1 & 4 & 2 & 5 & 3 & 6 \\
\hline
\end{latinsq} \\[3pt]
 $A{\rm~before}$ & $A{\rm~after}$
\end{tabular}
\end{center}
Therefore
\begin{eqnarray*}
I(2^\alpha m+ 1) &\geq& (2^\alpha m)^2 (3m.2^\alpha + 2^\alpha - 4)/16 -  2^\alpha m(2^\alpha m - m) + 2^\alpha {m+1 \choose 2} \\
&=& 2^\alpha m (2^\alpha m (3m.2^\alpha  + 2^\alpha - 4)/16 - 2^\alpha m + m + 1) + \\
&& \;\;\; 2^{\alpha-1} m(m-1) \\
&=& 2^\alpha m ( 2^\alpha m(3m.2^\alpha  + 2^\alpha - 20)/16 + m + 1) +  2^{\alpha-1} m (m-1).
\end{eqnarray*} 
}
\end{proof}

We note that this construction actually gives 264 intercalates when
$\alpha = 2$ and $m = 3$, since many more intercalates are added than
are counted above.

In \cite{MR85h:05026}, Kotzig and Zaks showed that $I(4m+1) \leq 
2m(8m^2-4m-1) = 16m^3-8m^2-2m$.  When $\alpha = 2$, the bound above gives $I(4m+1) \geq 12m^3 - 10m^2 + 2m$, and the old Heinrich and Wallis bound
gave $I(4m+1) \geq 8m^3 + 6m^2 - 10m$.  Thus, this new bound
is a significant improvement towards the theoretical bound.

\section{A construction of an $11 \times 11$ intercalate-rich Latin square}
If we begin with the Latin square ${\mathbb Z}_3^2$ and prolong it along
the main diagonal, we reach the following square $M'$:

\begin{center}
\begin{tabular}{|c|c|c|c|c|c|c|c|c|c|}
\hline 9 & 1 & 2 & 3 & 4 & 5 & 6 & 7 & 8 & 0 \\
\hline 1 & 9 & 0 & 4 & 5 & 3 & 7 & 8 & 6 & 2 \\
\hline 2 & 0 & 9 & 5 & 3 & 4 & 8 & 6 & 7 & 1 \\
\hline 3 & 4 & 5 & 9 & 7 & 8 & 0 & 1 & 2 & 6 \\
\hline 4 & 5 & 3 & 7 & 9 & 6 & 1 & 2 & 0 & 8 \\
\hline 5 & 3 & 4 & 8 & 6 & 9 & 2 & 0 & 1 & 7 \\
\hline 6 & 7 & 8 & 0 & 1 & 2 & 9 & 4 & 5 & 3 \\
\hline 7 & 8 & 6 & 1 & 2 & 0 & 4 & 9 & 3 & 5 \\
\hline 8 & 6 & 7 & 2 & 0 & 1 & 5 & 3 & 9 & 4 \\
\hline 0 & 2 & 1 & 6 & 8 & 7 & 3 & 5 & 4 & 9 \\
\hline
\end{tabular}
\vspace{5mm}
\\ $M'$ \\
\end{center}

Then $M'$ has 117 intercalates, and the largest known critical set of
order 10, shown in Appendix~\ref{app1}, is derived from this square.  If we prolong $M'$ along $T$
where 
\begin{eqnarray*}
T &=& \{ ( 3m+i, 3m+j ; 3m+i+j) \mid \\
&& \;\;\; (0 \leq i,j \leq 2) \wedge (j - i \equiv 1 \pmod{3}) \wedge (0 \leq m \leq 2) \} \cup \\
&& \;\;\; \{ (9,9;9) \},
\end{eqnarray*}
we reach the following square $M''$:

\begin{center}
\begin{latinsq}[10]{11}
\hline 9 & 10 &  2 &  3 &  4 &  5 &  6 &  7 &  8 &  0 &  1 \\ 
\hline  1 & 9 & 10 &  4 &  5 &  3 &  7 &  8 &  6 &  2 &  0 \\ 
\hline 10 &  0 &  9 &  5 &  3 &  4 &  8 &  6 &  7 &  1 &  2 \\ 
\hline  3 &  4 &  5 &  9 & 10 &  8 &  0 &  1 &  2 &  6 &  7 \\ 
\hline  4 &  5 &  3 &  7 &  9 & 10 &  1 &  2 &  0 &  8 &  6 \\ 
\hline  5 &  3 &  4 & 10 &  6 &  9 &  2 &  0 &  1 &  7 &  8 \\ 
\hline  6 &  7 &  8 &  0 &  1 &  2 &  9 & 10 &  5 &  3 &  4 \\ 
\hline  7 &  8 &  6 &  1 &  2 &  0 &  4 &  9 & 10 &  5 &  3 \\ 
\hline  8 &  6 &  7 &  2 &  0 &  1 & 10 &  3 &  9 &  4 &  5 \\ 
\hline  0 &  2 &  1 &  6 &  8 &  7 &  3 &  5 &  4 & 10 &  9 \\ 
\hline  2 &  1 &  0 &  8 &  7 &  6 &  5 &  4 &  3 &  9 & 10 \\ 
\hline 
\end{latinsq}
\vspace{5mm}
\\ $M''$ \\
\end{center}

This Latin square has 172 intercalates, which is more than double the
previous bound given by Heinrich and Wallis, who gave $I(11) \geq 80$.
Also, we can find very large critical sets in it; for example, the
following critical set in $M''$ is of size 70.

\begin{center}
\begin{latinsq}{11}
\hline 9 & \ecell & \ecell & 3 & 4 & 5 & 6 & 7 & 8 & 0 & \ecell \\
\hline  1 & 9 & \ecell & 4 & 5 & 3 & 7 & 8 & 6 & 2 & \ecell \\
\hline   \ecell & \ecell & 9 & 5 & 3 & 4 & \ecell & 6 & 7 & 1 & \ecell \\
\hline  3 & 4 & 5 & \ecell & \ecell & 8 & 0 & 1 & 2 & 6 & \ecell \\
\hline   \ecell & \ecell & \ecell & \ecell & \ecell & \ecell & \ecell & \ecell & \ecell & \ecell & \ecell \\
\hline   \ecell & 3 & 4 & \ecell & 6 & 9 & 2 & 0 & 1 & \ecell & \ecell \\
\hline  6 & \ecell & \ecell & \ecell & \ecell & 2 & 9 & \ecell & 5 & 3 & \ecell \\
\hline   \ecell & \ecell & 6 & \ecell & \ecell & 0 & 4 & 9 & \ecell & 5 & \ecell \\
\hline  8 & 6 & \ecell & 2 & \ecell & 1 & \ecell & \ecell & 9 & \ecell & \ecell \\
\hline   \ecell & \ecell & \ecell & 6 & 8 & 7 & 3 & 5 & 4 & \ecell & \ecell \\
\hline  2 & 1 & 0 & 8 & 7 & 6 & 5 & 4 & 3 & 9 & \ecell \\
\hline
\end{latinsq}
\vspace{5mm} \\ Critical set for $M''$ \\
\end{center}
Unfortunately, this construction does not appear to generalise well.

\section{A note on the $14 \times 14$ intercalate-rich Latin squares}
We focus on two known $14 \times 14$ constructions for intercalate-rich
Latin squares, and find critical sets of large size in these Latin squares.

If we take the direct product of
the Latin square corresponding to the Steiner triple
system of order 7 and ${\mathbb Z}_2$, the result is a $14 \times 14$ Latin square
with 385 intercalates.  The construction shown earlier, $L_1$, 
with $m = 7$, gives a $14 \times 14$ Latin square with 343 intercalates.

Donovan's critical set construction \cite{MR99a:05018} for Latin
squares of order $2m$ results in a critical set of size 112.  However,
critical sets larger than this can be found.

The first example given immediately below, $M_1$, is of size 117, and
is from the Latin square with 385 intercalates.  It was obtained by
starting with the relevant Latin square, and for each $0 \leq i,j,k \leq
13$, removing row $i$, column $j$ and all occurrences of symbol $k$, to
arrive at a partial Latin square we shall denote $Y(i,j,k)$.  For each
partial Latin square $Y(i,j,k)$, the subsquare of rows 0, \dots, 6 and of
columns 0, \dots, 6 was fixed while all unnecessary entries elsewhere
were removed, and when this process terminated, all unnecessary entries
everywhere in the partial Latin square were removed.

\begin{center}
\begin{latinsq}[10]{14}
\hline   \ecell & 3 & \ecell & 5 & \ecell & 7 & 6 & \ecell & \ecell & 9 & 12 & 11 & 14 & \ecell \\
\hline   \ecell & \ecell & \ecell & \ecell & \ecell & \ecell & \ecell & \ecell & \ecell & \ecell & \ecell & \ecell & \ecell & \ecell \\
\hline   \ecell & 1 & 3 & 7 & 6 & 5 & 4 & 9 & \ecell & 10 & 14 & 13 & \ecell & 11 \\
\hline 5 & 6 & 7 & 4 & 1 & \ecell & 3 & 12 & \ecell & \ecell & 11 & 8 & 9 & 10 \\
\hline 4 & 7 & 6 & 1 & 5 & 3 & \ecell & 11 & \ecell & 13 & 8 & 12 & 10 & 9 \\
\hline 7 & 4 & 5 & \ecell & \ecell & \ecell & \ecell & 14 & \ecell & 12 & 9 & 10 & \ecell & 8 \\
\hline   \ecell & 5 & 4 & 3 & \ecell & 1 & \ecell & 13 & \ecell & \ecell & 10 & 9 & 8 & \ecell \\
\hline   \ecell & 10 & 9 & \ecell & 11 & 14 & 13 & 1 & \ecell & \ecell & 5 & 4 & 7 & 6 \\
\hline 10 & 9 & 8 & \ecell & 14 & 11 & 12 & 3 & \ecell & 1 & 6 & 7 & 4 & 5 \\
\hline   \ecell & \ecell & \ecell & \ecell & \ecell & \ecell & 11 & \ecell & \ecell & 3 & 7 & 6 & \ecell & 4 \\
\hline 12 & 13 & 14 & \ecell & 8 & \ecell & 10 & 5 & \ecell & 7 & 4 & 1 & \ecell & 3 \\
\hline   \ecell & 14 & 13 & \ecell & \ecell & \ecell & \ecell & 4 & \ecell & 6 & 1 & 5 & 3 & \ecell \\
\hline   \ecell & 11 & 12 & 9 & 10 & \ecell & \ecell & \ecell & \ecell & 5 & \ecell & 3 & 6 & \ecell \\
\hline   \ecell & 12 & \ecell & \ecell & 9 & 8 & \ecell & 6 & \ecell & 4 & 3 & \ecell & 1 & 7 \\
\hline
\end{latinsq} \vspace{5mm}
\\ $M_1$ \\
\end{center}

The second example, $M_2$, is from the Latin square $L_1$ of Section~\ref{sec51} with 
$m=7$.  We have
that $|M_2| = 118$. Three of the four subsquares in the union
which defines $L_1$ contain critical sets of size 23 which are pairwise
conjugate to each other.  This critical set was constructed by starting
with a list of critical sets of size 23 in the back-circulant Latin square
of order 7 and combining critical sets isomorphic to it in $A$ and $A^T$
together with a complete subsquare, $A$, $A^T$, or $B$.  This critical set
is of interest because it is in a similar pattern to Donovan's construction
for a critical set of size $\displaystyle{\frac{5m^2-3m}{2}}$ in a $2m
\times 2m$ Latin square, but it uses conjugate critical sets of size greater
than $\displaystyle{\frac{7^2-7}{2}}$ in three of the four subsquares.
Such critical sets have not been achieved before.

Thus, this example raises the possibility that a construction of a
critical set of size greater than $\displaystyle{\frac{n^2-n}{2}}$ in
the back-circulant Latin square of order $n$ could lead to generalized
constructions of size greater than $\displaystyle{\frac{5m^2-3m}{2}}$ in a
$2m \times 2m$ Latin square.

\vspace{5mm}

\begin{center}
\begin{latinsq}[10]{14}
\hline  \ecell & \ecell & \ecell & \ecell & \ecell & \ecell & \ecell &  8 &  9 & 10 & 11 & 12 & 13 & 14 \\ 
\hline   \ecell &  1 & \ecell &  6 &  5 &  4 &  3 &  9 & 10 & 11 & 12 & 13 & 14 &  8 \\ 
\hline   \ecell & \ecell &  1 &  7 &  6 &  5 &  4 & 10 & 11 & 12 & 13 & 14 &  8 &  9 \\ 
\hline   \ecell & \ecell & \ecell & \ecell & \ecell &  6 &  5 & 11 & 12 & 13 & 14 &  8 &  9 & 10 \\ 
\hline   \ecell &  4 & \ecell & \ecell &  1 & \ecell &  6 & 12 & 13 & 14 &  8 &  9 & 10 & 11 \\ 
\hline   \ecell &  5 & \ecell & \ecell & \ecell &  1 &  7 & 13 & 14 &  8 &  9 & 10 & 11 & 12 \\ 
\hline   \ecell &  6 &  5 &  4 &  3 & \ecell &  1 & 14 &  8 &  9 & 10 & 11 & 12 & 13 \\ 
\hline   \ecell & \ecell & \ecell & \ecell & \ecell & \ecell & \ecell & \ecell & \ecell & \ecell & \ecell & \ecell & \ecell & \ecell \\ 
\hline   \ecell & 10 & 11 & 12 & 13 & 14 &  8 & \ecell &  1 & \ecell &  6 &  5 &  4 &  3 \\ 
\hline   \ecell & 11 & 12 & 13 & \ecell &  8 & \ecell & \ecell & \ecell &  1 &  7 &  6 &  5 &  4 \\ 
\hline   \ecell & 12 & 13 & \ecell &  8 & \ecell & 10 & \ecell & \ecell & \ecell & \ecell & \ecell &  6 &  5 \\ 
\hline   \ecell & 13 & 14 & \ecell & \ecell & \ecell & 11 & \ecell &  4 & \ecell & \ecell &  1 & \ecell &  6 \\ 
\hline   \ecell & \ecell &  8 & \ecell & \ecell & \ecell & 12 & \ecell &  5 & \ecell & \ecell & \ecell &  1 &  7 \\ 
\hline   \ecell &  8 & \ecell & \ecell & 11 & 12 & 13 & \ecell &  6 &  5 &  4 &  3 & \ecell &  1 \\ 
\hline 
\end{latinsq} \vspace{5mm}
\\ $M_2$ \\
\end{center}

\section{Conclusion}
We have now given a new construction for Latin squares of order $4m$ 
which proves that $I(4m) \geq I(4) m^3$.  This leads to a
better bound on $I(2^\alpha m)$ for $\alpha \geq 2$, $m$ odd.  Also,
critical sets can be discovered in such squares of extraordinary size.
This discovery is more evidence for the conjecture 
that Latin squares containing many intercalates are closely
related to the largest critical sets in Latin squares of a given
order.

Since a transversal existed in this construction, it was prolonged to
give a new bound on $I(4m+1)$, which was generalised to a new bound on
$I(2^\alpha m+1)$ for $\alpha = 2$ or $\alpha \geq 4$.

Further research might include trying to discover a combination of
subsquares of the form $A, A^T$ and $B$ into a square which could prove a
new bound on $I(7m)$: $I(7m) \geq I(7) m^3$, that is, $I(7m) \geq 42 m^3$.
It would also be interesting to look at the maximum number of $m \times m$
subsquares, $m > 2$, for a given order of Latin square.
\chapter{Closing a gap in the spectrum of critical sets}\label{ch6}
\section{Introduction}
In 1998 Donovan and Howse proved that for all $n$ there exist critical
sets of order $n$ and size $s$, where $\lfloor \displaystyle{\frac{n^{2}}{4}} \rfloor
\leq s \leq \displaystyle{\frac{n^{2}-n}{2}}$ with the exception of the case $s =
\displaystyle{\frac{n^{2}}{4}} + 1$ when $n$ is even.  In this chapter we shall present a
construction for this exception, where $n \geq 6$.  It is based on the discovery of a
critical set of size 17 for a Latin square of order 8.  Thus Theorem~\ref{thm636} verifies
that there does exist a critical set of order $n$ and size $\displaystyle{\frac{n^{2}}{4}}$ + 1 when $n$
is even and $n \geq 6$.

\section{Critical sets in Latin squares of orders 6 and 8}
Recall that $BC_n$ denotes the back circulant Latin square $\{ (i,j;i+j) 
\mid 0 \leq i,j \leq n-1 \}$ where the addition $i+j$ is taken modulo $n$.

Let $\mathcal{A} = \{ (i, j; i+j) \mid (0 \leq i,j \leq 5) \wedge ((0 \leq i + j \leq 1) \vee 
(8 \leq i + j \leq 10)) \}.$
Then $\mathcal{A}$ is a critical set of order 6 and size $\displaystyle{\frac{6^{2}}{4}} = 9$ in $BC_{6}$.
Beginning with $\mathcal{A}$, we remove entry $(5,4;3)$ and add entries $(3,2;5)$ and 
$(3,4;3)$ and denote the new partial Latin square by $\mathcal{A'}$.
Programs developed from Algorithm~\ref{311} can be used to verify that $\mathcal{A'}$
is a critical set of size $\displaystyle{\frac{6^{2}}{4}} + 1 = 10$ which
completes to the Latin square $\mathcal{LA}$ as shown in Table~\ref{closeone}.

\begin{table}
\begin{center}
\caption{Critical sets and Latin squares of order 6}
\label{closeone}
\begin{tabular}{ccc}
\begin{latinsq}{6}
\hline 0 & 1 & \ecell & \ecell & \ecell & \ecell \\
\hline 1 & \ecell & \ecell & \ecell & \ecell & \ecell \\
\hline   \ecell & \ecell & \ecell & \ecell & \ecell & \ecell \\
\hline   \ecell & \ecell & \ecell & \ecell & \ecell & 2 \\
\hline   \ecell & \ecell & \ecell & \ecell & 2 & 3 \\
\hline   \ecell & \ecell & \ecell & 2 & 3 & 4 \\
\hline 
\end{latinsq} &
\begin{latinsq}{6}
\hline 0 & 1 & \ecell & \ecell & \ecell & \ecell \\
\hline 1 & \ecell & \ecell & \ecell & \ecell & \ecell \\
\hline   \ecell & \ecell & \ecell & \ecell & \ecell & \ecell \\
\hline   \ecell & \ecell & 5 & \ecell & 3 & 2 \\
\hline   \ecell & \ecell & \ecell & \ecell & 2 & 3 \\
\hline   \ecell & \ecell & \ecell & 2 & \ecell & 4 \\
\hline 
\end{latinsq} &
\begin{latinsq}{6}
\hline 0 & 1 & 2 & 3 & 4 & 5 \\
\hline 1 & 2 & 3 & 4 & 5 & 0 \\
\hline 2 & 3 & 4 & 5 & 0 & 1 \\
\hline 4 & 0 & 5 & 1 & 3 & 2 \\
\hline 5 & 4 & 1 & 0 & 2 & 3 \\
\hline 3 & 5 & 0 & 2 & 1 & 4 \\
\hline 
\end{latinsq} \\[3pt]
 $\mathcal{A}$ & $\mathcal{A'}$ & $\mathcal{LA}$
\end{tabular}
\end{center}
\end{table}

Let $\mathcal{B} = \{ (i, j; i+j) \mid (0 \leq i,j \leq 7) \wedge ((0 \leq i + j \leq 2) \vee 
(11 \leq i + j \leq 14)) \}.$
Then $\mathcal{B}$ is a critical set of order 8 and size $\displaystyle{\frac{8^{2}}{4}} = 16$ in $BC_{8}$.
Beginning with $\mathcal{B}$, we remove entries $(7,5;4)$ and $(7,6;5)$ and add entries $(4,3;7)$, $(4,5;4)$,
and $(4,6;5)$, and denote the new partial Latin square by $\mathcal{B'}$.

Again, programs developed from Algorithm~\ref{311} can be used to verify that $\mathcal{B'}$ is a critical set of size $\displaystyle{\frac{8^{2}}{4}} + 1 = 17$ and completes to the Latin square $\mathcal{LB}$ as shown in Table~\ref{closetwo}.

\begin{table}
\begin{scriptsize}
\begin{center}
\caption{Critical sets and Latin squares of order 8}
\label{closetwo}
\begin{tabular}{@{}c@{\hspace{5pt}}c@{\hspace{5pt}}c@{}}
\begin{latinsq}{8}
\hline 0 & 1 & 2 & \ecell & \ecell & \ecell & \ecell & \ecell \\
\hline 1 & 2 & \ecell & \ecell & \ecell & \ecell & \ecell & \ecell \\
\hline 2 & \ecell & \ecell & \ecell & \ecell & \ecell & \ecell & \ecell \\
\hline   \ecell & \ecell & \ecell & \ecell & \ecell & \ecell & \ecell & \ecell \\
\hline   \ecell & \ecell & \ecell & \ecell & \ecell & \ecell & \ecell & 3 \\
\hline   \ecell & \ecell & \ecell & \ecell & \ecell & \ecell & 3 & 4 \\
\hline   \ecell & \ecell & \ecell & \ecell & \ecell & 3 & 4 & 5 \\
\hline   \ecell & \ecell & \ecell & \ecell & 3 & 4 & 5 & 6 \\
\hline 
\end{latinsq} &
\begin{latinsq}{8}
\hline 0 & 1 & 2 & \ecell & \ecell & \ecell & \ecell & \ecell \\
\hline 1 & 2 & \ecell & \ecell & \ecell & \ecell & \ecell & \ecell \\
\hline 2 & \ecell & \ecell & \ecell & \ecell & \ecell & \ecell & \ecell \\
\hline   \ecell & \ecell & \ecell & \ecell & \ecell & \ecell & \ecell & \ecell \\
\hline   \ecell & \ecell & \ecell & 7 & \ecell & 4 & 5 & 3 \\
\hline   \ecell & \ecell & \ecell & \ecell & \ecell & \ecell & 3 & 4 \\
\hline   \ecell & \ecell & \ecell & \ecell & \ecell & 3 & 4 & 5 \\
\hline   \ecell & \ecell & \ecell & \ecell & 3 & \ecell & \ecell & 6 \\
\hline 
\end{latinsq} &
\begin{latinsq}{8}
\hline 0 & 1 & 2 & 3 & 4 & 5 & 6 & 7 \\
\hline 1 & 2 & 3 & 4 & 5 & 6 & 7 & 0 \\
\hline 2 & 3 & 4 & 5 & 6 & 7 & 0 & 1 \\
\hline 3 & 4 & 5 & 6 & 7 & 0 & 1 & 2 \\
\hline 6 & 0 & 1 & 7 & 2 & 4 & 5 & 3 \\
\hline 5 & 7 & 6 & 1 & 0 & 2 & 3 & 4 \\
\hline 7 & 6 & 0 & 2 & 1 & 3 & 4 & 5 \\
\hline 4 & 5 & 7 & 0 & 3 & 1 & 2 & 6 \\
\hline \end{latinsq} \\[3pt]
 $\mathcal{B}$ & $\mathcal{B'}$ & $\mathcal{LB}$
\end{tabular}
\end{center}
\end{scriptsize}
\end{table}

Therefore $\mathcal{A'}$ and $\mathcal{B'}$ demonstrate that critical
sets of order $n$ and size $\displaystyle{\frac{n^2}{4}} + 1$ exist when $n = 6$
and $n = 8$ respectively.

\section{Critical sets in Latin squares of order $n$, $n$ even}
The above examples can be generalised to produce critical sets of
size $\displaystyle{\frac{n^{2}}{4}} + 1$, when $n$ is even.  

\begin{theorem}
\label{thm636}
Take the critical set
\begin{eqnarray*}
C & = & \{ (i, j; i+j ) \mid (0 \leq i + j \leq \displaystyle{\frac{n}{2}} - 2) \vee 
(\displaystyle{\frac{3n}{2}} - 1 \leq i + j \leq 2n - 2) \}.
\end{eqnarray*}
Construct the set
\begin{eqnarray*}
D & = & ( C
\backslash 
\{ ( n - 1, j; j - 1 )  \mid \displaystyle{\frac{n}{2}} + 1 \leq j 
\leq n - 2 \} ) \\ 
& & \quad{}\cup \{ ( \displaystyle{\frac{n}{2}}, j; j - 1)  \mid \displaystyle{\frac{n}{2}} + 1 
\leq j \leq n - 2 \} \cup 
\{ (\displaystyle{\frac{n}{2}}, \displaystyle{\frac{n}{2}} - 1; n - 1) \}.
\end{eqnarray*}
Then $D$ is a critical set of size $\displaystyle{\frac{n^{2}}{4}} + 1$.  
\end{theorem}


\begin{proof}{
Henceforth, we shall refer to the completion of $D$ as $\mathcal{LD}$.
The following process outlines how $D$ can be uniquely completed to $\mathcal{LD}$.
In completing $D$, at each step in the completion process
the given cell is forced to contain the specified symbol.  If any other
symbol were to be placed in the specified cell, the result
would not be a partial Latin square.

\vskip 2mm
\noindent We begin by filling row $\displaystyle{\frac{n}{2}}$ starting at column $j=0$ and moving right to column $j=\displaystyle{\frac{n}{2}} - 2$.
In row $\displaystyle{\frac{n}{2}}$, fill the cell in column $j$ with:
\begin{itemize}
\item[{}] $n - 2$, when $j=0$;
\item[{}] $j - 1$, when $1 \leq j \leq \displaystyle{\frac{n}{2}} - 2$;
\item[{}] $\displaystyle{\frac{n}{2}} - 2$, when $j = \displaystyle{\frac{n}{2}}$.
\end{itemize}

\noindent We shall fill rows $n-2$ to $\displaystyle{\frac{n}{2}}+1$ sequentially, from
left to right in columns 0 to $\displaystyle{\frac{n}{2}} - 2$,
then column $\displaystyle{\frac{n}{2}}$,
then column $\displaystyle{\frac{n}{2}} - 1$.  So, for $2 \leq x \leq \displaystyle{\frac{n}{2}} - 1$, and $0 \leq j \leq \displaystyle{\frac{n}{2}}$
fill the cell in row $n-x$ and column $j$ with:
\begin{itemize}
\item[{}] ($n-x)+j$ (mod $n$), when $j \neq x - 1$ and $j \neq x - 2$;
\item[{}] $n-1$, when $j = x - 2$;
\item[{}] $n-2$, when $j = x - 1$;
\item[{}] $\displaystyle{\frac{n}{2}} - 1 - x$, when $j = \displaystyle{\frac{n}{2}}$;
\item[{}] $\displaystyle{\frac{n}{2}}-x$, when $j = \displaystyle{\frac{n}{2}} - 1$.
\end{itemize}
\noindent When $n \geq 8$,
the triangle bounded by the cells $(\displaystyle{\frac{n}{2}}+1$, $\displaystyle{\frac{n}{2}}+1)$,
$(\displaystyle{\frac{n}{2}}+1, n-3)$, and $(n-3, \displaystyle{\frac{n}{2}}+1)$ is filled from
bottom to top and left to right.
\noindent If $n \geq 8$, for $3 \leq x \leq \displaystyle{\frac{n}{2}}-1$ fill the cell in row $n-x$,
column $j=\displaystyle{\frac{n}{2}}+1$ to $j=\displaystyle{\frac{n}{2}}+x-2$
with $(n-x)+j$ (mod $n$). \\
\noindent For $0 \leq j \leq \displaystyle{\frac{n}{2}} - 3$, fill the cell in row $n-1$ and column $j$ with
$\displaystyle{\frac{n}{2}}+j$ (mod $n$).  Fill the cell in row $n-1$ and column $j$ with
\begin{itemize} 
\item[{}] $n-1$, when $j = \displaystyle{\frac{n}{2}} - 2$ and
\item[{}] 0, when $j = \displaystyle{\frac{n}{2}} - 1$.
\end{itemize}

\noindent For $\displaystyle{\frac{n}{2}} + 1 \leq j \leq n - 2$, fill the cell in row $n-1$ and column $j$ with $j-\displaystyle{\frac{n}{2}}$ (mod $n$).

\noindent For $0 \leq x \leq \displaystyle{\frac{n}{2}} - 1$, fill the cells in row $x$ sequentially
right to left from column $j = n - 1$ to $j = \displaystyle{\frac{n}{2}} - 1 - x$ with $x+j$.

To prove the necessity of each of the symbols in the critical set $D$, 
three varieties of Latin interchanges will be used:

\vspace{.2cm}

\noindent {\bf Variety 1} \\
This Latin interchange uses only two rows and consequently
the same symbols in each row.
The disjoint mate is obtained by interchanging the rows.
For example, the Latin interchange $I$ and its disjoint mate $I'$ can be
represented as:
\begin{eqnarray*}
I & = & \{ ( r_{1}, c_{1}; e_{1} ), (r_{1}, c_{2}; e_{2} ), ..., (r_{1}, c_{m-1}; e_{m-1}), (r_{1}, c_{m}; e_{m} ) \} \\
&& \quad{}\cup \{ ( r_{2}, c_{1}; e_{2} ), (r_{2}, c_{2}; e_{3} ), ..., (r_{2}, c_{m-1}; e_{m}), (r_{2}, c_{m}; e_{1} ) \}, \mbox{ and} \\
I' & = & \{ ( r_{1}, c_{1}; e_{2} ), (r_{1}, c_{2}; e_{3} ), ... (r_{1}, c_{m-1}; e_{m}), (r_{1}, c_{m}; e_{1} ) \} \\
&& \quad{}\cup \{ ( r_{2}, c_{1}; e_{1} ), (r_{2}, c_{2}; e_{2} ), ..., (r_{2}, c_{m-1}; e_{m-1}), (r_{2}, c_{m}; e_{m}) \}.
\end{eqnarray*}

\noindent {\bf Variety 2} \\
This Latin interchange uses three rows, with the top row containing 
two entries.
For example, the Latin interchange $I$ and its disjoint mate $I'$ 
can be represented as: 
\begin{eqnarray*}
I & = & \{ (r_{1}, c_{1}; x ), (r_{1}, c_{m+1}; y) \} \\
&& \quad{}\cup \{ ( r_{2}, c_{1}; y ), (r_{2}, c_{2}; e_{1} ), (r_{2}, c_{3}; e_{2} ),  ..., (r_{2}, c_{m}; e_{m-1}), (r_{2}, c_{m+1}; e_{m} ) \} \\
&& \quad{}\cup \{ ( r_{3}, c_{1}; e_{1} ), (r_{3}, c_{2}; e_{2} ), (r_{3}, c_{3}; e_{3} ), ..., (r_{3}, c_{m}; e_{m}), (r_{3}, c_{m+1}; x ) \}, \mbox{ and} \\
I' & = & \{ (r_{1}, c_{1}; y ), (r_{1}, c_{m+1}; x) \} \\
&& \quad{}\cup \{ ( r_{2}, c_{1}; e_{1} ), (r_{2}, c_{2}; e_{2} ), (r_{2}, c_{3}; e_{3} ),  ..., (r_{2}, c_{m}; e_{m}), (r_{2}, c_{m+1}; y ) \} \\
&& \quad{}\cup \{ ( r_{3}, c_{1}; x ), (r_{3}, c_{2}; e_{1} ), (r_{3}, c_{3}, e_{2} ), ..., (r_{3}, c_{m}; e_{m-1}), (r_{3}, c_{m+1}; e_{m} ) \}.
\end{eqnarray*}

\noindent {\bf Variety 3} \\
The third variety of Latin interchanges take a number of forms and cannot be
written as simply as Variety 1 or Variety 2.  Full details of
these Latin interchanges are presented in Appendix~\ref{app3}.
\newline
\vskip 2mm 
\noindent For $n = 6$, proving that the entries in the example given above
are necessary can be verified using programs developed from Algorithm~\ref{311}.
We assume $n \geq 8$ and prove the following.
Latin interchanges $I_{1}$ through $I_{10}$, below, exist in $\mathcal{LD}$: 
\newline
\vskip 2mm
\noindent $I_{1}$ is a Latin interchange of Variety 1, and $I_{1} \cap D = \{ ( \displaystyle{\frac{n}{2}}, \displaystyle{\frac{n}{2}} - 1; n - 1) \}.$
\begin{eqnarray*}
I_{1} &= & \{ (\displaystyle{\frac{n}{2}}, 0; n - 2) \} \\
&& \quad{}\cup \{ (\displaystyle{\frac{n}{2}}, j; j-1) \mid 1 \leq j \leq \displaystyle{\frac{n}{2}} - 2 \} \\
&& \quad{}\cup \{ ( \displaystyle{\frac{n}{2}}, \displaystyle{\frac{n}{2}} - 1; n - 1 ), ( \displaystyle{\frac{n}{2}}, \displaystyle{\frac{n}{2}}; \displaystyle{\frac{n}{2}} - 2) \} \\
&& \quad{}\cup \{ (n - 2, 0; n - 1), (n - 2, 1; n - 2) \\ 
&& \quad{}\cup \{ ( n - 2, j; j - 2 ) \mid 2 \leq j \leq \displaystyle{\frac{n}{2}} - 2 \} \\
&& \quad{}\cup \{ ( n - 2, \displaystyle{\frac{n}{2}} - 1; \displaystyle{\frac{n}{2}} - 2 ), ( n - 2, \displaystyle{\frac{n}{2}};  \displaystyle{\frac{n}{2}}- 3) \}.
\end{eqnarray*}
\noindent $I_{2}$ is a Latin interchange of Variety 1, and $I_{2} \cap D = \{ (n - 1, n - 1; n - 2) \}$.
\begin{eqnarray*}
I_{2} & = & \{ ( \displaystyle{\frac{n}{2}} - 1, \displaystyle{\frac{n}{2}} - 1; n - 2 )  \} \\ 
&& \quad{}\cup \{ (\displaystyle{\frac{n}{2}} - 1, j; \displaystyle{\frac{n}{2}} + j - 1) | \displaystyle{\frac{n}{2}} + 1 \leq j \leq n - 1 \} \\ 
&& \quad{}\cup \{ ( n - 1, \displaystyle{\frac{n}{2}} - 1; 0 ) \} \\
&& \quad{}\cup \{ ( n - 1, j; j - \displaystyle{\frac{n}{2}} ) \mid \displaystyle{\frac{n}{2}} + 1 \leq j \leq n - 2 \} \\ 
&& \quad{}\cup \{ ( n - 1, n - 1; n - 2 ) \}.
\end{eqnarray*}

\noindent $I_{3}$ is a Latin interchange of Variety 1, and $I_{3} \cap D = \{ ( n - 1, \displaystyle{\frac{n}{2}}; \displaystyle{\frac{n}{2}} - 1 ) \}$.
\begin{eqnarray*}
I_{3} & = & \{ ( \displaystyle{\frac{n}{2}} - 1, j; \displaystyle{\frac{n}{2}} + j - 1) \mid 0 \leq j \leq \displaystyle{\frac{n}{2}} - 2 \} \\
&& \quad{}\cup \{ ( \displaystyle{\frac{n}{2}} - 1, \displaystyle{\frac{n}{2}}; n - 1 ) \} \\
&& \quad{}\cup \{ (n - 1, j; j + \displaystyle{\frac{n}{2}} ) \mid 0 \leq j \leq \displaystyle{\frac{n}{2}} - 3\} \\
&& \quad{}\cup \{ (n - 1, \displaystyle{\frac{n}{2}} - 2; n - 1 ), (n - 1, \displaystyle{\frac{n}{2}}; \displaystyle{\frac{n}{2}} - 1 ) \}.
\end{eqnarray*}

\noindent For $\displaystyle{\frac{n}{2}} + 2 \leq x \leq n - 2$,
$I_{4}$ is a Latin interchange of Variety 2, and 
$I_{4} \cap D = \{ (x, \displaystyle{\frac{3n}{2}} - 1 - x; \displaystyle{\frac{n}{2}} - 1) \}$.

For $\displaystyle{\frac{n}{2}} + 2 \leq x \leq n - 2$, construct the Latin interchange
\begin{eqnarray*}
H & = &  \{ (x - \displaystyle{\frac{n}{2}} - 1, n - x; \displaystyle{\frac{n}{2}} - 1), (x - \displaystyle{\frac{n}{2}} - 1, \displaystyle{\frac{3n}{2}} - 1 - x; n - 2) \} \\
&& \quad{}\cup \{ (x - 1, j; x - 1 + j) \mid \displaystyle{\frac{n}{2}} + 1 \leq j \leq \displaystyle{\frac{3n}{2}} - 1 - x \} \\
&& \quad{}\cup \{ (x - 1, n - x; n - 2) \} \\
&& \quad{}\cup \{ (x - 1, \displaystyle{\frac{n}{2}}-1; x - \displaystyle{\frac{n}{2}} - 1), (x - 1, \displaystyle{\frac{n}{2}}; x - \displaystyle{\frac{n}{2}} - 2) \} \\
&& \quad{}\cup \{ (x, j; x + j) \mid \displaystyle{\frac{n}{2}} + 1 \leq j \leq \displaystyle{\frac{3n}{2}} - 1 - x \} \\ 
&& \quad{}\cup \{ (x, j; x + j) \mid n - x \leq j \leq \displaystyle{\frac{n}{2}} - 2 \} \\ 
&& \quad{}\cup \{ (x, \displaystyle{\frac{n}{2}} - 1; x + \displaystyle{\frac{n}{2}}), (x, \displaystyle{\frac{n}{2}}; x + \displaystyle{\frac{n}{2}} - 1) \}.
\end{eqnarray*}

Then when $x = \displaystyle{\frac{n}{2}} + 2$, let $I_{4} = H$, and when $\displaystyle{\frac{n}{2}} + 3
\leq x \leq n - 2$, let $I_{4} = H \cup \{ (x - 1, i; x - 1 + i) \mid n -
x + 1 \leq i \leq \displaystyle{\frac{n}{2}} - 2 \}$.

\noindent For $\displaystyle{\frac{n}{2}} + 1 \leq x \leq n - 2$, 
$I_{5}$ is a Latin interchange of Variety 2, and 
$I_{5} \cap D$ = $\{(x, n-1; x-1)\}$.

\begin{eqnarray*}
I_{5} & = & \{ (x - \displaystyle{\frac{n}{2}}, j; x - \displaystyle{\frac{n}{2}} + j) \mid \displaystyle{\frac{n}{2}} - 1 \leq j \leq n-1 \} \\
&& \quad{}\cup \{ (x - \displaystyle{\frac{n}{2}} + 1, j; x- \displaystyle{\frac{n}{2}} + 1 + j) \mid \displaystyle{\frac{n}{2}} - 1 \leq j \leq n - 1 \} \\
&& \quad{}\cup \{ (x, \displaystyle{\frac{n}{2}} - 1; x - \displaystyle{\frac{n}{2}}), (x, n - 1; x - 1) \}. 
\end{eqnarray*}

\noindent $I_{6}$ is a Latin interchange of Variety 1, and $I_{6} \cap D = \{ (\displaystyle{\frac{n}{2}}+1,n-2;\displaystyle{\frac{n}{2}} - 1) \}$.

If $4 \mid n$, construct the Latin interchange
\begin{eqnarray*}
I_{6} & = & \{ (\displaystyle{\frac{n}{2}} - 1, 2j; \displaystyle{\frac{n}{2}} - 1 + 2j) \mid 0 \leq j < \displaystyle{\frac{n}{4}} \} \\
&& \quad{}\cup \{ (\displaystyle{\frac{n}{2}} - 1, \displaystyle{\frac{n}{2}} - 1; n - 2) \} \cup \\
&& \quad{}\cup \{ (\displaystyle{\frac{n}{2}} - 1, 2j; \displaystyle{\frac{n}{2}} - 1 + 2j) \mid \displaystyle{\frac{n}{4}} < j < \displaystyle{\frac{n}{2}} \} \\
&& \quad{}\cup \{ (\displaystyle{\frac{n}{2}} + 1, 2j; \displaystyle{\frac{n}{2}} + 1 + 2j) \mid 0 \leq j < \displaystyle{\frac{n}{4}}-1 \} \\
&& \quad{}\cup \{ (\displaystyle{\frac{n}{2}} + 1, \displaystyle{\frac{n}{2}} - 2; n - 2), (\displaystyle{\frac{n}{2}} + 1, \displaystyle{\frac{n}{2}} - 1; 1) \} \\
&& \quad{}\cup \{ (\displaystyle{\frac{n}{2}} + 1, 2j; \displaystyle{\frac{n}{2}} + 1 + 2j) \mid \displaystyle{\frac{n}{4}} < j < \displaystyle{\frac{n}{2}} \}.
\end{eqnarray*}

If $4 \nmid n$, construct the Latin interchange
\begin{eqnarray*}
I_{6} & = & \{ (\displaystyle{\frac{n}{2}}-1,2j;\displaystyle{\frac{n}{2}}-1+2j) \mid 0 \leq j < \displaystyle{\frac{n}{4}} - 1 \} \\
&& \quad{}\cup \{ (\displaystyle{\frac{n}{2}}-1,\displaystyle{\frac{n}{2}}; n-1) \} \\
&& \quad{}\cup \{ (\displaystyle{\frac{n}{2}}-1,2j;\displaystyle{\frac{n}{2}}-1+2j) \mid \displaystyle{\frac{n}{4}} < j < \displaystyle{\frac{n}{2}} \} \\
&& \quad{}\cup \{ (\displaystyle{\frac{n}{2}}+1,2j;\displaystyle{\frac{n}{2}}+1+2j) \mid 0 \leq j < \displaystyle{\frac{n}{4}} - 2 \} \\
&& \quad{}\cup \{ (\displaystyle{\frac{n}{2}}+1,\displaystyle{\frac{n}{2}}-3; n-1), (\displaystyle{\frac{n}{2}}+1,\displaystyle{\frac{n}{2}}; 0) \} \\
&& \quad{}\cup \{ (\displaystyle{\frac{n}{2}}+1,2j;\displaystyle{\frac{n}{2}}+1+2j) \mid \displaystyle{\frac{n}{4}} < j < \displaystyle{\frac{n}{2}} \}.
\end{eqnarray*}

\noindent $I_{7}$ is a Latin interchange of Variety 1, and $I_{7} \cap D = \{
(\displaystyle{\frac{n}{2}}, n - 1; \displaystyle{\frac{n}{2}} - 1) \}$.

If $4 \mid n$, construct the Latin interchange
\begin{eqnarray*}
I_{7} & = & \{ (\displaystyle{\frac{n}{4}} - 1, j; \displaystyle{\frac{n}{4}} + j - 1), (\displaystyle{\frac{n}{4}}, j; \displaystyle{\frac{n}{4}} + j) \mid \displaystyle{\frac{n}{4}} \leq j \leq n -1 \} \\
&& \quad{}\cup \{ (\displaystyle{\frac{n}{2}}, \displaystyle{\frac{n}{4}}; \displaystyle{\frac{n}{4}} - 1),  (\displaystyle{\frac{n}{2}}, n-1; \displaystyle{\frac{n}{2}} - 1) \}.
\end{eqnarray*}

If $4 \nmid n$, construct the Latin interchange
\begin{eqnarray*}
I_{7} & = & \{ (\displaystyle{\frac{n-2}{4}}, \displaystyle{\frac{n-2}{4}}; \displaystyle{\frac{n}{2}}-1), (\displaystyle{\frac{n-2}{4}}, n-1; \displaystyle{\frac{n-6}{4}}) \} \\
&& \quad{}\cup \{ (\displaystyle{\frac{n}{2}}, \displaystyle{\frac{n-2}{4}}; \displaystyle{\frac{n-6}{4}}),  ( \displaystyle{\frac{n}{2}}, n-1; \displaystyle{\frac{n}{2}} - 1) \}.
\end{eqnarray*}

\noindent For $\displaystyle{\frac{n}{2}} + 1 \leq x \leq n - 2$,
$I_{8}$ is a Latin interchange of Variety 1, and $I_{8} \cap D$ = \(
\{ ( \displaystyle{\frac{n}{2}}, x; x - 1) \}$.
\begin{eqnarray*}
I_{8} & = & \{ (\displaystyle{\frac{n}{2}} - 1, x - \displaystyle{\frac{n}{2}};  x - 1), (\displaystyle{\frac{n}{2}} - 1, x; \displaystyle{\frac{n}{2}} + x - 1) \} \\
&& \quad{}\cup \{ (\displaystyle{\frac{n}{2}}, x - \displaystyle{\frac{n}{2}}; \displaystyle{\frac{n}{2}} + x - 1),
(\displaystyle{\frac{n}{2}}, x; x-1 ) \}.
\end{eqnarray*}

\noindent For ($\displaystyle{\frac{3n}{2}} \leq x + y < 2n - 2) \wedge (x \neq n - 1) \wedge 
(y \neq n - 1)$, $I_{9}$ is a Latin interchange of Variety 1, and $I_{9} \cap D 
= \{ (y, x; y + x) \}$.  
\begin{eqnarray*}
I_{9} & = & \{ (y - \displaystyle{\frac{n}{2}}, x - \displaystyle{\frac{n}{2}}; y + x), (y - \displaystyle{\frac{n}{2}}, x; y + x - \displaystyle{\frac{n}{2}}) \} \\
&& \quad{}\cup \{ (y, x - \displaystyle{\frac{n}{2}}; y + x - \displaystyle{\frac{n}{2}}), (y, x; y+x) \}.
\end{eqnarray*}

Where $0 \leq x + y \leq \displaystyle{\frac{n}{2}} - 2$, there exists a Latin interchange $I_{10}$ of Variety 3, with $I_{10} \cap D = \{ (y, x; y+x) \}$.  

If $0 \leq x + y \leq \displaystyle{\frac{n}{2}} - 2$, determine the Latin interchange $I_{10}$ using
results found in \cite{MR1758263}.  See Appendix~\ref{app3} for details
on how to construct this interchange, which is referred to as $I$ therein.
}\end{proof}

Thus we have succeeded in proving the existence of a critical set
of order $n$ and size $\displaystyle{\frac{n^2}{4}}+1$ when $n$ is even, 
and $n \geq 6$, a problem which
has been open since 1977.  This completes the spectrum of critical
sets between the bounds Nelder conjectured in~\cite{nel2}, which
were $\displaystyle{\frac{n^2}{4}}$ and $\displaystyle{\frac{n^2-n}{2}}$ for the sizes of
the smallest and largest critical sets respectively in a Latin square of
order $n$.

\chapter{Steiner trades and Latin interchanges}\label{ch7}

Can our knowledge of the interchangeable sets in Latin squares (Latin interchanges) be used to classify the interchangeable sets in block designs (trades)?  It is this interesting question which we focus on here.   
In  Section~\ref{sec71} we detail the connection between Latin interchanges and Steiner trades.
In Section~\ref{sec72} we take all Steiner trades of volume less than or equal to nine and classify them according to the structure of the associated Latin interchanges.
A slight modification in the definition of ``Latin interchange'' will be
required for this chapter.  Here, we take a Latin interchange to represent
both the partial Latin square and its disjoint mate, instead of just
the partial Latin square.

\section{The connection between trades and Latin interchanges}\label{sec71}
\begin{lemma}
Let ${\mathcal T}=(T,T')$ be a $2$-$(v,3)$ Steiner trade based on the set $V$. Then the partial Steiner  Latin 
squares $I$ and $I'$ corresponding to $T$ and $T'$, respectively, 
form a Latin interchange and its disjoint mate denoted ${\mathcal I}=(I,I')$.
\end{lemma}

\begin{proof}{
 Note that $|T|=|T'|$ and $T\cap T'=\emptyset$;
hence $I$ and $I'$ have the same volume and shape and are disjoint. 
Next, assume that the rows of $I$ and $I'$ are not mutually balanced. That is, for some row $r$ there exists a column $j$ such that $(r,j;z)\in I$, but for the same row $r$, $(r,j';z)\notin I'$ for any column $j'$. Correspondingly the triple $\{r,j,z\}\in T$  for some $j\in V$, but $\{r,j',z\}\notin T'$  for any $j'\in V$, which is a contradiction as ${\mathcal T}=(T,T')$ is a trade. We may obtain a similar contradiction for the columns and so deduce that the rows and columns of $I$ and $I'$ are mutually balanced. Consequently ${\mathcal I}=(I,I')$ constitutes a Latin interchange and its disjoint mate as required.
}
\end{proof}

In \cite{dks2} Donovan, Khodkar and Street showed that for the given 
trade ${\mathcal T}=(T,T')$, where 
$T=\{1 2 3,1 4 5,1 6 7,2 4 8,3 6 8, 5 7 8\}$ and 
$T'=\{1 2 4,1 3 6,1 5 7,2 3 8,
4 5 8,
6 7 8\}$, 
the partial Steiner Latin squares associated with 
triples of ${\mathcal T}$ can be 
decomposed into six disjoint Latin  interchanges, 
denoted ${\mathcal I}_i=(I_i,I_i')$ for $i=1,\dots, 6$, in such a way that 
for each $i$ there is a one-to-one correspondence between the 
entries of $I_i$ and the triples of $T$. 
Further, they showed that no such decomposition exists for the Latin 
interchange associated with the trade  
${\mathcal T}=(T,T')$, where 
$T=\{1 2 3,1 4 5,1 6 7,2 4 7,3 4 6,3 5 7\}$ and 
$T'=\{1 2 4,1 3 6,1 5 7,2 3 7,3 4 5,4 6 7\}$. 
These results raise the following question:
\vspace{.2cm}

\begin{question}\label{714}
{
For which trades ${\mathcal T}=(T,T')$ can the corresponding 
Latin interchange, denoted ${\mathcal I}=(I,I')$, be decomposed into six disjoint 
Latin interchanges, denoted ${\mathcal I}_i=(I_i,I_i')$,  $1\leq i\leq 6$, 
such that for each $i=1,\dots, 6$ there is a one-to-one correspondence 
between the  triples of $T$ ($T'$) and the entries of $I_i$ ($I_i'$)
which maps $\{x,y,z\}\in T$ to $(x,y;z)\in I_i$?
}
\end{question}

In this chapter we give some partial answers to this question and, in addition, 
give an exact answer for all  Steiner trades with block size three and  
volume less than or equal to nine. 
Our list of trades of volume less than or equal to nine has been taken 
from \cite{MR2000g:05030} where Khosrovshahi and Maimani completely classified 
all Steiner trades with block size three and volume six to  nine.  

\section{Partial Answers}\label{sec72}
We begin by stating a result which identifies some Steiner trades 
whose corresponding partial Steiner  Latin squares 
can be decomposed into six disjoint 
Latin interchanges.

Let ${\mathcal T}=(T,T')$ be a trade. Recall that ${\mathcal T}$ is a
{\it minimal} trade if there is no $B$ with $\emptyset\neq B\subset T$ and
$B'$ with $\emptyset\neq B'\subset T'$ such that $(B,B')$
is a trade.  Also, the {\em foundation} of ${\mathcal T}$ is $F({\mathcal
T})=\{x\mid x \mbox{ is contained in a triple of } T\}$.

\begin{lemma}
\label{lemone}
Let ${\mathcal T}=(T,T')$ be a Steiner minimal trade based on the set $V$. 
For each element $x\in F({\mathcal T})$ suppose there exists a subset 
$S_x$ of $F({\mathcal T})$ such that $x\in S_x$ and so that each triple of 
$T$ intersects the set $S_x$ in precisely one element. 
Then the Latin interchanges corresponding  to ${\mathcal T}=(T,T')$,
denoted ${\mathcal I}=(I,I')$,  can be decomposed into six disjoint 
Latin interchanges. 
\end{lemma}

\begin{proof}{
First we prove that for $x,y\in F({\mathcal T})$ we
have either $S_x=S_y$ or $S_x\cap S_y=\emptyset$. Let $S_x\neq S_y$ and
$S_x\cap S_y\neq \emptyset$, as displayed in Figure~\ref{figure71}. Define 
\[\begin{array}{lcl}
T_1 & = & \{\{a,b,c\}\in T\mid \;a\in S_x\setminus S_y,\;b\in S_y\setminus
S_x,\; c\in F({\mathcal T})\setminus (S_x\cup S_y)\}, \\
T_2 & = & \{\{d,e,f\}\in T\mid \;d\in S_x\cap S_y,\;
e,f\in F({\mathcal T})\setminus (S_x\cup S_y)\}, \\ 
T_1' & = & \{\{a',b',c'\}\in T'\mid \;a'\in S_x\setminus S_y,\;
b'\in S_y\setminus S_x,\; c'\in F({\mathcal T})
\setminus (S_x\cup S_y)\}\;\; {\rm and} \\ 
T_2' & = & \{\{d',e',f'\}\in T'\mid \;d'\in S_x\cap S_y,\;
e',f'\in F({\mathcal T})\setminus (S_x\cup S_y)\}. \\
\end{array} \]
We note that if the pair $\{a,b\}$ occurs in a triple of $T$ then 
$a$ and $b$ cannot both be in $S_z$ for any $z\in {\mathcal T}$. This
leads to $T=T_1\cup T_2$ and $T'=T_1'\cup T_2'$. Now if the 
pair $\{a,b\}$ is in a triple of $T_1$ then $\{a,b\}$ is in a triple of
$T_1'$. So $(T_1,T_1')$ is a Steiner trade. This is a contradiction
since ${\mathcal T}=(T,T')$ is minimal. 
Hence either $S_x=S_y$ or $S_x\cap S_y=\emptyset$ for $x,y\in
F({\mathcal T})$. 
}
\end{proof}

\begin{figure}
\centering
\includegraphics[alt={Line drawing of a trade: within a dashed oval boundary, small circular vertices are connected by straight line segments forming several triangles; two large teardrop-shaped curved regions overlap in the centre, with labels $S_{x}$ (left), $S_{y}$ (right), and $F(\mathcal{T})$ near the top.}]{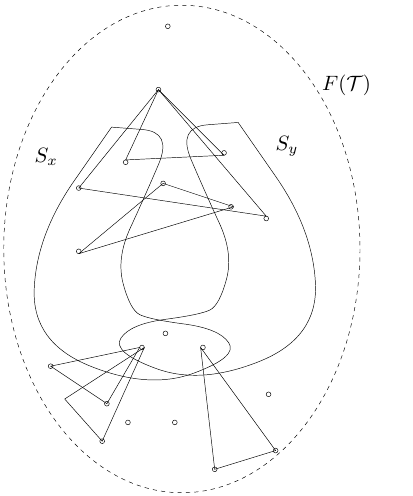}
\caption{A trade illustrating Lemma~\ref{lemone}}
\label{figure71}
\end{figure}

Now let the triple $\{a,b,c\}$ be in $T$; then 
$S_a\cup S_b\cup S_c=F({\mathcal T})$, 
$S_a\cap S_b=\emptyset$,
$S_a\cap S_c=\emptyset$ and
$S_b\cap S_c=\emptyset$. We define 
\[I_1=\{(x,y;z)\mid \; \{x,y,z\}\in T,\; x\in S_a,\;y\in S_b,\;z\in S_c\}. \]
It is easy to see that there is a one-to-one correspondence between the
entries of $I_1$ and the triples of $T$ which maps $(x,y;z) \in I_1$ to 
$\{x,y,z\} \in T$. Moreover, 
$I_1$ forms one of the Latin interchanges into which we are decomposing 
the partial Latin square associated with $T$. 
Now the six conjugates of $I_1$
decompose $I$ into six disjoint Latin interchanges, where
${\mathcal I}=(I,I')$ constitutes the Latin interchange and its disjoint mate
corresponding to 
${\mathcal T}=(T,T')$. 
\vspace{5mm}

However, the above condition of Lemma~\ref{lemone} is not necessary as 
is shown by the following example. The partial Steiner  Latin squares
corresponding to the trade ${\mathcal T}=(T,T')$ 
where $T=\{136,148,159,239,246,257,$ $347,358\}$ and
$T'=\{139,158,146,259,
236,247,357,348\}$ can be decomposed into six disjoint 
Latin interchanges. These may be obtained by taking the conjugates of the 
Latin interchange ${\mathcal I}_1=(I_1,I_1')$, where 
$I_1=\{(1,3;6),(1,4;8),(1,5;9),(2,3;9),$
$(2,4;6),(2,5;7),(3,4;7),\break (3,5;8)\}$.
This trade does not have the property set out in the above lemma but
it is decomposable.

Thus to further our study we focus on the trades of volume less than 
ten.
\vspace{.5cm}
%
%
%
%
%
%
%
%
%
%

\noindent {\bf REMARK:} We note that if such a decomposition exists for 
each $i=1,\dots, 6$ and each
$x\in V$, the partial Latin square $I_i$ is such 
that  $|R_x(I_i)|=|C_x(I_i)|=|E_x(I_i)|$ 
equals the replication number for symbol $x$. Also, since for 
$i=1,\dots, 6$, it follows that $|T|=|I_i|$,  the volume of each of the Latin 
interchanges ${\mathcal I}_i$ studied here is less than or equal to nine. In the 
paper  \cite{MR96a:05027} Keedwell classified the type of all 
Latin interchanges of volume less than or equal to 10. We have used 
his classifications when arguing that  decomposition is not possible 
and in many of these cases we shall frequently use the following lemma.

\begin{lemma}\label{lemma2} 
If the replication number of an element $e$ is $2$ or $3$, then 
for $i=1,\ldots,6$ in any given Latin interchange $\mathcal{I}_i$, 
$e$ can only occur as a row or a column
or a symbol.  If the replication number of a symbol $e$ is $4$,
then in any given Latin interchange $\mathcal{I}_i$, $e$ can only occur as a row,
a column, or a symbol, or a row and a column, a row and a symbol, 
or a column and a symbol.
\end{lemma}

\begin{proof} 
{
Since a Latin interchange requires that 
$|R_e(I)| \geq 2$ 
and $|C_e(I)| \geq 2$, and 
$|E_e(I)| \geq 2$, 
when the replication number of $e$
is 2 or 3, the element cannot be split among all of rows and columns,
or rows and symbols, or columns and symbols.  Similarly when
the replication number of $e$ is 4, the element cannot be split
between rows, columns, and symbols.
}
\end{proof}
\vspace{.5cm}

There are 25 Steiner trades of volume less than or equal to nine, and 
classifying these further we see that up to isomorphism there is one 
Steiner trade of volume 4, two Steiner trades of volume 6, two 
Steiner trades of volume 7, nine Steiner trades of volume 8 and eleven 
Steiner trades of volume 9. The triples of these trades are listed below. 
Our testing verified that for twelve of these Steiner trades the 
corresponding partial Steiner Latin square can be decomposed into six 
disjoint Latin interchanges satisfying the properties given in 
Question~\ref{714}. These twelve cases are discussed below and the general nature 
of the decomposition is given.  For the remainder of the cases, we present 
theoretical arguments  that indicate why such a 
decomposition is not possible. 
\vspace{.5cm}

\noindent {\bf Trade of volume 4} ${\mathcal T}_0=(T,T')$ 
where $T=\{1 2 3,1 5 6,4 3 5,4 2 6\}$ and
$T'=\{1 2 6,1 3 5,\break 4 2 3,4 5 6\}$.
This trade can be decomposed into Latin interchanges, corresponding to
$\mathcal{I}_1 = (I_1,I_1')$ where 
$I_1$ = \{(2, 3; 1),(5,6;1),(5,3;4),(2, 6; 4)\}.
\vspace{.5cm}

\noindent {\bf Trade of volume 6} ${\mathcal T}_1=(T,T')$ where 
$T=\{1 2 3,1 4 5,1 6 7,2 4 7,3 4 6,3 5 7\}$ and 
$T'=\{1 2 4,1 3 6,$ $ 1 5 7,2 3 7,3 4 5,4 6 7\}$.
The replication numbers for the elements $1,\dots, 7$ are:
\begin{center}
\begin{tabular}{c|c|c|c|c|c|c|c}
             Element & 1 & 2 & 3 & 4 & 5 & 6 & 7 \\
			  \hline
 Replication number in $T$ & 3 & 2 & 3 & 3 & 2 & 2 & 3 
\end{tabular}
\end{center}

Assume that the Latin interchanges associated 
with ${\mathcal T}_1$ can be decomposed into six disjoint Latin 
interchanges; then since  $volume(T_1)=6$, one of these Latin 
interchanges must have type
\begin{eqnarray*}
\left(
\begin{array}{c}
    3+3\cr
3+3\\
2+2+2
\end{array}
\right).
\end{eqnarray*}
So without loss of generality assume column $1$ contains three entries; 
but this implies there are three nonempty rows, which is a 
contradiction. Therefore no such decomposition exists.
\vspace{.5cm}

\noindent {\bf Trade of volume 6} ${\mathcal T}_2=(T,T')$ where 
$T=\{1 2 3,1 4 5,1 6 7,2 4 8,3 6 8,5 7 8 \}$
and $T'= \{1 2 4,1 3 6,1 5 7,$ $2 3 8,4 5 8,6 7 8 \}$.
This trade can be decomposed into Latin interchanges, corresponding to 
$\mathcal{I}_1 = (I_1,I_1')$ where 
$I_1= \{(1, 3; 2),(1, 4; 5),
(1, 7; 6),(8,4;2),(8, 3; 6),\break (8, 7; 5)\}$.

Here we digress for a moment and use this trade  to illustrate  Lemma~\ref{lemone}.
Note that $S_1=S_8=\{1,8\}$, $S_2=S_5=S_6=\{2,5,6\}$ and
$S_3=S_4=S_7=\{3,4,7\}.$
\vspace{.5cm}

\noindent {\bf Trade of volume 7} ${\mathcal T}_3=(T,T')$ where 
$T=\{1 2 3,1 4 5,1 6 7,2 4 6,2 5 7,3 5 6,3 4 7\}$ and 
$T'=\{1 2 4,1 3 6,$ $1 5 7,2 3 7,2 5 6,3 4 5,4 6 7\}$. 
The only possible type of a Latin interchange $\mathcal I$ of volume
seven is
\begin{eqnarray*}
\left(
\begin{array}{c}
3+2+2 \\
3+2+2 \\
3+2+2 \\
\end{array}
\right).
\end{eqnarray*}
Since the replication number of each element is $3$, this type 
is not possible.
\vspace{.5cm}

\noindent {\bf Trade of volume 7} ${\mathcal T}_4=(T,T')$ 
where $T=\{1 2 3,1 4 5,1 6 7,2 4 8,3 5 8,3 6 9,5 7 9\}$ and 
$T'=\{1 2 4,1 3 6,$ $1 5 7,2 3 8,3 5 9,4 5 8,6 7 9\}$.
A decomposition exists in which 
$\mathcal{I}_1 = (I_1,I_1')$ and where
\[I_1 =\{(1, 2; 3),(1, 5; 4),(1, 6; 7),(8, 2; 4),(8, 5; 3),
(9, 6; 3),(9,5;7)\}.\]

\noindent {\bf Trade of volume 8}  ${\mathcal T}_5=(T,T')$ where 
$T=\{1 2 3,1 4 5,1 6 7,2 4 8,2 5 7,3 4 6,3 7 8,
5 6 8\}$ and
$T'=\{1 2 4,1 3 6,1 5 7,2 3 7,2 5 8,3 4 8,4 5 6,6 7 8\}$.
As in the case of trade ${\mathcal T}_3$, the replication number for each element is 3,
and so it is not possible to find a type 

\begin{eqnarray*}
\left(
\begin{array}{c}
    W\cr
X\\
Y
\end{array}
\right),
\end{eqnarray*}
in which the sums W, X, and Y consist only of 3s. 
Thus no decomposition exists.
\vspace{.5cm}

\noindent {\bf Trade of volume 8} ${\mathcal T}_6=(T,T')$ 
where $T=\{1 2 3,1 4 5,1 6 7,2 4 6,2 5 7,3 5 9,3 6 8,
4 8 9\}$ and
$T'=\{1 2 4,1 3 6,1 5 7,2 3 5,2 6 7,3 8 9,4 5 9,4 6 8\}$.
The replication numbers for the elements $1,\ldots,9$ are as follows.
\begin{center}
\begin{tabular}{c|c|c|c|c|c|c|c|c|c}
Element             & 1 & 2 & 3 & 4 & 5 & 6 & 7 & 8 & 9 \\
			  \hline
 Replication  number in $T$ & 3 & 3 & 3 & 3 & 3 & 3 & 2 & 2 & 2 \\
\end{tabular}
\end{center}
But there is no Latin interchange of size 8 which has type
\begin{eqnarray*}
\left(
\begin{array}{c}
    3+3+2\cr
3+3+2\\
3+3+2
\end{array}
\right)
\end{eqnarray*}
and thus no decomposition exists.
\vspace{.5cm}

\noindent {\bf Trade of volume 8} ${\mathcal T}_7=(T,T')$ 
where $T=\{1 2 3,1 4 5,1 6 7,1 8 9,2 4 7,3 4 6,3 5 8,
3 7 9\}$ and
$T'=\{1 2 4,1 3 6,1 5 8,1 7 9,2 3 7,3 4 5,3 8 9,4 6 7\}$.
The replication numbers for the elements $1,\dots, 9$ are as follows.
\begin{center}
\begin{tabular}{c|c|c|c|c|c|c|c|c|c}
Element              & 1 & 2 & 3 & 4 & 5 & 6 & 7 & 8 & 9 \\
			  \hline
 Replication  number in $T$ & 4 & 2 & 4 & 3 & 2 & 2 & 3 & 2 & 2
\end{tabular}
\end{center}

Assume that the Latin interchanges associated 
with ${\mathcal T}_7$ can be decomposed into six disjoint 
Latin interchanges; then since  $volume(T)=8$, one of these 
Latin interchanges must have type
\begin{eqnarray*}
\left(
\begin{array}{c}
    3+3+2\cr
X\\
Y
\end{array}
\right),
\end{eqnarray*}
where $X$ and $Y$ represent the appropriate sum values of  
the number of filled entries in the rows of the Latin interchange
and the number of occurrences of each symbol in the Latin interchange.
By Lemma~\ref{lemma2}, this implies that both row $4$ and row $7$ are 
simultaneously non-empty. Moreover, the elements $4$ and $7$ cannot
occur as symbols. This is a contradiction as $247\in T$.
Therefore no such decomposition exists.
\vspace{.5cm}

\noindent {\bf Trade of volume 8} ${\mathcal T}_8=(T,T')$ 
where $T=\{1 2 7,1 3 8,2 8 A,3 7 9,4 5 9,4 6 A,5 7 A,
6 8 9\}$ 
and $T'=\{1 2 8,1 3 7,2 7 A,3 8 9,4 5 A,4 6 9,5 7 9,6 8 A \}$.
A decomposition exists in which
$\mathcal{I}_1 = (I_1,I_1')$ where 
$I_1=\{(1, 2; 7),(1, 3; 8), 
(A, 2; 8),(9, 3; 7),(9, 5; 4),
(A, 6; 4),\break (A, 5; 7),(9, 6; 8)\}$.  
\vspace{0.5cm}

\noindent {\bf Trade of volume 8} ${\mathcal T}_9=(T,T')$ 
where $T=\{1 2 3,1 4 5,1 6 7,1 8 9,2 4 A,2 6 8,2 7 9,
3 5 A\}$ and 
$T'=\{1 2 4,1 3 5,1 6 8,1 7 9,2 3 A,2 6 7,2 8 9,4 5 A\}$.
A decomposition exists in which
$\mathcal{I}_1 = (I_1,I_1')$ where 
$I_1=\{(1, 2; 3),(1, 5; 4),
(1, 6; 7),(1, 9; 8),(A, 2; 4),
(2, 6; 8),(2, 9; 7),\break (A, 5; 3)\}$.
\vspace{0.5cm}

\noindent {\bf Trade of volume 8} ${\mathcal T}_{10}=(T,T')$ where 
$T=\{1 2 3,1 4 5,1 6 7,1 8 9,2 4 A,3 5 A,6 8 A,
7 9 A\}$ and 
$T'=\{1 2 4,1 3 5,1 6 8,1 7 9,2 3 A,4 5 A,6 7 A,8 9 A \}$.
A decomposition exists in which
$\mathcal{I}_1 = (I_1,I_1')$ where 
$I_1=\{(1, 2; 3),(1, 5; 4),
(1, 6; 7),
(1, 9; 8),(A, 2; 4),(A, 5; 3),
(A, 6; 8),\break (A, 9; 7)\}$.
\vspace{0.5cm}

\noindent {\bf Trade of volume 8} ${\mathcal T}_{11}=(T,T')$ 
where $T=\{1 2 3,1 4 5,1 6 7,1 8 9,2 4 A,3 6 A,5 8 A,
7 9 A\}$ 
and $T'=\{1 2 4,1 3 6,1 5 8,1 7 9, 2 3 A,4 5 A,6 7 A,8 9 A \}$.
A decomposition exists in which
$\mathcal{I}_1 = (I_1,I_1')$ where 
$I_1=\{(1, 2; 3),(1, 5; 4),(1, 6; 7),
(1, 9; 8),(A, 2; 4),(A, 6; 3),
(A, 5; 8),\break (A, 9; 7)\}$.
\vspace{0.5cm}

\noindent {\bf Trade of volume 8} ${\mathcal T}_{12}=(T,T')$ 
where $T=\{1 2 3,1 4 5,1 6 7,1 8 9,2 4 A,6 8 B,7 9 B,
3 5 A\}$ and 
$T'=\{1 2 4,1 3 5,1 6 8,1 7 9,2 3 A,4 5 A,6 7 B,8 9 B \}$.
A decomposition exists in which
$\mathcal{I}_1 = (I_1,I_1')$ where 
$I_1=\{(3, 2; 1),(4, 5; 1),
(7, 6; 1),
(8, 9; 1),(4, 2; A),
(8, 6; B),(7, 9; B),\break (3, 5; A)\}$.
\vspace{0.5cm}

\noindent {\bf Trade of volume 8} ${\mathcal T}_{13}=(T,T')$ 
where $T=\{1 2 3,1 4 5,2 4 A,3 5 A,6 7 8,6 9 B,7 9 C,\break
8 B C\}$ 
and $T'=\{1 2 4,1 3 5,2 3 A,4 5 A,6 7 A,6 8 B,7 8 C,9 B C\}$.
A decomposition exists in which
$\mathcal{I}_1 = (I_1,I_1')$ where 
$I_1=\{(3, 2; 1),(4, 5; 1),
(4, 2; A),
(3, 5; A),
(8, 7; 6),\break (9, B; 6),(9, 7; C),(8, B; C)\}$.

\noindent {\bf Trade of volume 9} ${\mathcal T}_{14}=(T,T')$ 
where $T=\{1 4 5,1 6 7,1 8 9,2 3 9,2 5 7,2 6 8,
3 4 6,
3 5 8,\break 4 7 9 \}$ 
and $T'=\{1 4 6,1 5 8,1 7 9,2 3 5,2 6 7,2 8 9,3 4 9,3 6 8,4 5 7 \}$.  
Again the replication number for each element $e$ is 3.  By Lemma~\ref{lemma2}, 
any Latin interchange $I_1$ must be a 3 $\times$ 3 subsquare.
Assume that ${\mathcal I}_1$ is one of the Latin interchanges into
which the partial Latin square associated with ${\mathcal T}_{14}$ can
be decomposed.  
There are no 3 $\times$ 3 subsquares in the partial Latin square associated 
with $T$.  We can show this by considering the partial Latin square $I_1$ 
containing
the entry $(4, 5; 1)$.  By Lemma~\ref {lemma2}, 4 can only occur as 
a row, and 1 can only occur as a symbol.  Because $671 \in T$,
6 must occur only as a row or column.
Assume that 6 occurs only as a row.  In this
case, because $346 \in T$, either $(6, 4; 3)$ or $(6, 3; 4)$ 
must occur in $I_1$ which is a contradiction since 4 can only be a row. 
Thus 6 must occur only as a column.  In this case $(7, 6; 1)$ 
must be in $I_1$ and thus because $479 \in T$, $(7, 9; 4)$
or $(7, 4; 9)$ must be an entry in $I_1$ which is a contradiction
since we are assuming that 4 is a row. Thus no such decomposition exists.
\vspace{0.5cm}

\noindent {\bf Trade of volume 9} ${\mathcal T}_{15}=(T,T')$ 
where $T=\{1 4 7,1 5 8,1 6 9,2 4 8,2 5 9,2 6 7,3 4 9,\break 
3 5 7,3 6 8\}$ and
$T'=\{1 4 8,1 5 9,1 6 7,2 4 9,2 5 7,2 6 8,3 4 7,3 5 8,3 6 9 \}$.
A decomposition exists where one Latin interchange is given by 
$\mathcal{I}_1 = (I_1,I_1')$ where 
$I_1=\{(1,4;7),\break
(1,5;8),
(1,6; 9),(2,4; 8),(2,5;9),$
$(2,6;7),
(3,4;9),(3,5;7),$ $(3,6;8)\}$.
\vspace{0.5cm}

\noindent {\bf Trade of volume 9} ${\mathcal T}_{16}=(T,T')$ 
where $T=\{1 2 3,1 4 5,1 6 7,1 8 9,2 4 8,2 5 7,2 6 9,
3 4 6,\break 4 7 9 \}$ and
$T'=\{1 2 5,1 3 6,1 4 8,1 7 9,2 3 4,2 6 7,2 8 9,4 5 7,4 6 9 \}$.

The replication numbers for the elements $1,\dots, 9$ are as follows. 

\begin{center}
\begin{tabular}{c|c|c|c|c|c|c|c|c|c}
 Element             & 1 & 2 & 3 & 4 & 5 & 6 & 7 & 8 & 9 \\
			  \hline
 Replication  number in $T$ & 4 & 4 & 2 & 4 & 2 & 3 & 3 & 2 & 3
\end{tabular}
\end{center}

Assume $(6, 9; 2)\in I_1$, where $I_1$ forms one of the Latin interchanges
into which we are decomposing the partial Latin square associated 
with $T$. Then by Lemma~\ref{lemma2} we can say that 6 
occurs only as a row, and 9 occurs only
as a column.  Since $167 \in T$, then 7 occurs only as a
column or symbol, and since $479 \in T$, then 7 occurs
only as a row or symbol.  This means that 7 occurs only as a symbol.
Thus $\{ (6, 1; 7), (4, 9; 7) \} \subseteq I_1$.
With this information, plus the fact that $257$ 
and $145$ are triples, we have four cases:

{\bf Case 1}  $\{ (5, 2; 7), (5, 1; 4) \} \subseteq I_1$. Since  
$(6,9;2)$, $(6, 1; 7)$ and $(4, 9; 7)$ are also in $I_1$, by Lemma~\ref{lemma2} we find that $(6,3;4)$, $(1,3;2)$, $(1,9;8)$ and
$(4,2;8)$ must be in $I_1$. Now it is easy to see that $I_1$ is not 
a Latin interchange. This is a contradiction.

{\bf Case 2} $\{ (5, 2; 7), (5, 4; 1) \} \subseteq I_1$.  
Since $(6,9;2)$, $(6, 1; 7)$ and $(4, 9; 7)$ are also in $I_1$, by Lemma~\ref{lemma2} we find that $(6,4;3)\in I_1$. Now either $(1,2;3)$ or
$(2,1;3)$ must be in $I_1$. But both are impossible by Lemma~\ref{lemma2}.  So this case is also impossible.

{\bf Case 3} $\{ (2, 5; 7), (4, 5; 1) \} \subseteq I_1$.  
Since $(6,9;2)$, $(6, 1; 7)$ and $(4, 9; 7)$ are also in $I_1$, by Lemma~\ref{lemma2} we find that $(8,9;1)$, $(8,4;2)$, $(6,4;3)$ and 
$(2,1;3)$ must be in $I_1$. Now it is easy to see that $I_1$ with these
entries cannot be a Latin interchange. This is a contradiction.

{\bf Case 4} $\{ (2, 5; 7), (1, 5; 4) \} \subseteq I_1$. 
Since $(6,9;2)$, $(6, 1; 7)$ and $(4, 9; 7)$ are also in $I_1$, by Lemma~\ref{lemma2} we find that $(1,9;8)\in I_1$. Now either
$(4,2;8)$ or $(2,4;8)$ must be in $T_1$. But both are impossible by
Lemma~\ref{lemma2}.

Thus no decomposition exists.
\vspace{0.5cm}

\noindent {\bf Trade of volume 9} ${\mathcal T}_{17}=(T,T')$ 
where $T=\{1 2 3,1 4 5,1 6 7,1 8 9,2 4 8,2 5 6,2 7 9,
3 4 6,\break 3 5 8\}$ and
$T'=\{1 2 4,1 3 6,1 5 8,1 7 9,2 3 5,2 6 7,2 8 9,3 4 8,4 5 6 \}$.

The replication numbers for the elements $1,\dots, 9$ are as follows. 

\begin{center}
\begin{tabular}{c|c|c|c|c|c|c|c|c|c}
Element              & 1 & 2 & 3 & 4 & 5 & 6 & 7 & 8 & 9 \\
			  \hline
 Replication  number in $T$ & 4 & 4 & 3 & 3 & 3 & 3 & 2 & 3 & 2
\end{tabular}
\end{center}

Assume $(3, 4; 6)$ $\in I_1$, where $I_1$ forms one of the Latin interchanges
into which we are decomposing the partial Latin square associated 
with $T$.  Then 
by Lemma~\ref{lemma2} we can say that 3 occurs only as a row, 4 occurs only
as a column, and 6 occurs only as a symbol.  Since $256 \in T$, 
then 5 occurs only as a column or a row, and since $358 \in T$,
then 5 occurs only as a column or a symbol.  This implies that
5 occurs only as a column.  However, this leads to a contradiction
since if we look at the triple $145$ of $T$, 4 and 5 must both occur
as columns.  Thus no decomposition exists. 
\vspace{0.5cm}

\noindent {\bf Trade of volume 9} ${\mathcal T}_{18}=(T,T')$ 
where $T=\{1 2 3,1 4 5,1 6 7,2 4 8,3 6 9,3 7 8,4 9 A,
5 7 9,\break 6 8 A\}$ and
$T'=\{1 2 4,1 3 6,1 5 7,2 3 8,3 7 9,4 5 9,4 8 A,6 7 8,6 9 A \}$.

The replication numbers for the elements $1,\dots, 9, A$ are as follows.

\begin{center}
\begin{tabular}{c|c|c|c|c|c|c|c|c|c|c}
Element              & 1 & 2 & 3 & 4 & 5 & 6 & 7 & 8 & 9 & A \\
			  \hline
 Replication  number in $T$ & 3 & 2 & 3 & 3 & 2 & 3 & 3 & 3 & 3 & 2
\end{tabular}
\end{center}

Assume $(1, 2; 3)$ $\in I_1$, where $I_1$ forms one of the Latin interchanges
into which we are decomposing the partial Latin square associated 
with $T$.  Then by Lemma~\ref{lemma2} we can say that 
1 occurs only as a row, 2 occurs only as a column, and 3 occurs 
only as a symbol.  Since $248 \in T$, 
we see that 4 occurs only as a row or a symbol, and since $145$
is a triple, we see that 4 occurs only as a column or a symbol.
Therefore, 4 occurs only as a symbol, 5 occurs only as a column,
and 8 occurs only as a row.
Since $49A \in T$, we see that 9 occurs only as a row
or a column, and since $579 \in T$, we see that 9
occurs only as a row or a symbol.  Therefore, 9 occurs only as
a row, A occurs only as a column, and 7 occurs only as a symbol.
However, this leads to a contradiction since in the triple $378$ of $T$,
3 and 7 must both be symbols.  Therefore no decomposition exists.

\noindent {\bf Trade of volume 9} ${\mathcal T}_{19}=(T,T')$ 
where $T=\{1 2 3,1 4 5,1 6 7,1 8 9,2 4 A,3 5 6,3 7 A,
4 6 8,\break 4 7 9\}$ and
$T'=\{1 2 4,1 3 5,1 6 8,1 7 9,2 3 A,3 6 7,4 5 6,4 8 9,4 7 A \}$.
A decomposition exists in which 
${\mathcal I}_1 = (I_1,I_1')$ where 
$I_1=\{(1, 2; 3),
(1, 5; 4),
(6, 7; 1),(9, 8; 1),(A, 2; 4),(6, 5; 3),\break
(A, 7; 3),(6, 8; 4),$ 
$(9,7;4)\}$.

\noindent {\bf Trade of volume 9} ${\mathcal T}_{20}=(T,T')$ 
where $T=\{1 2 3,1 4 5,1 6 7,1 8 9,2 4 A,3 6 8,3 9 A,\break
4 7 9,5 7 8\}$ 
and $T'=\{1 2 4,1 3 6,1 5 8,1 7 9,2 3 A,3 8 9,4 5 7,4 9 A,6 7 8\}$.

The replication numbers for the elements $1,\dots, 9, A$ are as
follows.

\begin{center}
\begin{tabular}{c|c|c|c|c|c|c|c|c|c|c}
Element              & 1 & 2 & 3 & 4 & 5 & 6 & 7 & 8 & 9 & $A$ \\
			  \hline
 Replication  number in $T$ & 4 & 2 & 3 & 3 & 2 & 2 & 3 & 3 & 3 & 2
\end{tabular}
\end{center}

Assume $(1, 2; 4)$ $\in I_1$, where $I_1$ forms one of the Latin interchanges 
into which we are decomposing the partial Latin square associated 
with $T$.  Then, by Lemma~\ref{lemma2}, we can say that 
2 occurs only as a column and 4 occurs 
only as a symbol.  Since $23A \in T$, we see that 
$A$ occurs only as a row or a symbol, and since $49A \in T$, 
we find that $A$ occurs only as a row or a column. Therefore, $A$ occurs 
only as a row and we must have $(A,2;3),(A,9;4)\in I_1$.
Then we must have $(8,9;3)\in I_1$.
Since $678 \in T$, we see that $6$ occurs only 
as a column or a symbol, and since $136 \in T$, 
we find that $6$ occurs only as a row or a column. Therefore, $6$ occurs 
only as a column and we must have $(8,6;7),(1,6;3)\in I_1$.
However, this leads to a contradiction since in the triple $457$ of $T$,
4 and 7 must both be symbols.  Therefore no decomposition exists.

\noindent {\bf Trade of volume 9} ${\mathcal T}_{21}=(T,T')$ 
where $T=\{1 2 3,1 4 5,1 6 7,1 8 9,2 4 A, 3 4 6,3 5 8,
3 9 A,\break 4 7 9\}$ 
and $T'=\{1 2 4,1 3 6,1 5 8,1 7 9,2 3 A,3 4 5,3 8 9,4 6 7,4 9 A\}$.

The replication numbers for the elements $1,\dots, 9, A$ are as follows.

\begin{center}
\begin{tabular}{c|c|c|c|c|c|c|c|c|c|c}
Element              & 1 & 2 & 3 & 4 & 5 & 6 & 7 & 8 & 9 & $A$ \\
			  \hline
 Replication  number in $T$ & 4 & 2 & 4 & 4 & 2 & 2 & 2 & 2 & 3 & 2
\end{tabular}
\end{center}

Assume that the partial Steiner Latin square $I$ 
associated with $T$ can be decomposed into six disjoint 
Latin interchanges; then, since  $volume(T)=9$, one of these Latin 
interchanges must have type
\begin{eqnarray*}
\left(
\begin{array}{c}
    W\cr
X\\
Y
\end{array}
\right),
\end{eqnarray*}
where  $W$, $X$ and $Y$  are all odd and represent the appropriate 
sums for the number of symbols in each row and column and the
frequency of each symbol's occurrence in $I$,
$|E_e(I)|$. However it is 
not possible to partition the multiset 
$\{4,2,4,4,2,2,2,2,3,2\}$ into three multisubsets 
such that the sum of the entries in each of these multisubsets is odd.
Therefore no such decomposition exists.

\noindent {\bf Trade of volume 9} ${\mathcal T}_{22}=(T,T')$ 
where $T=\{1 2 3,1 4 5,1 6 7,1 8 9,2 4 A, 3 6 A,4 6 8,
4 7 9,\break 5 7 8\}$ 
and $T'=\{1 2 4,1 3 6,1 5 8,1 7 9,2 3 A,4 5 7,4 6 A,4 8 9,6 7 8\}$.

The replication numbers for the elements $1,\dots, 9, A$ are as
follows.

\begin{center}
\begin{tabular}{c|c|c|c|c|c|c|c|c|c|c}
Element              & 1 & 2 & 3 & 4 & 5 & 6 & 7 & 8 & 9 & $A$ \\
			  \hline
Replication  number in $T$ & 4 & 2 & 2 & 4 & 2 & 3 & 3 & 3 & 2 & 2
\end{tabular}
\end{center}

Assume $(5, 7; 8)$ $\in I_1$, where $I_1$ forms one of the Latin interchanges 
into which we are decomposing the partial Latin square associated 
with $T$.  Then, by Lemma~\ref{lemma2}, we can say that 
$5$ occurs only as a row, $7$ occurs only as a column, and $8$ occurs
only as a symbol. On the other hand, since $167$ and $468$ are triples of $T$,
we must have $(6,7;1),(6,4;8)\in I_1$. Now considering
the triples $36A$ and $479$ of $T$, we have four different cases:

{\bf Case 1} $\{ (6, 3; A), (4, 7; 9) \} \subseteq I_1$ which means 
that 9 occurs only as a symbol by Lemma~\ref{lemma2}. Then $189$ being a 
triple of $T$ means that both 8 and 9 are symbols. This is a contradiction.

{\bf Case 2} $\{ (6, 3; A), (9, 7; 4) \} \subseteq I_1$ which means that 
9 occurs only as a row, A occurs only as a symbol, and 3 occurs only as a
column, by Lemma~\ref{lemma2}.  Then $189$ being a triple of $T$ means that 
$(9, 1; 8)$ is a entry in $I_1$.  Thus 1 occurs as both a column and 
a symbol. Since $123 \in T$, this means that $(2, 3; 1)$ must be
an entry in $I_1$, which means that 2 can only occur as a row.
Then $24A$ being a triple of $T$ means that $(2, 4; A)$ is an entry of $I_1$.
Thus 4 occurs as both a column and a symbol.
Then $145$ being a triple of $T$ means that $(5, 1; 4)$ must be an entry of $I_1$,
since column 1 needs to have two entries in it. It is now easy to see 
that $I_1$ with these entries cannot be a Latin interchange.
This is a contradiction.

{\bf Case 3} $\{ (6, A; 3), (4, 7; 9) \} \subseteq I_1$ which means that 
9 occurs only as a symbol by Lemma~\ref{lemma2}, leading to a 
contradiction as in Case 1.

{\bf Case 4} $\{ (6, A; 3), (9, 7; 4) \} \subseteq I_1$ which means 
that 9 occurs only as a row, A occurs only as a column, and 3 occurs 
only as a symbol by Lemma~\ref{lemma2}.  
Then $189$ being a triple of $T$ means that that $(9, 1; 8)$
must be an entry in $I_1$.  Thus 1 occurs as both a column and a symbol.
Since $123 \in T$, this means that $(2, 1; 3)$ must be
an entry in $I_1$, which means that 2 can only occur as a row.
Then $24A$ being a triple of $T$ means that $(2, A; 4)$ must be an entry of $I_1$.
Thus 4 occurs as both a column and a symbol.  Then $145$ being
a triple of $T$ means that $(5, 4; 1)$ must be an entry of $I_1$,
since row 5 needs to have two symbols in it. 
It is now easy to see that $I_1$ with these entries cannot be a 
Latin interchange.  This is a contradiction.  Thus no decomposition is possible.

\noindent {\bf Trade of volume 9} ${\mathcal T}_{23}=(T,T')$ 
where $T=\{1 2 3,1 4 5,1 6 7,2 4 8,3 6 8,4 9 A,5 7 9,
6 9 B,\break 8 A B\}$ 
and $T'=\{1 2 4,1 3 6,1 5 7,2 3 8,4 5 9,6 7 9,4 8 A,6 8 B,9 A B\}$.
A decomposition exists in which
$\mathcal{I}_1 = (I_1,I_1')$ where 
$I_1=\{(3, 2; 1),(4, 5; 1),
(7, 6; 1),(4, 2; 8),(3, 6; 8),\break (4, A; 9),
 (7, 5; 9),(B, 6; 9),$ $(B, A; 8)\}$.

\noindent {\bf Trade of volume 9} ${\mathcal T}_{24}=(T,T')$ 
where $T=\{1 2 3,1 4 5,1 6 7,1 8 9,2 4 A,3 6 A,4 9 B,\break
5 8 B,7 9 A\}$ 
and $T'=\{1 2 4,1 3 6,1 5 8,1 7 9,2 3 A,4 5 B,4 9 A,6 7 A,8 9 B\}$.
A decomposition exists in which
$\mathcal{I}_1 = (I_1,I_1')$ where 
$I_1=\{(3, 2; 1),(4, 5; 1),
(7, 6; 1),(8, 9; 1),(4, 2; A),\break (3, 6; A),
(4, 9; B),$ 
$(8, 5; B),
(7, 9; A)\}$.

\section{Conclusion}
Thus in answer to Question~\ref{714}, we have developed
a theorem which determines the decomposability of certain Latin
interchanges corresponding to Steiner minimal trades.  Also, we have given
a definite answer to the question of the decomposability of the Steiner
partial Latin squares corresponding to each Steiner trade of volume 
less than or equal to 9. 





\chapter{A census of critical sets in the Latin squares of order at most six}\label{ch8}

Many papers have examined the problems of determining the smallest and
largest critical sets for particular orders of Latin square, or given
examples of critical sets for small orders of Latin square.  We give a
brief overview of these papers.

The sizes of smallest critical sets for the Latin squares of orders
four and five were determined in \cite {MR95k:05030,MR80j:05022}.
Howse in \cite {MR98m:05026} finds smallest critical sets for all the
Latin squares of order six.  This chapter enumerates all critical
sets for each main class of order six, and Appendix~\ref{app2} gives examples of
each possible size of critical set in each main class.

Also, a paper \cite{MR99k:05038} by Donovan gave
examples of critical sets of order six of all possible sizes.

Adams and Khodkar in \cite{AK2} give smallest critical sets for all the
Latin squares of order at most seven. They also find, in \cite{MR2000k:05054},
smallest weak and smallest totally weak critical sets for the Latin
squares of order at most seven.  The size of smallest strong critical
sets in a Latin square has also been considered in the past (see \cite
{MR2000g:05034}).

This chapter deals with critical sets of different sizes in the Latin
squares of order at most six.  First, for each main class of Latin
square of order at most six, we calculate every possible critical set.
These will be of various sizes.  
Then, for each main class of
Latin square and possible size of critical set, we determine the main
classes and isotopy classes for this set of critical sets.  Next,
we determine which of the main classes of critical sets are strong,
near-strong, totally weak, and Bedford-Whitehouse
totally weak.  Some interesting properties concerning the greatest common
divisors of numbers of critical sets across main classes in the $6 \times
6$ Latin squares and ratios of various kinds are discussed.

Finally, for some of the Latin squares we consider, the possibility of the Latin
square being partitioned into disjoint critical sets is examined.

\vspace{5mm}

\section{Algorithms}

%
%
To obtain the results presented here, we used two basic algorithms 
to calculate all critical sets of a given size $m$ for a given main
class of $n \times n$ Latin squares.

The first was Algorithm~\ref{311}, and the second algorithm used the
improvements noted in Chapter~\ref{ch3}.  This algorithm divided the Latin
square up into disjoint Latin interchanges, ensuring that each candidate for
a critical set had at least one entry in each of the Latin interchanges.

For the case of the $6 \times 6$ Latin squares, in the search for critical
sets of size greater than 18, the improvements noted in Chapter~\ref{ch3} were
used.  We briefly recap these improvements here.  For such subsets examined,
the search speed was further increased by ensuring that 
no row or column was full and no symbol occurred six times.
Such subsets cannot be critical sets since any entry may be removed from
the relevant row, column or symbol set while maintaining the unique
completion property.

We also use the result of Chapter~\ref{ch4}, that \lcs{n} $\leq n^2-3n+3$, to
exclude from consideration any subset of size greater than 21.

\section{Tables of results}

\subsection{Explanation of headings}
The first column in the tables of results (Tables~\ref{one}, \ref{two}, 
\ref{three} and \ref{four}) is
the main class number $n.z$ (LS), followed by the size(s) of the critical
set(s) for the main class (Size), the number of critical sets of that
size in the main class (\#CS); this is then followed by the number of
isotopy classes (\#Iso) and the number of main classes of those critical sets (\#Main).
(The notation $n.z$ denotes main class $z$ in
a Latin square of order $n$, as in the CRC Handbook of Combinatorial
Designs,~\cite{crc}.)  


For the next four columns, we consider representatives of each main
class of critical sets, and list the number of critical sets of various
``strengths'' within the main classes of critical sets.  That is,
we calculate how many of the representatives of each main class of
critical set have the various ``strengths''.  We need only consider
representatives of each main class of critical set, since, for example,
a strong critical set remains a strong critical set when the rows, columns and
symbols are interchanged or swapped.  Similarly, a 
near-strong critical set remains near-strong  under
permutations or interchanges of rows, columns or symbols.
These last four columns are, in order, the number of 
near-strong critical sets (\#NS), the number of strong critical sets
(\#Strong), the number of totally weak critical sets (\#TW), and the
number of Bedford-Whitehouse totally weak critical sets (\#BWTW).

\subsection{Latin squares of order 3}

There is only one main class, denoted 3.1, for Latin squares of order three \cite{crc}:
\vspace{5mm}

\begin{center}
\begin{tabular}{|c|c|c|} \hline
1&2&3 \\ \hline
2&3&1 \\ \hline
3&1&2 \\ \hline
\end{tabular} \vspace{3mm}\\
{\bf 3.1}
\end{center}
\vspace{5mm}

For this Latin square, we have the results presented in Table~\ref{one} concerning the number
of critical sets of every possible size.

\begin{table}
\begin{center}
\caption{Critical set statistics for Latin squares of order 3}
\label{one}
\begin{tabular}{|c|c|c|c|c|c|c|c|c|} \hline
LS & Size & \#CS & \#Iso & \#Main & \#NS & \#Strong & \#TW & \#BWTW \\ \hline
$3.1$  & $2$ & $9$ & $1$ & $1$ & 1 & 1 & 0 & 0  \\ 
& $3$ & 18 & $1$ & $1$ & 1 & 1 & 0 & 0 \\ \hline
\end{tabular}\vspace {3mm} \\
\end{center}
\end{table}

\subsection{Latin squares of order 4}

There are two main classes, denoted 4.1 and 4.2, for Latin squares of order four \cite{crc}:
\vspace{5mm}

\begin{center}
\begin{tabular}{cccc}
  \begin{tabular}{|c|c|c|c|} \hline
  1&2&3&4 \\ \hline
  2&3&4&1 \\ \hline
  3&4&1&2 \\ \hline
  4&1&2&3 \\ \hline
  \end{tabular} &&&
  \begin{tabular}{|c|c|c|c|} \hline
  1&2&3&4 \\ \hline
  2&1&4&3  \\ \hline
  3&4&1&2  \\ \hline
  4&3&2&1 \\ \hline
  \end{tabular} \\
  {\bf 4.1} &&& {\bf 4.2} \\
\end{tabular}
\end{center}
\vspace{5mm}

For these Latin squares, we have the results presented in Table~\ref{two} concerning the number
of critical sets of every possible size.

\begin{table}
\begin{center}
\caption{Critical set statistics for Latin squares of order 4}
\label{two}
\begin{tabular}{|c|c|c|c|c|c|c|c|c|} \hline
LS & Size & \#CS & \#Iso & \#Main & \#NS & \#Strong & \#TW & \#BWTW \\ \hline
4.1 & $4$ & $32$  & $1$  & $1$ & 1 & 1 & 0 & 0\\ 
    & $5$ & $576$ & $18$ & $4$ & 4 & 4 & 0 & 0\\
    & $6$ & $128$ & $4$  & $2$ & 2 & 2 & 0 & 0\\ \hline
4.2 & $5$ & $96$  & $1$  & $1$ & 1 & 1 & 0 & 0\\
    & $6$ & $432$ & $7$  & $3$ & 3 & 3 & 0 & 0\\
    & $7$ & $48$  & $1$ & $1 $ & 1 & 1 & 0 & 0\\ \hline
\end{tabular}\vspace {3mm} \\
\end{center}
\end{table}

\subsection{Latin squares of order 5}

There are two main classes, denoted 5.1 and 5.2, for Latin squares of order five \cite{crc}:

\vspace{5mm}

\begin{center}
\begin{tabular}{cccc}
  \begin{tabular}{|c|c|c|c|c|} \hline
  1&2&3&4&5  \\ \hline
  2&3&4&5&1  \\ \hline
  3&4&5&1&2  \\ \hline
  4&5&1&2&3  \\ \hline
  5&1&2&3&4   \\ \hline
  \end{tabular} &&&
  \begin{tabular}{|c|c|c|c|c|} \hline
  1&2&3&4&5 \\ \hline
  2&1&4&5&3 \\ \hline
  3&4&5&1&2 \\ \hline
  4&5&2&3&1 \\ \hline
  5&3&1&2&4 \\ \hline
  \end{tabular} \\
  {\bf 5.1} &&& {\bf 5.2} \\
\end{tabular}
\end{center}
\vspace{5mm}

For these Latin squares, we have the results presented in Table~\ref{three} concerning the number
of critical sets of every possible size.

\begin{table}
\begin{center}
\caption{Critical set statistics for Latin squares of order 5}
\label{three}
\begin{tabular}{|c|c|c|c|c|c|c|c|c|} \hline
 LS & Size & \#CS & 
\#Iso & \#Main & \#NS & \#Strong & \#TW & \#BWTW \\ \hline
$5.1$ & $6$  & $50$    & $1$    & $1$ & 1 & 1 & 0 & 0 \\
      & $7$  & $1000$  & $10$   & $4$ & 4 & 4  & 0 & 0 \\
      & $8$  & $30900$ & $312$  & $57$ & 57 & 57 & 0 & 0 \\
      & $9$  & $18800$ & $188$  & $37$ & 37 & 37 & 0 & 0 \\
      & $10$ & $2500$  & $25$   & $6$ & 6 & 6 & 0 & 0 \\ \hline
$5.2$ & $7$  & $600$   & $50$   & $11$ & 10 & 10 & 1 & 1 \\
      & $8$  & $21588$ & $1802$ & $322$ & 311 & 311 & 1 & 1 \\
      & $9$  & $23718$ & $1981$ & $348$ & 348 & 348 & 0 & 0 \\
      & $10$ & $2340$  & $198$  & $39$ & 38 & 36 & 2 & 0 \\
      & $11$ & $216$   & $18$   & $4$ & 4 & 4 & 0 & 0 \\ \hline
\end{tabular}\vspace {3mm} \\
\end{center}
\end{table}

\subsection{Latin squares of order 6}
There are twelve main classes, denoted 6.1, \dots, 6.12, for Latin squares of order six \cite{crc}:

\begin{center}
\setlength{\tabcolsep}{4.5pt}
\begin{tabular}{cccc}
\begin{tabular}{|*{6}{c|}} \hline
1&2&3&4&5&6  \\ \hline
2&3&4&5&6&1  \\ \hline
3&4&5&6&1&2  \\ \hline
4&5&6&1&2&3  \\ \hline
5&6&1&2&3&4  \\ \hline
6&1&2&3&4&5  \\ \hline
\end{tabular}&
\begin{tabular}{|*{6}{c|}} \hline
1&2&3&4&5&6  \\ \hline
2&1&4&3&6&5  \\ \hline
3&4&5&6&1&2  \\ \hline
4&3&6&5&2&1  \\ \hline
5&6&1&2&4&3  \\ \hline
6&5&2&1&3&4  \\ \hline
\end{tabular}&
\begin{tabular}{|*{6}{c|}} \hline
1&2&3&4&5&6  \\ \hline
2&1&4&5&6&3  \\ \hline
3&4&1&6&2&5  \\ \hline
4&5&6&1&3&2  \\ \hline
5&6&2&3&1&4  \\ \hline
6&3&5&2&4&1  \\ \hline
\end{tabular} &
\begin{tabular}{|*{6}{c|}} \hline
1&2&3&4&5&6  \\ \hline
2&1&4&5&6&3  \\ \hline
3&4&1&6&2&5  \\ \hline
4&5&6&1&3&2  \\ \hline
5&6&2&3&4&1  \\ \hline
6&3&5&2&1&4  \\ \hline
\end{tabular} \\ 
{\bf 6.1} & {\bf 6.2} & {\bf 6.3} & {\bf 6.4}\\
\end{tabular}
\end{center}

\begin{center}
\setlength{\tabcolsep}{4.5pt}
\begin{tabular}{cccc}
\begin{tabular}{|*{6}{c|}} \hline
1&2&3&4&5&6  \\ \hline
2&1&4&5&6&3  \\ \hline
3&4&2&6&1&5  \\ \hline
4&5&6&2&3&1  \\ \hline
5&6&1&3&4&2  \\ \hline
6&3&5&1&2&4  \\ \hline
\end{tabular} & 
\begin{tabular}{|*{6}{c|}} \hline
1&2&3&4&5&6  \\ \hline
2&1&4&5&6&3  \\ \hline
3&4&5&6&1&2  \\ \hline
4&5&6&3&2&1  \\ \hline
5&6&1&2&3&4  \\ \hline
6&3&2&1&4&5  \\ \hline
\end{tabular} & 
\begin{tabular}{|*{6}{c|}} \hline
1&2&3&4&5&6  \\ \hline
2&1&4&3&6&5  \\ \hline
3&5&1&6&2&4  \\ \hline
4&6&2&5&1&3  \\ \hline
5&3&6&1&4&2  \\ \hline
6&4&5&2&3&1  \\ \hline
\end{tabular} & 
\begin{tabular}{|*{6}{c|}} \hline
1&2&3&4&5&6  \\ \hline
2&1&4&3&6&5  \\ \hline
3&5&1&6&2&4  \\ \hline
4&6&2&5&1&3  \\ \hline
5&3&6&2&4&1  \\ \hline
6&4&5&1&3&2  \\ \hline
\end{tabular} \\
{\bf 6.5} & {\bf 6.6} & {\bf 6.7} & {\bf 6.8}\\
\end{tabular}
\end{center}

\begin{center}
\setlength{\tabcolsep}{4.5pt}
\begin{tabular}{cccc}
\begin{tabular}{|*{6}{c|}} \hline
1&2&3&4&5&6  \\ \hline
2&1&4&3&6&5  \\ \hline
3&5&1&6&2&4  \\ \hline
4&6&2&5&3&1  \\ \hline
5&4&6&2&1&3  \\ \hline
6&3&5&1&4&2  \\ \hline
\end{tabular} & 
\begin{tabular}{|*{6}{c|}} \hline
1&2&3&4&5&6  \\ \hline
2&1&4&3&6&5  \\ \hline
3&5&1&6&4&2  \\ \hline
4&6&5&1&2&3  \\ \hline
5&3&6&2&1&4  \\ \hline
6&4&2&5&3&1  \\ \hline
\end{tabular} & 
\begin{tabular}{|*{6}{c|}} \hline
1&2&3&4&5&6  \\ \hline
2&1&4&5&6&3  \\ \hline
3&4&2&6&1&5  \\ \hline
4&6&5&2&3&1  \\ \hline
5&3&6&1&2&4  \\ \hline
6&5&1&3&4&2  \\ \hline
\end{tabular} & 
\begin{tabular}{|*{6}{c|}} \hline
1&2&3&4&5&6  \\ \hline
2&3&1&5&6&4  \\ \hline
3&1&2&6&4&5  \\ \hline
4&6&5&2&1&3  \\ \hline
5&4&6&3&2&1  \\ \hline
6&5&4&1&3&2  \\ \hline
\end{tabular} \\
{\bf 6.9} & {\bf 6.10} & {\bf 6.11} & {\bf 6.12}\\
\end{tabular}
\end{center}

For these Latin squares, we have the results presented in Table~\ref{four} concerning the number
of critical sets of every possible size.

\begin{table}
\begin{center}
\caption{Critical set statistics for Latin squares of order 6}
\label{four}
\begin{tabular}{|c|c|c|c|c|c|c|c|c|} \hline
 LS & Size & \#CS & \#Iso & \#Main& \#NS & \#Strong & \#TW & \#BWTW \\ \hline
$6.1$ & $9$  & $72$    & $1$    & $1$ & 1 & 1 & 0 & 0 \\
      & $11$  & $39384$ & $547$  & $97$ & 97 & 95 & 0 & 0  \\
      & $12$  & $1161036$ & $16149$  & $2740$ & 2541 & 2513 & 11 & 8 \\
      & $13$  & $3634344$ & $50492$  & $8481$ & 7815 & 7792 & 19 & 14 \\
      & $14$  & $886428$ & $12346$  & $2090$ & 1931 & 1920 & 10 & 9 \\
      & $15$  & $80064$ & $1118$  & $202$ & 182 & 168 & 8 & 0 \\
      & $16$  & $3240$ & $45$  & $8$ & 8 & 8 & 0 & 0 \\
      & $17$  & $108$  & $3$   & $1$ & 0 & 0 & 0 & 0 \\ \hline
$6.2$ & $11$  & $7848$ & $327$  & $60$ & 50 & 48 & 3 & 3 \\
      & $12$  & $658908$ & $27477$  & $4633$ & 4370 & 4325 & 35 & 27 \\
      & $13$  & $3328908$ & $138708$  & $23267$ & 22226 & 22187 & 52 & 36 \\
      & $14$  & $1800228$ & $75035$  & $12617$ & 12267 & 12263 & 11 & 8 \\
      & $15$  & $192480$ & $8022$  & $1362$ & 1354 & 1351 & 3 & 1 \\
      & $16$  & $15840$ & $660$  & $115$ & 113 & 113 & 0 & 0 \\
      & $17$  & $240$  & $10$   & $3$ & 3 & 3 & 0 & 0  \\ \hline
$6.3$ & $11$  & $1200$ & $10$  & $7$ & 2 & 2 & 0 & 0 \\
      & $12$  & $192360$ & $1603$  & $836$ & 749 & 748 & 14 & 14 \\
      & $13$  & $1837440$ & $15315$  & $7757$ & 7445 & 7440 & 33 & 29 \\
      & $14$  & $1727880$ & $14400$  & $7279$ & 7252 & 7252 & 2 & 2 \\
      & $15$  & $378928$ & $3162$  & $1610$ & 1610 & 1610 & 0 & 0 \\
      & $16$  & $20280$ & $169$  & $90$ & 90 & 90 & 0 & 0 \\
      & $17$  & $840$  & $7$   & $4$ & 4 & 4 & 0 & 0 \\ \hline
\end{tabular}\vspace {3mm} \\
\end{center}
\end{table}
\newpage
\begin{center}
Table~\ref{four} (continued)
\begin{tabular}{|c|c|c|c|c|c|c|c|c|} \hline
 LS & Size & \#CS & \#Iso & \#Main& \#NS & \#Strong & \#TW & \#BWTW \\ \hline
$6.4$ & $10$  & $56$ & $7$  & $5$ & 5 & 5 & 0 & 0 \\
      & $11$  & $34000$ & $4250$  & $2149$ & 2001 & 1980 & 16 & 10 \\
      & $12$  & $1590608$ & $198826$  & $99654$ & 94485 & 94024 & 197 & 136 \\
      & $13$  & $5498076$ & $687262$  & $344044$ & 328754 & 327801 & 331 & 232 \\
      & $14$  & $1931424$ & $241428$  & $120895$ & 116390 & 115691 & 58 & 34 \\
      & $15$  & $168752$ & $21095$  & $10586$ & 10102 & 9704 & 137 & 10 \\
      & $16$  & $13736$ & $1717$  & $871$ & 821 & 780 & 24 & 1 \\
      & $17$  & $148$  & $19$   & $11$ & 9 & 9 & 1 & 1 \\ \hline
$6.5$ & $10$  & $60$ & $15$  & $9$ & 9 & 9 & 0 & 0 \\
      & $11$  & $42992$ & $10748$  & $5406$ & 5132 & 5078 & 30 & 24 \\
      & $12$  & $1878236$ & $469559$  & $235063$ & 224705 & 223776 & 401 & 292 \\
      & $13$  & $6475142$ & $1618806$  & $809952$ & 778258 & 776251 & 648 & 473 \\
      & $14$  & $2182652$ & $545663$  & $273120$ & 264790 & 263229 & 119 & 76 \\
      & $15$  & $192416$ & $48104$  & $24108$ & 23304 & 22340 & 281 & 28 \\
      & $16$  & $16908$ & $4227$  & $2135$ & 2041 & 1961 & 43 & 3 \\
      & $17$  & $112$  & $28$   & $16$ & 15 & 15 & 0 & 0 \\ \hline
$6.6$ & $11$  & $12888$ & $358$  & $187$ & 177 & 175 & 0 & 0 \\
      & $12$  & $856908$ & $23803$  & $12005$ & 11191 & 11155 & 64 & 55 \\
      & $13$  & $4097790$ & $113839$  & $57151$ & 54038 & 53898 & 162 & 120 \\
      & $14$  & $1476864$ & $41024$  & $20664$ & 19770 & 19697 & 27 & 23 \\
      & $15$  & $155196$ & $4311$  & $2201$ & 2139 & 2117 & 8 & 4 \\
      & $16$  & $12744$ & $354$  & $186$ & 175 & 166 & 3 & 0 \\
      & $17$  & $216$  & $6$   & $4$ & 4 & 4 & 0 & 0 \\ \hline
\end{tabular} \vspace{3mm} \\
\end{center}
\newpage
\begin{center}
Table~\ref{four} (continued)
\begin{tabular}{|c|c|c|c|c|c|c|c|c|} \hline
LS & Size & \#CS & \#Iso & \#Main& \#NS & \#Strong & \#TW & \#BWTW \\ \hline
$6.7$ & $12$  & $4752$ & $22$ & $5$ & 3 & 3 & 0 & 0 \\
      & $13$  & $212328$ & $985$ & $183$ & 165 & 165 & 5 & 5 \\
      & $14$  & $893700$ & $4151$ & $736$ & 706 & 705 & 3 & 2 \\
      & $15$  & $545508$ & $2536$ & $465$ & 465 & 465 & 0 & 0 \\
      & $16$  & $125766$ & $583$ & $109$ & 109 & 109 & 0 & 0 \\
      & $17$  & $8208$ & $38$ & $13$ & 13 & 13 & 0 & 0 \\ 
      & $18$  & $648$ & $3$ & $1$ & 1 & 1 & 0 & 0 \\ \hline
$6.8$ & $11$  & $3264$ & $408$ & $75$ & 67 & 66 & 3 & 2 \\
      & $12$  & $324608$ & $40576$ & $6817$ & 6023 & 5986 & 37 & 33 \\
      & $13$  & $1826592$ & $228335$ & $38265$ & 35161 & 35063 & 123 & 80 \\
      & $14$  & $1093796$ & $136729$ & $22909$ & 21804 & 21764 & 10 & 8 \\
      & $15$  & $106296$ & $13290$ & $2260$ & 2178 & 2155 & 10 & 1 \\
      & $16$  & $8464$ & $1058$ & $185$ & 175 & 167 & 6 & 1 \\ 
      & $17$  & $216$ & $27$ & $5$ & 5 & 5 & 0 & 0 \\ \hline
$6.9$ & $10$  & $24$ & $2$ & $2$ & 2 & 2 & 0 & 0 \\
      & $11$  & $13980$ & $1165$ & $596$ & 546 & 535 & 7 & 7 \\
      & $12$  & $716352$ & $59714$ & $29939$ & 27999 & 27723 & 127 & 84 \\
      & $13$  & $2784264$ & $232027$ & $116246$ & 109378 & 108885 & 243 & 173 \\
      & $14$  & $1065876$ & $88856$ & $44575$ & 42345 & 42068 & 24 & 19 \\
      & $15$  & $85884$ & $7159$ & $3607$ & 3462 & 3314 & 72 & 4 \\
      & $16$  & $6960$ & $580$ & $302$ & 283 & 259 & 15 & 1 \\ 
      & $17$  & $24$ & $2$ & $2$ & 2 & 2 & 0 & 0 \\ \hline
\end{tabular}\vspace {3mm} \\
\end{center}
\newpage
\begin{center}
Table~\ref{four} (continued)
\begin{tabular}{|c|c|c|c|c|c|c|c|c|} \hline
 LS & Size & \#CS & \#Iso & \#Main& \#NS & \#Strong & \#TW & \#BWTW \\ \hline
$6.10$& $10$  & $4$ & $1$ & $1$ & 1 & 1 & 0 & 0 \\
      & $11$  & $13748$ & $3437$ & $587$ & 555 & 547 & 6 & 3 \\
      & $12$  & $858348$ & $214587$ & $35899$ & 33814 & 33644 & 169 & 123 \\
      & $13$  & $3894038$ & $973520$ & $162538$ & 154803 & 154404 & 279 & 195 \\
      & $14$  & $1715492$ & $428873$ & $71685$ & 69560 & 69375 & 38 & 29 \\
      & $15$  & $155000$ & $38753$ & $6513$ & 6355 & 6232 & 34 & 4 \\
      & $16$  & $10540$ & $2635$ & $461$ & 443 & 423 & 13 & 1 \\ 
      & $17$  & $120$ & $30$ & $6$ & 6 & 6 & 0 & 0 \\ \hline
$6.11$& $10$  & $40$ & $10$ & $3$ & 3 & 3 & 0 & 0 \\
      & $11$  & $63540$ & $15885$ & $2673$ & 2617 & 2590 & 9 & 5 \\
      & $12$  & $2292266$ & $573254$ & $95781$ & 92453 & 92029 & 96 & 59 \\
      & $13$  & $7075888$ & $1768972$ & $295196$ & 284917 & 284027 & 145 & 96 \\
      & $14$  & $2203696$ & $550993$ & $91977$ & 88209 & 87499 & 39 & 22 \\
      & $15$  & $188344$ & $47086$ & $7890$ & 7584 & 7175 & 161 & 8  \\
      & $16$  & $17172$ & $4293$ & $729$ & 685 & 645 & 35 & 4 \\ 
      & $17$  & $36$ & $9$ & $2$ & 1 & 1 & 0 & 0 \\ \hline
$6.12$& $11$  & $143208$ & $1326$ & $232$ & 229 & 228 & 0 & 0 \\
      & $12$  & $3518478$ & $32664$ & $5510$ & 5384 & 5358 & 7 & 6 \\
      & $13$  & $9025344$ & $83568$ & $14037$ & 13636 & 13584 & 17 & 14 \\
      & $14$  & $2104704$ & $19506$ & $3315$ & 3146 & 3107 & 6 & 4 \\
      & $15$  & $200340$ & $1855$ & $326$ & 316 & 297 & 8 & 1 \\
      & $16$  & $17820$ & $165$ & $32$ & 29 & 28 & 1 & 0 \\ \hline
\end{tabular}\vspace {3mm} \\
\end{center}


D{\'e}nes and Keedwell ~\cite{dk} point out that, for a given order $n$, each
isotopy class of $n \times n$ Latin squares has a number of Latin
squares associated with it, and similarly each main class of $n \times
n$ Latin squares has a number of isotopy classes associated with it.
Similarly, for any given main class $n.z$ of $n \times n$ Latin squares
and given size of critical set $m$, if we consider the main classes of
critical sets of size $m$ within the main class $n.z$, we have several
associated isotopy classes of critical sets.  In the same way, if we
consider the isotopy classes of critical sets of size $m$ within the
main class $n.z$, we have several associated critical sets of size $m$
in the main class $n.z$.
\newpage
In Tables~\ref{five} to \ref{eight}, the head line refers to the twelve main classes of $6 \times
6$ Latin squares, and the side line refers to the possible sizes of
critical sets.  For $6 \times 6$ Latin squares, we consider results related to these observations.

Each isotopy class of critical sets in $6 \times 6$ Latin squares has between 
2 to 216 associated critical sets.  This result is given in Table~\ref{five}.

\begin{table}
\begin{center}
\setlength{\tabcolsep}{4.5pt}
\begin{footnotesize}
\caption{Numbers of critical sets in each isotopy class of critical sets of order six}
\vspace{1mm}
\label{five}
\begin{tabular}{|*{13}{c|}} \hline
& $6.1$ & $6.2$ & $6.3$ & $6.4$ & $6.5$ & $6.6$ & $6.7$ & $6.8$ & $6.9$ & $6.10$ & $6.11$ & $6.12$ \\
\hline
9 &  72 &   -   &    -   &  -   &  -   &   -   &    -   &  -   & -     &  -   &  -   &  - \\
10 &   -   &   -   &    -   & 8 & 4 &   -   &    -   &  -   & 12 & 4 & 4 &  - \\
11 & 72 & 24 & 120 & 8 & 4 & 36 &    -   & 8 & 12 & 4 & 4 & 108  \\
12 & 12,36,72 & 8,12,24 & 8  & 8 & 4 & 36 & 216 & 8 & 6,12 & 4 & 2,4 & 54,108 \\
13 & 36,72 & 12,24 & 60,120 & 4,8 & 2,4 & 18,36 & 108,216 & 4,8 & 6,12 & 2,4 & 4 & 108 \\
14 & 36,72 & 12,24 & 60,120 & 8 & 4  & 36 & 108,216 & 4,8 & 6,12 & 4 & 2,4 & 54,108 \\
15 & 36,72 & 12,24 & 24,40,120 & 4,8 & 4 & 36 & 108,216 & 4,8 & 6,12 & 2,4 & 4 & 108 \\
16 & 72 & 24 & 120 & 8 & 4 & 36 & 54,216 & 8 & 12 & 4 & 4 & 108 \\
17 & 36 & 24 & 120 & 4,8 & 4 & 36 & 216 & 8 & 12 & 4 & 4 &    - \\
18 &   -   &   -   &    -   &  -   &  -   &   -   & 216 &  -   &   -   &  -   &  -   &    -  \\
\hline
\end{tabular}\vspace {3mm} \\
\end{footnotesize}
\end{center}
\end{table}

Each main class of critical sets in $6 \times 6$ Latin squares has either
1, 2, 3 or 6 associated isotopy classes of critical sets.  
This result is given in Table~\ref{six}. 

\begin{table}
\begin{center}
\setlength{\tabcolsep}{4.5pt}
\caption{Numbers of isotopy classes of critical sets in each main class of critical sets of order six \vspace{3mm}}
\label{six}
\begin{tabular}{|*{13}{c|}} \hline
& $6.1$ & $6.2$ & $6.3$ & $6.4$ & $6.5$ & $6.6$ & $6.7$ & $6.8$ & $6.9$ & $6.10$ & $6.11$ & $6.12$ \\
\hline
9 &  1 &   -   &    -   &  -   &  -   &   -   &    -   &  -   & -     &  -   &  -   &  - \\
10 &   -   &   -   &    -   & 1,2 & 1,2 &   -   &    -   &  -   & 1 & 1 & 1,3,6 &  - \\
11 & 1,3,6 & 3,6 & 1,2 & 1,2 & 1,2 & 1,2 &    -   & 3,6 & 1,2  & 2,3,6 & 1,2,3,6 & 3,6  \\
12 & 1,2,3,6 & 1,2,3,6 & 1,2 & 1,2 & 1,2 & 1,2 & 1,3,6 & 1,2,3,6 & 1,2 & 1,2,3,6 & 1,2,3,6 & 1,2,3,6 \\
13 & 2,3,6 & 1,2,3,6 & 1,2 & 1,2 & 1,2 & 1,2 & 1,3,6 & 1,2,3,6 & 1,2 & 1,2,3,6 & 1,2,3,6 & 3,6 \\
14 & 1,2,3,6 & 1,2,3,6 & 1,2 & 1,2 & 1,2 & 1,2 & 1,2,3,6 & 1,2,3,6 & 1,2 & 2,3,6 & 2,3,6 & 3,6 \\
15 & 1,2,3,6 & 1,2,3,6 & 1,2 & 1,2 & 1,2 & 1,2 & 1,3,6 & 3,6 & 1,2 & 1,2,3,6 & 1,3,6 & 1,2,3,6 \\
16 & 3,6 & 3,6 & 1,2 & 1,2 & 1,2 & 1,2 & 1,3,6 & 2,3,6 & 1,2 & 1,2,3,6 & 3,6 & 3,6 \\
17 & 3 & 1,3,6 & 1,2 & 1,2 & 1,2 & 1,2 & 1,3,6 & 3,6 & 1 & 3,6 & 3,6 &    - \\
18 &   -   &   -   &    -   &  -   &  -   &   -   & 3 &  -   &   -   &  -   &  -   &    -  \\
\hline
\end{tabular}\vspace {3mm} \\
\end{center}
\end{table}
 
\section{Some observations}
We define some notation used in the following observations. We
shall denote the number of critical sets of size $x$ in a main class
$n.z$ by $CS(n,z,x)$. 
The number of isotopy classes of critical sets
of size $x$ in a main class $n.z$ shall be denoted $IC(n,z,x)$, and
the number of main classes of these critical sets shall be denoted
$MC(n,z,x)$.  The greatest common divisor of the number of critical
sets of all sizes in a particular main class $n.z$ will be referred to
as GCDCS$(n,z)$.

We shall concentrate on observations concerning the $6 \times 6$ Latin squares.

We find that when the main class $6.z$ is fixed and $x$ takes all possible
values, \break $CS(6,z,x)$ / $IC(6,z,x)$ is in most cases close to an integer
constant.  There is one exception: the critical sets of size 17 in main
class 6.1, where all other values of $CS(6,1,x) / IC(6,1,x)$ are approximately 72,
but $CS(6,1,17) / IC(6,1,17) = 36$.  We also find that this integer constant is a
multiple of GCDCS$(6,z)$.  These ratios are given in Table~\ref{seven},
truncated at two decimal places.  The last line tabulates the values
of GCDCS$(6,z)$.
\begin{table}
\begin{center}
\begin{footnotesize}
\caption{Ratio of critical sets to isotopy classes of critical sets of order six}
\label{seven}
\begin{tabular}{|c|c|c|c|c|c|c|c|c|c|c|c|c|} \hline
& $6.1$ & $6.2$ & $6.3$ & $6.4$ & $6.5$ & $6.6$ & $6.7$ & $6.8$ & $6.9$ & $6.10$ & $6.11$ & $6.12$ \\
\hline
9 &  72.00 &   -   &    -   &  -   &  -   &   -   &    -   &  -   & -     &  -   &  -   &  - \\
10 &   -   &   -   &    -   & 8.00 & 4.00 &   -   &    -   &  -   & 12.00 & 4.00 & 4.00 &  - \\
11 & 72.00 & 24.00 & 120.00 & 8.00 & 4.00 & 36.00 &    -   & 8.00 & 12.00 & 4.00 & 4.00 & 108.00  \\
12 & 71.89 & 23.98 & 120.00 & 8.00 & 4.00 & 36.00 & 216.00 & 8.00 & 11.99 & 4.00 & 3.99 & 107.71 \\
13 & 71.97 & 23.99 & 119.97 & 7.99 & 3.99 & 35.99 & 215.56 & 7.99 & 11.99 & 3.99 & 4.00 & 108.00 \\
14 & 71.79 & 23.99 & 119.99 & 8.00 & 4.00 & 36.00 & 215.29 & 7.99 & 11.99 & 4.00 & 3.99 & 107.90 \\
15 & 71.61 & 23.99 & 119.83 & 7.99 & 4.00 & 36.00 & 215.10 & 7.99 & 11.99 & 3.99 & 4.00 & 108.00 \\
16 & 72.00 & 24.00 & 120.00 & 8.00 & 4.00 & 36.00 & 215.72 & 8.00 & 12.00 & 4.00 & 4.00 & 108.00 \\
17 & 36.00 & 24.00 & 120.00 & 7.78 & 4.00 & 36.00 & 216.00 & 8.00 & 12.00 & 4.00 & 4.00 &    - \\
18 &   -   &   -   &    -   &  -   &  -   &   -   & 216.00 &  -   &   -   &  -   &  -   &    -  \\
\hline
gcd & 36 & 12 & 2 & 4 & 2 & 18 & 54 & 4 & 12 & 2 & 2 & 54  \\
\hline
\end{tabular}\vspace {3mm} \\
\end{footnotesize}
\end{center}
\end{table}

In seven of the twelve main classes (6.1, 6.2, 6.7, 6.8,
6.10, 6.11, and 6.12), the ratio $MC(6,z,x)$/$IC(6,z,x)$ is close to 6 with a few exceptions.
In the other five main classes (6.3, 6.4, 6.5, 6.6, and 6.9) this ratio
is close to 2 with a few exceptions.  These ratios are given in Table~\ref{eight}, truncated at
two decimal places.

\begin{table}
\begin{center}
\caption{Ratio of main classes of critical sets to isotopy classes of critical sets of order six \vspace {3mm}}
\label{eight}
\begin{tabular}{|c|c|c|c|c|c|c|c|c|c|c|c|c|} \hline
& $6.1$ & $6.2$ & $6.3$ & $6.4$ & $6.5$ & $6.6$ & $6.7$ & $6.8$ & $6.9$ & $6.10$ & $6.11$ & $6.12$ \\
\hline
9 &   1.00 &   -  &  -   &  -   &  -   &  -   &  -   &  -   &   -  &  -   &  -   &  - \\
10 &   -   &   -  &  -   & 1.40 & 1.66 &  -   &  -   &  -   & 1.00 & 1.00 & 3.33 &  - \\
11 &  5.63 & 5.45 & 1.42 & 1.97 & 1.98 & 1.91 &  -   & 5.44 & 1.95 & 5.85 & 5.94 & 5.71  \\
12 &  5.89 & 5.93 & 1.91 & 1.99 & 1.99 & 1.98 & 4.40 & 5.95 & 1.99 & 5.97 & 5.98 & 5.92 \\
13 &  5.95 & 5.96 & 1.97 & 1.99 & 1.99 & 1.99 & 5.38 & 5.96 & 1.99 & 5.98 & 5.99 & 5.95 \\
14 &  5.90 & 5.94 & 1.97 & 1.99 & 1.99 & 1.98 & 5.63 & 5.96 & 1.99 & 5.98 & 5.99 & 5.88 \\
15 &  5.53 & 5.88 & 1.96 & 1.99 & 1.99 & 1.95 & 5.45 & 5.88 & 1.98 & 5.95 & 5.96 & 5.69 \\
16 &  5.62 & 5.73 & 1.87 & 1.97 & 1.97 & 1.90 & 5.34 & 5.71 & 1.92 & 5.71 & 5.88 & 5.15 \\
17 &  3.00 & 3.33 & 1.75 & 1.72 & 1.75 & 1.50 & 2.92 & 5.40 & 1.00 & 5.00 & 4.50 &  - \\
18 &   -   &   -  &  -   &  -   &  -   &  -   & 3.00 &  -   &  -   &  -   &  -   &  -  \\
\hline
\end{tabular}\vspace {3mm} \\
\end{center}
\end{table}

In each main class $6.z$, GCDCS$(6,z)$ is a multiple of
2, and for those main classes with $3 \times 3$ subsquares (6.1, 6.6,
6.7 and 6.12) this number is a multiple of 18.

We also note that in the main classes of the $4 \times 4$ and $6 \times 6$ 
Latin squares, the smallest and largest possible critical sets (4 and 7
for the $4 \times 4$ case and 9 and 18 for the $6 \times 6$ case) each
have only one isotopy and main class.

This is an interesting property which we are unable to explain at
the present time.  It may be the case that the enumeration of
all critical sets of order 7 would give more insight into this
property.

\section{Observations concerning the union of critical sets}

For the $4 \times 4$ and $6 \times 6$ back-circulant Latin squares
it is possible to find four disjoint critical sets which partition
the corresponding Latin square.
This can easily be generalised to the $n \times n$ case \cite{disj}. 

If $L$ is any $6 \times 6$ Latin square it is possible to find three
disjoint critical sets of size 12 which partition $L$.  We give a
visual representation of these decompositions for Latin squares from
representatives of each of the main classes, denoted $6.1$, \dots, $6.12$.

\begin{center} 
\setlength{\tabcolsep}{4.5pt}
\begin{tabular}{ccccccc} 
${6.1}$   & $=$ & 
\begin{tabular}{|*{6}{c|}} 
\hline \ecell & \ecell & \ecell & \ecell & \ecell & \ecell \\
\hline \ecell & \ecell & \ecell & \ecell & \ecell & 1 \\
\hline \ecell & \ecell & \ecell & \ecell & 1 & 2 \\
\hline \ecell & \ecell & \ecell & 1 & 2 & 3 \\
\hline \ecell & 6 & 1 & 2 & 3 & \ecell \\
\hline 6 & \ecell & \ecell & \ecell & \ecell & 5 \\
\hline 
\end{tabular} & 
 $+$ & 
 
\begin{tabular}{|*{6}{c|}} 
\hline \ecell & \ecell & \ecell & \ecell & 5 & 6 \\
\hline 2 & \ecell & \ecell & 5 & 6 & \ecell \\ 
\hline 3 & 4 & 5 & \ecell & \ecell & \ecell \\ 
\hline 4 & \ecell & 6 & \ecell & \ecell & \ecell \\ 
\hline \ecell & \ecell & \ecell & \ecell & \ecell & 4 \\
\hline \ecell & 1 & \ecell & \ecell & \ecell & \ecell \\ 
\hline 
\end{tabular} & 
 $+$ & 
 
\begin{tabular}{|*{6}{c|}} 
\hline 1 & 2 & 3 & 4 & \ecell & \ecell \\ 
\hline \ecell & 3 & 4 & \ecell & \ecell & \ecell \\ 
\hline \ecell & \ecell & \ecell & 6 & \ecell & \ecell \\ 
\hline \ecell & 5 & \ecell & \ecell & \ecell & \ecell \\ 
\hline 5 & \ecell & \ecell & \ecell & \ecell & \ecell \\ 
\hline \ecell & \ecell & 2 & 3 & 4 & \ecell \\ 
\hline 
\end{tabular}
\\
\end{tabular} 
\end{center} 

\vspace*{1cm}

\vspace*{1mm}
\begin{center} 
\setlength{\tabcolsep}{4.5pt}
\begin{tabular}{ccccccc} 
${6.2}$   & $=$ & 
\begin{tabular}{|*{6}{c|}} 
\hline \ecell & \ecell & \ecell & \ecell & \ecell & \ecell \\
\hline \ecell & 1 & \ecell & 3 & \ecell & 5 \\
\hline \ecell & \ecell & \ecell & \ecell & 1 & 2 \\
\hline \ecell & 3 & 6 & \ecell & \ecell & \ecell \\
\hline 5 & \ecell & \ecell & 2 & \ecell & \ecell \\
\hline 6 & \ecell & 2 & \ecell & \ecell & 4 \\
\hline 
\end{tabular} & 
 $+$ & 
 
\begin{tabular}{|*{6}{c|}} 
\hline \ecell & 2 & \ecell & 4 & \ecell & 6 \\
\hline \ecell & \ecell & 4 & \ecell & \ecell & \ecell \\ 
\hline 3 & 4 & \ecell & \ecell & \ecell & \ecell \\ 
\hline \ecell & \ecell & \ecell & 5 & 2 & \ecell \\ 
\hline \ecell & \ecell & 1 & \ecell & \ecell & 3 \\
\hline \ecell & 5 & \ecell & \ecell & 3 & \ecell \\ 
\hline 
\end{tabular} & 
 $+$ & 
 
\begin{tabular}{|*{6}{c|}} 
\hline 1 & \ecell & 3 & \ecell & 5 & \ecell \\ 
\hline 2 & \ecell & \ecell & \ecell & 6 & \ecell \\ 
\hline \ecell & \ecell & 5 & 6 & \ecell & \ecell \\ 
\hline 4 & \ecell & \ecell & \ecell & \ecell & 1 \\
\hline \ecell & 6 & \ecell & \ecell & 4 & \ecell \\ 
\hline \ecell & \ecell & \ecell & 1 & \ecell & \ecell \\ 
\hline 
\end{tabular} 
\end{tabular} 
\end{center} 

\vspace*{1cm}

\vspace*{1mm}
\begin{center} 
\setlength{\tabcolsep}{4.5pt}
\begin{tabular}{ccccccc} 
${6.3}$   & $=$ & 
\begin{tabular}{|*{6}{c|}} 
\hline \ecell & \ecell & \ecell & \ecell & \ecell & 6 \\
\hline \ecell & 1 & \ecell & 3 & \ecell & \ecell \\
\hline \ecell & \ecell & 1 & \ecell & 2 & \ecell \\
\hline \ecell & \ecell & 5 & 1 & \ecell & \ecell \\
\hline 5 & \ecell & \ecell & 2 & \ecell & \ecell \\
\hline 6 & \ecell & 2 & \ecell & 4 & \ecell \\
\hline 
\end{tabular} & 
 $+$ & 
 
\begin{tabular}{|*{6}{c|}} 
\hline \ecell & 2 & \ecell & 4 & \ecell & \ecell \\ 
\hline \ecell & \ecell & \ecell & \ecell & 6 & 5 \\
\hline 3 & 5 & \ecell & \ecell & \ecell & \ecell \\ 
\hline 4 & \ecell & \ecell & \ecell & 3 & \ecell \\ 
\hline \ecell & \ecell & 6 & \ecell & 1 & \ecell \\ 
\hline \ecell & 3 & \ecell & \ecell & \ecell & 1 \\
\hline 
\end{tabular} & 
 $+$ & 
 
\begin{tabular}{|*{6}{c|}} 
\hline 1 & \ecell & 3 & \ecell & 5 & \ecell \\ 
\hline 2 & \ecell & 4 & \ecell & \ecell & \ecell \\ 
\hline \ecell & \ecell & \ecell & 6 & \ecell & 4 \\
\hline \ecell & 6 & \ecell & \ecell & \ecell & 2 \\
\hline \ecell & 4 & \ecell & \ecell & \ecell & 3 \\
\hline \ecell & \ecell & \ecell & 5 & \ecell & \ecell \\ 
\hline 
\end{tabular} 
\end{tabular} 
\end{center} 

\vspace*{10mm}

\vspace*{1mm}
\begin{center} 
\setlength{\tabcolsep}{4.5pt}
\begin{tabular}{ccccccc} 
${6.4}$   & $=$ & 
\begin{tabular}{|*{6}{c|}} 
\hline \ecell & \ecell & \ecell & \ecell & \ecell & \ecell \\
\hline \ecell & 1 & \ecell & 3 & \ecell & 5 \\
\hline \ecell & \ecell & 1 & \ecell & \ecell & 2 \\
\hline \ecell & \ecell & \ecell & 1 & \ecell & 3 \\
\hline \ecell & 4 & 6 & \ecell & \ecell & \ecell \\
\hline 6 & 3 & \ecell & 5 & \ecell & \ecell \\
\hline 
\end{tabular} & 
 $+$ & 
 
\begin{tabular}{|*{6}{c|}} 
\hline \ecell & \ecell & 3 & \ecell & \ecell & 6 \\
\hline 2 & \ecell & \ecell & \ecell & 6 & \ecell \\ 
\hline \ecell & \ecell & \ecell & 6 & 4 & \ecell \\ 
\hline 4 & \ecell & 5 & \ecell & \ecell & \ecell \\ 
\hline \ecell & \ecell & \ecell & 2 & 3 & \ecell \\ 
\hline \ecell & \ecell & \ecell & \ecell & 1 & 4 \\
\hline 
\end{tabular} & 
 $+$ & 
 
\begin{tabular}{|*{6}{c|}} 
\hline 1 & 2 & \ecell & 4 & 5 & \ecell \\ 
\hline \ecell & \ecell & 4 & \ecell & \ecell & \ecell \\ 
\hline 3 & 5 & \ecell & \ecell & \ecell & \ecell \\ 
\hline \ecell & 6 & \ecell & \ecell & 2 & \ecell \\ 
\hline 5 & \ecell & \ecell & \ecell & \ecell & 1 \\
\hline \ecell & \ecell & 2 & \ecell & \ecell & \ecell \\ 
\hline 
\end{tabular} 
\end{tabular} 
\end{center} 

\vspace*{10mm}

\begin{center} 
\setlength{\tabcolsep}{4.5pt}
\begin{tabular}{ccccccc} 
${6.5}$   & $=$ & 
\begin{tabular}{|*{6}{c|}} 
\hline \ecell & \ecell & \ecell & \ecell & \ecell & \ecell \\
\hline \ecell & 1 & \ecell & 3 & \ecell & 5 \\
\hline \ecell & \ecell & \ecell & \ecell & 1 & 2 \\
\hline \ecell & \ecell & 6 & \ecell & 3 & \ecell \\
\hline 5 & \ecell & \ecell & \ecell & \ecell & 3 \\
\hline 6 & \ecell & \ecell & 5 & \ecell & 4 \\
\hline 
\end{tabular} & 
 $+$ & 
 
\begin{tabular}{|*{6}{c|}} 
\hline \ecell & \ecell & 3 & \ecell & 5 & 6 \\
\hline 2 & \ecell & \ecell & \ecell & \ecell & \ecell \\ 
\hline \ecell & 4 & \ecell & 6 & \ecell & \ecell \\ 
\hline 4 & 5 & \ecell & \ecell & \ecell & \ecell \\ 
\hline \ecell & 6 & 2 & \ecell & \ecell & \ecell \\ 
\hline \ecell & \ecell & 1 & \ecell & 2 & \ecell \\ 
\hline 
\end{tabular} & 
 $+$ & 
 
\begin{tabular}{|*{6}{c|}} 
\hline 1 & 2 & \ecell & 4 & \ecell & \ecell \\ 
\hline \ecell & \ecell & 4 & \ecell & 6 & \ecell \\ 
\hline 3 & \ecell & 5 & \ecell & \ecell & \ecell \\ 
\hline \ecell & \ecell & \ecell & 2 & \ecell & 1 \\
\hline \ecell & \ecell & \ecell & 1 & 4 & \ecell \\ 
\hline \ecell & 3 & \ecell & \ecell & \ecell & \ecell \\ 
\hline 
\end{tabular} 
\end{tabular} 
\end{center} 

\vspace*{10mm}

\vspace*{1mm}
\begin{center} 
\setlength{\tabcolsep}{4.5pt}
\begin{tabular}{ccccccc} 
${6.6}$   & $=$ & 
\begin{tabular}{|*{6}{c|}} 
\hline \ecell & \ecell & \ecell & \ecell & \ecell & 6 \\
\hline \ecell & 1 & \ecell & 3 & \ecell & \ecell \\
\hline \ecell & \ecell & \ecell & \ecell & 1 & \ecell \\
\hline \ecell & \ecell & \ecell & 1 & 2 & 3 \\
\hline 5 & \ecell & \ecell & 2 & \ecell & \ecell \\
\hline 6 & \ecell & 2 & \ecell & 4 & \ecell \\
\hline 
\end{tabular} & 
 $+$ & 
 
\begin{tabular}{|*{6}{c|}} 
\hline \ecell & 2 & 3 & 4 & \ecell & \ecell \\ 
\hline \ecell & \ecell & \ecell & \ecell & 6 & \ecell \\ 
\hline 3 & 4 & 5 & 6 & \ecell & \ecell \\ 
\hline 4 & 5 & \ecell & \ecell & \ecell & \ecell \\ 
\hline \ecell & \ecell & \ecell & \ecell & \ecell & 4 \\
\hline \ecell & \ecell & \ecell & \ecell & \ecell & 1 \\
\hline 
\end{tabular} & 
 $+$ & 
 
\begin{tabular}{|*{6}{c|}} 
\hline 1 & \ecell & \ecell & \ecell & 5 & \ecell \\ 
\hline 2 & \ecell & 4 & \ecell & \ecell & 5 \\
\hline \ecell & \ecell & \ecell & \ecell & \ecell & 2 \\
\hline \ecell & \ecell & 6 & \ecell & \ecell & \ecell \\ 
\hline \ecell & 6 & 1 & \ecell & 3 & \ecell \\ 
\hline \ecell & 3 & \ecell & 5 & \ecell & \ecell \\ 
\hline 
\end{tabular} 
\end{tabular} 
\end{center} 

\vspace*{10mm}

\vspace*{1mm}
\begin{center} 
\setlength{\tabcolsep}{4.5pt}
\begin{tabular}{ccccccc} 
${6.7}$   & $=$ & 
\begin{tabular}{|*{6}{c|}} 
\hline \ecell & \ecell & \ecell & \ecell & \ecell & 6 \\
\hline \ecell & \ecell & 1 & 6 & \ecell & 5 \\
\hline \ecell & 1 & 2 & \ecell & \ecell & \ecell \\
\hline 4 & \ecell & \ecell & \ecell & 2 & \ecell \\
\hline 5 & \ecell & \ecell & 3 & \ecell & \ecell \\
\hline \ecell & 4 & \ecell & \ecell & 3 & \ecell \\
\hline 
\end{tabular} & 
 $+$ & 
 
\begin{tabular}{|*{6}{c|}} 
\hline \ecell & 2 & \ecell & 4 & 5 & \ecell \\ 
\hline 2 & 3 & \ecell & \ecell & \ecell & \ecell \\ 
\hline \ecell & \ecell & \ecell & 5 & \ecell & \ecell \\ 
\hline \ecell & \ecell & 6 & \ecell & \ecell & 3 \\
\hline \ecell & \ecell & 4 & \ecell & 1 & \ecell \\ 
\hline 6 & \ecell & \ecell & \ecell & \ecell & 1 \\
\hline 
\end{tabular} & 
 $+$ & 
 
\begin{tabular}{|*{6}{c|}} 
\hline 1 & \ecell & 3 & \ecell & \ecell & \ecell \\ 
\hline \ecell & \ecell & \ecell & \ecell & 4 & \ecell \\ 
\hline 3 & \ecell & \ecell & \ecell & 6 & 4 \\
\hline \ecell & 5 & \ecell & 1 & \ecell & \ecell \\ 
\hline \ecell & 6 & \ecell & \ecell & \ecell & 2 \\
\hline \ecell & \ecell & 5 & 2 & \ecell & \ecell \\ 
\hline 
\end{tabular} 
\end{tabular} 
\end{center}

\vspace*{10mm}

\vspace*{1mm}
\vspace*{1mm}
\begin{center} 
\setlength{\tabcolsep}{4.5pt}
\begin{tabular}{ccccccc} 
${6.8}$   & $=$ & 
\begin{tabular}{|*{6}{c|}} 
\hline \ecell & \ecell & \ecell & \ecell & \ecell & \ecell \\
\hline \ecell & 1 & \ecell & 3 & \ecell & 5 \\
\hline \ecell & \ecell & 1 & \ecell & 2 & 4 \\
\hline 4 & 6 & \ecell & \ecell & \ecell & \ecell \\
\hline \ecell & 3 & 6 & \ecell & \ecell & \ecell \\
\hline \ecell & \ecell & \ecell & \ecell & 3 & 2 \\
\hline 
\end{tabular} & 
 $+$ & 
 
\begin{tabular}{|*{6}{c|}} 
\hline \ecell & \ecell & 3 & 4 & 5 & \ecell \\ 
\hline 2 & \ecell & \ecell & \ecell & 6 & \ecell \\ 
\hline \ecell & 5 & \ecell & \ecell & \ecell & \ecell \\ 
\hline \ecell & \ecell & 2 & \ecell & \ecell & 3 \\
\hline 5 & \ecell & \ecell & \ecell & \ecell & 1 \\
\hline \ecell & 4 & \ecell & 1 & \ecell & \ecell \\ 
\hline 
\end{tabular} & 
 $+$ & 
 
\begin{tabular}{|*{6}{c|}} 
\hline 1 & 2 & \ecell & \ecell & \ecell & 6 \\
\hline \ecell & \ecell & 4 & \ecell & \ecell & \ecell \\ 
\hline 3 & \ecell & \ecell & 6 & \ecell & \ecell \\ 
\hline \ecell & \ecell & \ecell & 5 & 1 & \ecell \\ 
\hline \ecell & \ecell & \ecell & 2 & 4 & \ecell \\ 
\hline 6 & \ecell & 5 & \ecell & \ecell & \ecell \\ 
\hline 
\end{tabular} 
\end{tabular} 
\end{center}

\vspace*{10mm}

\vspace*{1mm}
\begin{center} 
\setlength{\tabcolsep}{4.5pt}
\begin{tabular}{ccccccc} 
${6.9}$   & $=$ & 
\begin{tabular}{|*{6}{c|}} 
\hline \ecell & \ecell & \ecell & \ecell & \ecell & \ecell \\
\hline \ecell & 1 & \ecell & 3 & \ecell & 5 \\
\hline \ecell & \ecell & 1 & \ecell & 2 & 4 \\
\hline \ecell & 6 & 2 & \ecell & \ecell & \ecell \\
\hline \ecell & 4 & \ecell & \ecell & 1 & 3 \\
\hline 6 & \ecell & \ecell & \ecell & \ecell & \ecell \\
\hline 
\end{tabular} & 
 $+$ & 
 
\begin{tabular}{|*{6}{c|}} 
\hline \ecell & \ecell & \ecell & 4 & 5 & 6 \\
\hline 2 & \ecell & \ecell & \ecell & \ecell & \ecell \\ 
\hline 3 & 5 & \ecell & \ecell & \ecell & \ecell \\ 
\hline \ecell & \ecell & \ecell & 5 & \ecell & 1 \\
\hline \ecell & \ecell & 6 & \ecell & \ecell & \ecell \\ 
\hline \ecell & 3 & 5 & \ecell & 4 & \ecell \\ 
\hline 
\end{tabular} & 
 $+$ & 
 
\begin{tabular}{|*{6}{c|}} 
\hline 1 & 2 & 3 & \ecell & \ecell & \ecell \\ 
\hline \ecell & \ecell & 4 & \ecell & 6 & \ecell \\ 
\hline \ecell & \ecell & \ecell & 6 & \ecell & \ecell \\ 
\hline 4 & \ecell & \ecell & \ecell & 3 & \ecell \\ 
\hline 5 & \ecell & \ecell & 2 & \ecell & \ecell \\ 
\hline \ecell & \ecell & \ecell & 1 & \ecell & 2 \\
\hline 
\end{tabular} 
\end{tabular} 
\end{center}

\vspace*{10mm}

\vspace*{1mm}
\begin{center} 
\setlength{\tabcolsep}{4.5pt}
\begin{tabular}{ccccccc} 
${6.10}$   & $=$ & 
\begin{tabular}{|*{6}{c|}} 
\hline \ecell & \ecell & \ecell & \ecell & \ecell & \ecell \\
\hline \ecell & 1 & \ecell & 3 & \ecell & 5 \\
\hline \ecell & \ecell & 1 & \ecell & \ecell & 2 \\
\hline \ecell & \ecell & \ecell & 1 & \ecell & 3 \\
\hline 5 & \ecell & 6 & \ecell & \ecell & \ecell \\
\hline 6 & 4 & \ecell & \ecell & 3 & \ecell \\
\hline 
\end{tabular} & 
 $+$ & 
 
\begin{tabular}{|*{6}{c|}} 
\hline \ecell & \ecell & 3 & 4 & \ecell & 6 \\
\hline 2 & \ecell & \ecell & \ecell & \ecell & \ecell \\ 
\hline \ecell & 5 & \ecell & 6 & \ecell & \ecell \\ 
\hline 4 & \ecell & \ecell & \ecell & 2 & \ecell \\ 
\hline \ecell & \ecell & \ecell & \ecell & 1 & 4 \\
\hline \ecell & \ecell & 2 & 5 & \ecell & \ecell \\ 
\hline 
\end{tabular} & 
 $+$ & 
 
\begin{tabular}{|*{6}{c|}} 
\hline 1 & 2 & \ecell & \ecell & 5 & \ecell \\ 
\hline \ecell & \ecell & 4 & \ecell & 6 & \ecell \\ 
\hline 3 & \ecell & \ecell & \ecell & 4 & \ecell \\ 
\hline \ecell & 6 & 5 & \ecell & \ecell & \ecell \\ 
\hline \ecell & 3 & \ecell & 2 & \ecell & \ecell \\ 
\hline \ecell & \ecell & \ecell & \ecell & \ecell & 1 \\
\hline 
\end{tabular} 
\end{tabular} 
\end{center}

\vspace*{10mm}

\vspace*{1mm}
\vspace*{1mm}
\begin{center} 
\setlength{\tabcolsep}{4.5pt}
\begin{tabular}{ccccccc} 
${6.11}$   & $=$ & 
\begin{tabular}{|*{6}{c|}} 
\hline \ecell & \ecell & \ecell & \ecell & \ecell & \ecell \\
\hline \ecell & 1 & \ecell & \ecell & \ecell & 3 \\
\hline \ecell & \ecell & \ecell & \ecell & 1 & 5 \\
\hline \ecell & 6 & \ecell & 2 & 3 & \ecell \\
\hline \ecell & \ecell & \ecell & 1 & 2 & 4 \\
\hline 6 & 5 & \ecell & \ecell & \ecell & \ecell \\
\hline 
\end{tabular} & 
 $+$ & 
 
\begin{tabular}{|*{6}{c|}} 
\hline \ecell & \ecell & \ecell & 4 & 5 & 6 \\
\hline 2 & \ecell & \ecell & \ecell & 6 & \ecell \\ 
\hline \ecell & \ecell & 2 & \ecell & \ecell & \ecell \\ 
\hline 4 & \ecell & 5 & \ecell & \ecell & \ecell \\ 
\hline \ecell & 3 & 6 & \ecell & \ecell & \ecell \\ 
\hline \ecell & \ecell & \ecell & 3 & 4 & \ecell \\ 
\hline 
\end{tabular} & 
 $+$ & 
 
\begin{tabular}{|*{6}{c|}} 
\hline 1 & 2 & 3 & \ecell & \ecell & \ecell \\ 
\hline \ecell & \ecell & 4 & 5 & \ecell & \ecell \\ 
\hline 3 & 4 & \ecell & 6 & \ecell & \ecell \\ 
\hline \ecell & \ecell & \ecell & \ecell & \ecell & 1 \\
\hline 5 & \ecell & \ecell & \ecell & \ecell & \ecell \\ 
\hline \ecell & \ecell & 1 & \ecell & \ecell & 2 \\
\hline 
\end{tabular} 
\end{tabular} 
\end{center} 

\vspace*{10mm}

\vspace*{1mm}
\begin{center} 
\setlength{\tabcolsep}{4.5pt}
\begin{tabular}{ccccccc} 
${6.12}$   & $=$ & 
\begin{tabular}{|*{6}{c|}} 
\hline \ecell & \ecell & \ecell & \ecell & \ecell & \ecell \\
\hline \ecell & \ecell & 1 & \ecell & \ecell & 5 \\
\hline \ecell & 1 & 2 & \ecell & 6 & \ecell \\
\hline \ecell & \ecell & \ecell & \ecell & 3 & \ecell \\
\hline \ecell & \ecell & 4 & \ecell & 2 & 3 \\
\hline 6 & 4 & 5 & \ecell & \ecell & \ecell \\
\hline 
\end{tabular} & 
 $+$ & 
 
\begin{tabular}{|*{6}{c|}} 
\hline \ecell & \ecell & 3 & \ecell & 5 & 6 \\
\hline 2 & 3 & \ecell & 6 & \ecell & \ecell \\ 
\hline \ecell & \ecell & \ecell & \ecell & \ecell & 4 \\
\hline 4 & \ecell & \ecell & \ecell & \ecell & 1 \\
\hline \ecell & 6 & \ecell & 1 & \ecell & \ecell \\ 
\hline \ecell & \ecell & \ecell & 3 & \ecell & \ecell \\ 
\hline 
\end{tabular} & 
 $+$ & 
 
\begin{tabular}{|*{6}{c|}} 
\hline 1 & 2 & \ecell & 4 & \ecell & \ecell \\ 
\hline \ecell & \ecell & \ecell & \ecell & 4 & \ecell \\ 
\hline 3 & \ecell & \ecell & 5 & \ecell & \ecell \\ 
\hline \ecell & 5 & 6 & 2 & \ecell & \ecell \\ 
\hline 5 & \ecell & \ecell & \ecell & \ecell & \ecell \\ 
\hline \ecell & \ecell & \ecell & \ecell & 1 & 2 \\
\hline 
\end{tabular} 
\end{tabular} 
\end{center} 
\chapter{Conclusion}\label{ch9}
In this thesis, we have developed several new results.  The algorithms
developed in Chapter~\ref{ch3} have been used throughout the remainder of the
thesis to discover new bounds and existence results on the sizes of
possible critical sets.

We have proved a new bound on \lcs{n} in Chapter~\ref{ch4}, and given many
examples of critical sets of sizes not previously known.

In Chapter~\ref{ch5}, a new bound was given on the maximum number of intercalates
in Latin squares of orders $2^\alpha m$ and $2^\alpha m+1$ for $\alpha
\geq 2$ and $m$ odd ($\alpha \neq 3$ in the $2^\alpha m+1$ case).  Also, a
critical set in Latin squares of order $4m$ was given, and large critical
sets in intercalate-rich Latin squares of orders 11 and 14 were examined.

We have completed the spectrum of critical sets between the bounds
conjectured by Nelder ($\displaystyle{\frac{n^2}{4}}$ and $\displaystyle{\frac{n^2-n}{2}}$).
This result was given in Chapter~\ref{ch6}.

In Chapter~\ref{ch7}, we looked at all trades of volume between 6 and 9, and
determined which of the corresponding partial Steiner Latin squares were
decomposable into disjoint Latin interchanges.

In Chapter~\ref{ch8}, the new bound on lcs$(n)$ from Chapter~\ref{ch4} was used to
reduce the search space of critical sets for Latin squares of order 6,
and all the critical sets in Latin squares of order at most six were
then determined.

Further research could include determining those values $s >
\displaystyle{\frac{n^{2}-n}{2}}$ for which there exists a critical set of order $n$
and size $s$.  Instead of looking at the maximum number of intercalates,
we could look at the maximum number of $m \times m$ subsquares, $m > 
2$, in a Latin square of given order.

The questions proposed at the conclusion to Chapter~\ref{ch4} also deserve more
research.  That is, is there a relationship between critical sets of 
size \lcs{n} and Latin
squares with $I(n)$ intercalates, and must critical sets with \lcs{n}
entries have a missing row, column, and symbol?

Enumerating the critical sets for the Latin squares of order 7 is not
possible with current computer hardware and algorithms, but may become
possible in the future.  This would settle the question of what \lcs7
is, and could possibly provide more information to help prove new bounds
on \scs{n}.  Such an enumeration might also give more information to
assist in understanding the patterns noted in Chapter~\ref{ch8}.

A further idea for research based on the results of Chapter~\ref{ch8} is to check
the critical sets of size 17 and order 6 for a possible construction
for a critical set of size $\displaystyle{\frac{n^2-n}{2}} + 2$ for $n \geq 6$.

The results of this thesis could possibly also be used in music
composition, as the 20th century composers Karlheinz Stockhausen
\cite{stock} and Peter Maxwell Davies \cite{pmd} have used Latin squares
extensively in their compositions.

\clearpage
\iflatexml
  \appendix
  \renewcommand{\thechapter}{\arabic{chapter}}
  \setcounter{chapter}{0}
\else
  \setcounter{chapter}{0}
  \makeatletter
  \renewcommand\theHchapter{app.\arabic{chapter}}
  \makeatother
\fi


\iflatexml
  \chapter{Examples of critical sets}\label{app1}
\else
  \refstepcounter{chapter}
  \chapter*{Appendix \thechapter\newline\newline Examples of critical sets}\label{app1}
  \addcontentsline{toc}{chapter}{Appendices}
  \addcontentsline{toc}{chapter}{\hspace{1.5em}Appendix~\thechapter\quad Examples of critical sets}
  \vspace{0.5cm}
\fi

Here we give some examples of large
critical sets for $n = 5, 7$, $9$, and $10$.  By combining the results
of the two papers ~\cite{MR1758263} and ~\cite{MR99k:05038}, and this appendix,
we can show the existence of critical sets of all sizes between
$\lfloor \displaystyle{\frac{n^2}{4}} \rfloor$ and the current upper bound for 
\lcs{n}
for $1 \leq n \leq 10$.

\noindent
A critical set of order 5 and size 11:

\begin{center}
\begin{latinsq}{5}
\hline \ecell & \ecell & \ecell & \ecell & \ecell \\
\hline  2& \ecell &4&5& \ecell \\
\hline \ecell & \ecell &1&2& \ecell \\
\hline \ecell & \ecell &5&1&2 \\
\hline  5& \ecell &2& \ecell &1 \\
\hline
\end{latinsq}
\end{center}

\noindent
A critical set of order 7 and size 25:

\begin{center}
\begin{latinsq}{7}
\hline \ecell &3&2&7& \ecell &5& \ecell \\
\hline \ecell & \ecell & \ecell & \ecell & \ecell & \ecell & \ecell \\
\hline \ecell & \ecell &3&5&4&7&6 \\
\hline \ecell &6&5&4&3&2& \ecell \\
\hline \ecell & \ecell &4&3&5& \ecell & \ecell \\
\hline \ecell &4&7&2& \ecell &6&3 \\
\hline \ecell & \ecell &6& \ecell & \ecell &3&7 \\
\hline
\end{latinsq}
\end{center}



\newpage
\noindent
A critical set of order 9 and size 40:

\begin{center}
\begin{latinsq}{9}
\hline  \ecell & \ecell & \ecell & \ecell & \ecell & \ecell & \ecell & \ecell & \ecell \\
\hline  \ecell & \ecell &6&9&1&8&2&4& \ecell \\
\hline 9& \ecell &4& \ecell &2& \ecell & \ecell & \ecell & \ecell \\
\hline  \ecell & \ecell &8&2& \ecell &1& \ecell & \ecell & \ecell \\
\hline  \ecell & \ecell &1& \ecell &6&5& \ecell &2& \ecell \\
\hline 7& \ecell &2&8&4&9& \ecell & \ecell & \ecell \\
\hline 8& \ecell &7&5&9&4& \ecell &1&2 \\
\hline 4& \ecell &9& \ecell &8&2&6&5&7 \\
\hline 2& \ecell &5& \ecell & \ecell & \ecell &4&9&8 \\
\hline
\end{latinsq}
\end{center}

\noindent
A critical set of order 9 and size 41:

\begin{center}
\begin{latinsq}{9}
\hline  \ecell & \ecell & \ecell & \ecell & \ecell & \ecell & \ecell & \ecell & \ecell \\
\hline  \ecell &7& \ecell & \ecell &1& \ecell & \ecell & \ecell &3 \\
\hline  \ecell &8&4& \ecell & \ecell & \ecell & \ecell & \ecell &1 \\
\hline  \ecell &4& \ecell & \ecell &3&1&9& \ecell &5 \\
\hline  \ecell &9&1&7&6& \ecell &8& \ecell &4 \\
\hline  \ecell &5& \ecell &8&4&9&1& \ecell &6 \\
\hline  \ecell &6& \ecell & \ecell & \ecell & \ecell &3&1& \ecell \\
\hline  \ecell &3&9&1&8& \ecell &6&5&7 \\
\hline  \ecell &1&5&6&7&3&4&9&8 \\
\hline
\end{latinsq}
\end{center}

\noindent
A critical set of order 9 and size 42:

\begin{center}
\begin{latinsq}{9}
\hline  \ecell & \ecell &3& \ecell &5&6&7&8&9 \\
\hline  \ecell & \ecell & \ecell & \ecell & \ecell & \ecell & \ecell & \ecell & \ecell \\
\hline  \ecell &8&4& \ecell & \ecell & \ecell &5&6& \ecell \\
\hline  \ecell &4& \ecell & \ecell &3& \ecell &9&7& \ecell \\
\hline  \ecell &9& \ecell & \ecell &6&5&8& \ecell & \ecell \\
\hline  \ecell &5& \ecell &8&4&9&1&3&6 \\
\hline  \ecell &6& \ecell &5& \ecell & \ecell &3& \ecell & \ecell \\
\hline  \ecell &3&9&1&8& \ecell &6&5& \ecell \\
\hline  \ecell &1&5&6&7&3&4&9&8 \\
\hline
\end{latinsq}
\end{center}

\noindent
A critical set of order 9 and size 43:

\begin{center}
\begin{latinsq}{9}
\hline  \ecell &3& \ecell & \ecell &9& \ecell &4&6& \ecell \\
\hline  \ecell & \ecell &1&9&8& \ecell &6& \ecell & \ecell \\
\hline  \ecell & \ecell & \ecell & \ecell & \ecell & \ecell & \ecell & \ecell & \ecell \\
\hline 7&9&8&4&6& \ecell &1&3&2 \\
\hline 9&8& \ecell & \ecell & \ecell & \ecell &3&2&1 \\
\hline 8& \ecell &9& \ecell &4& \ecell &2&1& \ecell \\
\hline  \ecell &6& \ecell &1&3& \ecell &7&9&8 \\
\hline  \ecell & \ecell & \ecell & \ecell &2& \ecell &9&8&7 \\
\hline  \ecell &4&6&2&1& \ecell &8&7&9 \\
\hline
\end{latinsq}
\end{center}

\noindent
A critical set of order 9 and size 44:

\begin{center}
\begin{latinsq}{9}
\hline \ecell & \ecell & \ecell & \ecell & \ecell & \ecell & \ecell & \ecell & \ecell \\
\hline \ecell &1& \ecell &3& \ecell &5& \ecell & \ecell &7 \\
\hline \ecell & \ecell &1&2& \ecell & \ecell & \ecell &6&5 \\
\hline \ecell &3&2&1& \ecell & \ecell &6&5&8 \\
\hline \ecell & \ecell & \ecell & \ecell &1& \ecell &2&3&4 \\
\hline \ecell &5& \ecell & \ecell &2&1&4&7&3 \\
\hline \ecell & \ecell & \ecell &5&3&2&1&4&6 \\
\hline \ecell & \ecell &6&7&4&3& \ecell &1&2 \\
\hline \ecell &7&5&6&8&4&3&2&1 \\
\hline
\end{latinsq}
\end{center}


\noindent
A critical set of order 10 and size 56:

\begin{center}
\begin{latinsq}[10]{10}
\hline  \ecell & \ecell & \ecell & \ecell & \ecell & \ecell & \ecell & \ecell & \ecell & \ecell \\
\hline  \ecell &1&4&3&6&5&8&7&10&9 \\
\hline  \ecell &4&1& \ecell & \ecell & \ecell &5& \ecell &6&8 \\
\hline  \ecell &3& \ecell &1& \ecell &8& \ecell & \ecell &7&5 \\
\hline  \ecell &6& \ecell & \ecell &1& \ecell &3&9& \ecell &4 \\
\hline  \ecell &5& \ecell & \ecell & \ecell &1&10&4&3& \ecell \\
\hline  \ecell &8&5& \ecell &3&10&1& \ecell &4&6 \\
\hline  \ecell &7& \ecell &6&9& \ecell & \ecell &1&5&3 \\
\hline  \ecell &10&6&7& \ecell &3&4&5&1& \ecell \\
\hline  \ecell &9&8&5&4& \ecell &6&3& \ecell &1 \\
\hline
\end{latinsq}
\end{center}

\noindent
A critical set of order 10 and size 57:

\begin{center}
\begin{latinsq}[10]{10}
\hline   \ecell & \ecell & \ecell & \ecell & \ecell & \ecell & \ecell & \ecell & \ecell & \ecell \\
\hline   \ecell &   1 & \ecell &   3 & \ecell &   5 & \ecell &   7 & \ecell &   9 \\
\hline   \ecell & \ecell &   1 &   2 & \ecell & \ecell &   5 & \ecell &   6 &   8 \\
\hline   \ecell &   3 &   2 &   1 & \ecell & \ecell &   9 &   6 &   7 &   5 \\
\hline   \ecell & \ecell & \ecell & \ecell &   1 &   2 &   3 & \ecell &   8 &   4 \\
\hline   \ecell &   5 & \ecell & \ecell &   2 &   1 &  10 &   4 &   3 & \ecell \\
\hline   \ecell & \ecell &   5 &   9 &   3 &  10 &   1 &   2 & \ecell &   6 \\
\hline   \ecell &   7 & \ecell &   6 & \ecell &   4 &   2 &   1 &   5 &   3 \\
\hline   \ecell & \ecell &   6 &   7 &   8 &   3 & \ecell &   5 &   1 &   2 \\
\hline   \ecell &   9 &   8 &   5 &   4 & \ecell &   6 &   3 &   2 &   1 \\
\hline
\end{latinsq}
\end{center}



\iflatexml
  \chapter{Critical sets in Latin squares of order~6}\label{app2}
\else
  \refstepcounter{chapter}
  \chapter*{Appendix \thechapter\newline\newline Critical sets in Latin squares of order~6}\label{app2}
  \addcontentsline{toc}{chapter}{\hspace{1.5em}Appendix~\thechapter\quad Critical sets in Latin squares of order~6}
  \vspace{0.5cm}
\fi

This appendix is associated with Chapter~\ref{ch8} and gives examples of 
critical sets of all possible sizes in each of the 12 main classes,
denoted 6.1 to 6.12, of $6 \times 6$ Latin squares.

\begin{footnotesize}
\begin{center}
\begin{tabular}{c@{\hspace{1.0mm}}c@{\hspace{1.0mm}}c@{\hspace{1.0mm}}c@{\hspace{1.0mm}}}
\begin{tabular}{|c|c|c|c|c|c|}
\hline 1 & 2 & 3 & \ecell & \ecell & \ecell \\
\hline 2 & 3 & \ecell & \ecell & \ecell & \ecell \\
\hline 3 & \ecell & \ecell & \ecell & \ecell & \ecell \\
\hline \ecell & \ecell & \ecell & \ecell & \ecell & \ecell \\
\hline \ecell & \ecell & \ecell & \ecell & \ecell & 4 \\
\hline \ecell & \ecell & \ecell & \ecell & 4 & 5 \\
\hline 
\end{tabular} &
\begin{tabular}{|c|c|c|c|c|c|}
\hline \ecell & \ecell & \ecell & \ecell & \ecell & \ecell \\
\hline \ecell & \ecell & \ecell & \ecell & \ecell & 1 \\
\hline \ecell & \ecell & \ecell & \ecell & 1 & 2 \\
\hline \ecell & \ecell & \ecell & 1 & 2 & 3 \\
\hline \ecell & \ecell & 1 & 2 & 3 & 4 \\
\hline 6 & \ecell & \ecell & \ecell & \ecell & \ecell \\
\hline 
\end{tabular} &
\begin{tabular}{|c|c|c|c|c|c|}
\hline \ecell & \ecell & \ecell & \ecell & \ecell & \ecell \\
\hline \ecell & \ecell & \ecell & \ecell & \ecell & 1 \\
\hline \ecell & \ecell & \ecell & \ecell & 1 & 2 \\
\hline \ecell & \ecell & \ecell & 1 & 2 & 3 \\
\hline \ecell & 6 & 1 & 2 & 3 & \ecell \\
\hline 6 & \ecell & \ecell & \ecell & \ecell & 5 \\
\hline 
\end{tabular} &
\begin{tabular}{|c|c|c|c|c|c|}
\hline \ecell & \ecell & \ecell & \ecell & \ecell & \ecell \\
\hline \ecell & \ecell & \ecell & \ecell & \ecell & 1 \\
\hline \ecell & \ecell & \ecell & \ecell & 1 & 2 \\
\hline \ecell & \ecell & 6 & 1 & 2 & \ecell \\
\hline \ecell & 6 & 1 & 2 & 3 & \ecell \\
\hline \ecell & \ecell & \ecell & 3 & 4 & 5 \\
\hline 
\end{tabular}\\ 
6.1, size 9
 & 6.1, size 11
 & 6.1, size 12
 & 6.1, size 13
\\
\end{tabular} 
\end{center} 
\begin{center}
\begin{tabular}{c@{\hspace{1.0mm}}c@{\hspace{1.0mm}}c@{\hspace{1.0mm}}c@{\hspace{1.0mm}}}
\begin{tabular}{|c|c|c|c|c|c|}
\hline \ecell & \ecell & \ecell & \ecell & \ecell & \ecell \\
\hline \ecell & \ecell & 1 & \ecell & \ecell & 4 \\
\hline \ecell & 1 & 2 & \ecell & 4 & 5 \\
\hline \ecell & \ecell & \ecell & \ecell & \ecell & 1 \\
\hline \ecell & \ecell & 4 & 3 & 1 & 2 \\
\hline \ecell & 4 & 5 & \ecell & 2 & \ecell \\
\hline 
\end{tabular} &
\begin{tabular}{|c|c|c|c|c|c|}
\hline \ecell & \ecell & \ecell & \ecell & \ecell & \ecell \\
\hline \ecell & \ecell & \ecell & \ecell & \ecell & 1 \\
\hline \ecell & \ecell & \ecell & \ecell & 1 & 2 \\
\hline \ecell & \ecell & \ecell & 1 & 2 & 3 \\
\hline \ecell & \ecell & 1 & 2 & 3 & 4 \\
\hline \ecell & 1 & 2 & 3 & 4 & 5 \\
\hline 
\end{tabular} &
\begin{tabular}{|c|c|c|c|c|c|}
\hline \ecell & \ecell & \ecell & \ecell & \ecell & \ecell \\
\hline \ecell & \ecell & \ecell & \ecell & \ecell & 1 \\
\hline \ecell & \ecell & 5 & \ecell & 1 & 2 \\
\hline \ecell & 5 & 6 & 1 & 2 & \ecell \\
\hline \ecell & 6 & 1 & 2 & \ecell & 4 \\
\hline \ecell & 1 & 2 & 3 & \ecell & 5 \\
\hline 
\end{tabular} &
\begin{tabular}{|c|c|c|c|c|c|}
\hline \ecell & \ecell & \ecell & \ecell & \ecell & \ecell \\
\hline \ecell & \ecell & \ecell & 5 & 6 & 1 \\
\hline \ecell & \ecell & 5 & \ecell & 1 & 2 \\
\hline \ecell & 5 & 6 & 1 & 2 & 3 \\
\hline \ecell & 6 & 1 & \ecell & 3 & \ecell \\
\hline \ecell & 1 & 2 & 3 & \ecell & \ecell \\
\hline 
\end{tabular}\\ 
6.1, size 14
 & 6.1, size 15
 & 6.1, size 16
 & 6.1, size 17
\\
\end{tabular} 
\end{center} 
\begin{center}
\begin{tabular}{c@{\hspace{1.0mm}}c@{\hspace{1.0mm}}c@{\hspace{1.0mm}}c@{\hspace{1.0mm}}}
\begin{tabular}{|c|c|c|c|c|c|}
\hline \ecell & \ecell & \ecell & \ecell & \ecell & \ecell \\
\hline \ecell & 1 & \ecell & 3 & \ecell & 5 \\
\hline \ecell & \ecell & \ecell & 6 & 1 & \ecell \\
\hline \ecell & 3 & 6 & \ecell & 2 & \ecell \\
\hline 5 & \ecell & \ecell & \ecell & \ecell & \ecell \\
\hline \ecell & \ecell & 2 & \ecell & \ecell & 4 \\
\hline 
\end{tabular} &
\begin{tabular}{|c|c|c|c|c|c|}
\hline \ecell & \ecell & \ecell & \ecell & \ecell & \ecell \\
\hline \ecell & 1 & \ecell & 3 & \ecell & 5 \\
\hline \ecell & \ecell & \ecell & \ecell & 1 & 2 \\
\hline \ecell & 3 & 6 & \ecell & \ecell & \ecell \\
\hline 5 & \ecell & \ecell & 2 & \ecell & \ecell \\
\hline 6 & \ecell & 2 & \ecell & \ecell & 4 \\
\hline 
\end{tabular} &
\begin{tabular}{|c|c|c|c|c|c|}
\hline \ecell & \ecell & \ecell & \ecell & \ecell & \ecell \\
\hline \ecell & 1 & \ecell & 3 & \ecell & 5 \\
\hline \ecell & \ecell & \ecell & \ecell & 1 & 2 \\
\hline \ecell & 3 & \ecell & 5 & 2 & \ecell \\
\hline \ecell & 6 & 1 & \ecell & 4 & \ecell \\
\hline \ecell & \ecell & 2 & \ecell & \ecell & 4 \\
\hline 
\end{tabular} &
\begin{tabular}{|c|c|c|c|c|c|}
\hline \ecell & \ecell & \ecell & \ecell & \ecell & \ecell \\
\hline \ecell & 1 & \ecell & 3 & \ecell & 5 \\
\hline \ecell & \ecell & \ecell & \ecell & 1 & 2 \\
\hline \ecell & 3 & \ecell & 5 & 2 & \ecell \\
\hline \ecell & \ecell & 1 & 2 & 4 & \ecell \\
\hline \ecell & 5 & 2 & \ecell & \ecell & 4 \\
\hline 
\end{tabular}\\ 
6.2, size 11
 & 6.2, size 12
 & 6.2, size 13
 & 6.2, size 14
\\
\end{tabular} 
\end{center} 
\begin{center}
\begin{tabular}{c@{\hspace{1.0mm}}c@{\hspace{1.0mm}}c@{\hspace{1.0mm}}c@{\hspace{1.0mm}}}
\begin{tabular}{|c|c|c|c|c|c|}
\hline \ecell & \ecell & \ecell & \ecell & \ecell & \ecell \\
\hline \ecell & 1 & \ecell & 3 & \ecell & 5 \\
\hline \ecell & \ecell & \ecell & \ecell & 1 & 2 \\
\hline \ecell & 3 & \ecell & 5 & 2 & \ecell \\
\hline \ecell & \ecell & 1 & 2 & 4 & \ecell \\
\hline 6 & \ecell & 2 & 1 & 3 & \ecell \\
\hline 
\end{tabular} &
\begin{tabular}{|c|c|c|c|c|c|}
\hline \ecell & \ecell & \ecell & \ecell & \ecell & \ecell \\
\hline \ecell & 1 & \ecell & 3 & \ecell & 5 \\
\hline \ecell & \ecell & \ecell & \ecell & 1 & 2 \\
\hline \ecell & 3 & 6 & 5 & 2 & 1 \\
\hline \ecell & 6 & 1 & \ecell & \ecell & \ecell \\
\hline \ecell & 5 & 2 & 1 & 3 & \ecell \\
\hline 
\end{tabular} &
\begin{tabular}{|c|c|c|c|c|c|}
\hline \ecell & \ecell & \ecell & \ecell & \ecell & \ecell \\
\hline \ecell & 1 & \ecell & 3 & \ecell & 5 \\
\hline \ecell & \ecell & \ecell & \ecell & 1 & 2 \\
\hline \ecell & 3 & \ecell & 5 & 2 & 1 \\
\hline \ecell & \ecell & 1 & 2 & \ecell & 3 \\
\hline \ecell & 5 & 2 & 1 & 3 & 4 \\
\hline 
\end{tabular} &
\begin{tabular}{|c|c|c|c|c|c|}
\hline \ecell & \ecell & \ecell & \ecell & 5 & 6 \\
\hline \ecell & 1 & \ecell & 3 & \ecell & \ecell \\
\hline \ecell & \ecell & 1 & \ecell & \ecell & \ecell \\
\hline 4 & \ecell & \ecell & \ecell & \ecell & 2 \\
\hline \ecell & 4 & \ecell & \ecell & 1 & \ecell \\
\hline 6 & \ecell & 2 & \ecell & \ecell & \ecell \\
\hline 
\end{tabular}\\ 
6.2, size 15
 & 6.2, size 16
 & 6.2, size 17
 & 6.3, size 11
\\
\end{tabular} 
\end{center} 
\begin{center}
\begin{tabular}{c@{\hspace{1.0mm}}c@{\hspace{1.0mm}}c@{\hspace{1.0mm}}c@{\hspace{1.0mm}}}
\begin{tabular}{|c|c|c|c|c|c|}
\hline \ecell & \ecell & \ecell & \ecell & \ecell & 6 \\
\hline \ecell & 1 & \ecell & 3 & \ecell & \ecell \\
\hline \ecell & \ecell & 1 & \ecell & 2 & \ecell \\
\hline \ecell & \ecell & 5 & 1 & \ecell & \ecell \\
\hline 5 & \ecell & \ecell & 2 & \ecell & \ecell \\
\hline 6 & \ecell & 2 & \ecell & 4 & \ecell \\
\hline 
\end{tabular} &
\begin{tabular}{|c|c|c|c|c|c|}
\hline \ecell & \ecell & \ecell & \ecell & \ecell & 6 \\
\hline \ecell & 1 & \ecell & 3 & \ecell & \ecell \\
\hline \ecell & \ecell & 1 & \ecell & 2 & \ecell \\
\hline \ecell & \ecell & \ecell & 1 & 3 & 2 \\
\hline 5 & 4 & \ecell & \ecell & \ecell & \ecell \\
\hline \ecell & 3 & \ecell & 5 & 4 & \ecell \\
\hline 
\end{tabular} &
\begin{tabular}{|c|c|c|c|c|c|}
\hline \ecell & \ecell & \ecell & \ecell & \ecell & 6 \\
\hline \ecell & 1 & \ecell & 3 & \ecell & \ecell \\
\hline \ecell & \ecell & 1 & \ecell & 2 & \ecell \\
\hline \ecell & \ecell & \ecell & 1 & 3 & 2 \\
\hline \ecell & 4 & \ecell & \ecell & 1 & 3 \\
\hline 6 & \ecell & 2 & \ecell & 4 & \ecell \\
\hline 
\end{tabular} &
\begin{tabular}{|c|c|c|c|c|c|}
\hline \ecell & \ecell & \ecell & \ecell & \ecell & \ecell \\
\hline \ecell & 1 & \ecell & 3 & \ecell & 5 \\
\hline \ecell & \ecell & 1 & \ecell & 2 & 4 \\
\hline \ecell & \ecell & \ecell & 1 & 3 & 2 \\
\hline \ecell & \ecell & 6 & 2 & 1 & \ecell \\
\hline 6 & \ecell & 2 & \ecell & 4 & \ecell \\
\hline 
\end{tabular}\\ 
6.3, size 12
 & 6.3, size 13
 & 6.3, size 14
 & 6.3, size 15
\\
\end{tabular} 
\end{center} 
\begin{center}
\begin{tabular}{c@{\hspace{1.0mm}}c@{\hspace{1.0mm}}c@{\hspace{1.0mm}}c@{\hspace{1.0mm}}}
\begin{tabular}{|c|c|c|c|c|c|}
\hline \ecell & \ecell & \ecell & \ecell & \ecell & \ecell \\
\hline \ecell & 1 & \ecell & 3 & \ecell & 5 \\
\hline \ecell & \ecell & 1 & \ecell & 2 & 4 \\
\hline \ecell & \ecell & \ecell & 1 & 3 & 2 \\
\hline \ecell & \ecell & 6 & 2 & 1 & \ecell \\
\hline \ecell & 3 & \ecell & 5 & 4 & 1 \\
\hline 
\end{tabular} &
\begin{tabular}{|c|c|c|c|c|c|}
\hline \ecell & \ecell & \ecell & \ecell & \ecell & \ecell \\
\hline \ecell & 1 & \ecell & 3 & \ecell & 5 \\
\hline \ecell & \ecell & 1 & \ecell & 2 & 4 \\
\hline \ecell & \ecell & \ecell & 1 & 3 & 2 \\
\hline \ecell & 4 & \ecell & \ecell & 1 & 3 \\
\hline \ecell & 3 & 2 & 5 & 4 & 1 \\
\hline 
\end{tabular} &
\begin{tabular}{|c|c|c|c|c|c|}
\hline \ecell & \ecell & \ecell & 4 & \ecell & 6 \\
\hline \ecell & 1 & \ecell & \ecell & \ecell & \ecell \\
\hline \ecell & 5 & 1 & \ecell & \ecell & \ecell \\
\hline 4 & \ecell & \ecell & \ecell & \ecell & 3 \\
\hline \ecell & \ecell & \ecell & \ecell & \ecell & \ecell \\
\hline 6 & \ecell & 2 & \ecell & \ecell & 4 \\
\hline 
\end{tabular} &
\begin{tabular}{|c|c|c|c|c|c|}
\hline \ecell & \ecell & \ecell & \ecell & \ecell & \ecell \\
\hline \ecell & 1 & \ecell & 3 & \ecell & 5 \\
\hline \ecell & \ecell & 1 & \ecell & \ecell & 2 \\
\hline \ecell & \ecell & \ecell & 1 & \ecell & 3 \\
\hline \ecell & 4 & 6 & \ecell & \ecell & \ecell \\
\hline 6 & \ecell & 2 & \ecell & \ecell & \ecell \\
\hline 
\end{tabular}\\ 
6.3, size 16
 & 6.3, size 17
 & 6.4, size 10
 & 6.4, size 11
\\
\end{tabular} 
\end{center} 
\begin{center}
\begin{tabular}{c@{\hspace{1.0mm}}c@{\hspace{1.0mm}}c@{\hspace{1.0mm}}c@{\hspace{1.0mm}}}
\begin{tabular}{|c|c|c|c|c|c|}
\hline \ecell & \ecell & \ecell & \ecell & \ecell & \ecell \\
\hline \ecell & 1 & \ecell & 3 & \ecell & 5 \\
\hline \ecell & \ecell & 1 & \ecell & \ecell & 2 \\
\hline \ecell & \ecell & \ecell & 1 & \ecell & 3 \\
\hline \ecell & 4 & 6 & \ecell & \ecell & \ecell \\
\hline 6 & 3 & \ecell & 5 & \ecell & \ecell \\
\hline 
\end{tabular} &
\begin{tabular}{|c|c|c|c|c|c|}
\hline \ecell & \ecell & \ecell & \ecell & \ecell & \ecell \\
\hline \ecell & 1 & \ecell & 3 & \ecell & 5 \\
\hline \ecell & \ecell & 1 & \ecell & \ecell & 2 \\
\hline \ecell & \ecell & \ecell & 1 & \ecell & 3 \\
\hline \ecell & \ecell & 6 & 2 & 3 & \ecell \\
\hline 6 & 3 & \ecell & \ecell & 1 & \ecell \\
\hline 
\end{tabular} &
\begin{tabular}{|c|c|c|c|c|c|}
\hline \ecell & \ecell & \ecell & \ecell & \ecell & \ecell \\
\hline \ecell & 1 & \ecell & 3 & \ecell & 5 \\
\hline \ecell & \ecell & 1 & \ecell & \ecell & 2 \\
\hline \ecell & \ecell & \ecell & 1 & \ecell & 3 \\
\hline \ecell & \ecell & 6 & 2 & 3 & \ecell \\
\hline \ecell & 3 & \ecell & 5 & 1 & 4 \\
\hline 
\end{tabular} &
\begin{tabular}{|c|c|c|c|c|c|}
\hline \ecell & \ecell & \ecell & \ecell & \ecell & \ecell \\
\hline \ecell & 1 & \ecell & 3 & \ecell & 5 \\
\hline \ecell & \ecell & 1 & \ecell & \ecell & 2 \\
\hline \ecell & \ecell & \ecell & 1 & \ecell & 3 \\
\hline \ecell & 4 & 6 & 2 & 3 & \ecell \\
\hline \ecell & 3 & 2 & \ecell & 1 & 4 \\
\hline 
\end{tabular}\\ 
6.4, size 12
 & 6.4, size 13
 & 6.4, size 14
 & 6.4, size 15
\\
\end{tabular} 
\end{center} 
\begin{center}
\begin{tabular}{c@{\hspace{1.0mm}}c@{\hspace{1.0mm}}c@{\hspace{1.0mm}}c@{\hspace{1.0mm}}}
\begin{tabular}{|c|c|c|c|c|c|}
\hline \ecell & \ecell & \ecell & \ecell & \ecell & \ecell \\
\hline \ecell & 1 & \ecell & 3 & \ecell & 5 \\
\hline \ecell & \ecell & 1 & \ecell & \ecell & 2 \\
\hline \ecell & \ecell & \ecell & 1 & \ecell & 3 \\
\hline \ecell & 4 & \ecell & 2 & 3 & 1 \\
\hline \ecell & 3 & 2 & 5 & 1 & 4 \\
\hline 
\end{tabular} &
\begin{tabular}{|c|c|c|c|c|c|}
\hline \ecell & \ecell & \ecell & \ecell & \ecell & \ecell \\
\hline \ecell & 1 & \ecell & 3 & \ecell & 5 \\
\hline 3 & \ecell & 1 & \ecell & \ecell & 2 \\
\hline 4 & \ecell & 5 & 1 & \ecell & 3 \\
\hline 5 & 4 & \ecell & 2 & \ecell & \ecell \\
\hline \ecell & 3 & 2 & 5 & \ecell & 4 \\
\hline 
\end{tabular} &
\begin{tabular}{|c|c|c|c|c|c|}
\hline \ecell & \ecell & \ecell & \ecell & \ecell & 6 \\
\hline \ecell & 1 & \ecell & 3 & \ecell & 5 \\
\hline \ecell & \ecell & 5 & \ecell & \ecell & \ecell \\
\hline \ecell & 5 & \ecell & \ecell & \ecell & 1 \\
\hline \ecell & \ecell & 2 & \ecell & 4 & \ecell \\
\hline \ecell & \ecell & \ecell & \ecell & 2 & \ecell \\
\hline 
\end{tabular} &
\begin{tabular}{|c|c|c|c|c|c|}
\hline \ecell & \ecell & \ecell & \ecell & \ecell & \ecell \\
\hline \ecell & 1 & \ecell & 3 & \ecell & 5 \\
\hline \ecell & \ecell & \ecell & \ecell & 1 & 2 \\
\hline \ecell & 5 & \ecell & \ecell & \ecell & \ecell \\
\hline \ecell & \ecell & 2 & \ecell & 4 & \ecell \\
\hline \ecell & \ecell & 1 & \ecell & 2 & 4 \\
\hline 
\end{tabular}\\ 
6.4, size 16
 & 6.4, size 17
 & 6.5, size 10
 & 6.5, size 11
\\
\end{tabular} 
\end{center} 
\begin{center}
\begin{tabular}{c@{\hspace{1.0mm}}c@{\hspace{1.0mm}}c@{\hspace{1.0mm}}c@{\hspace{1.0mm}}}
\begin{tabular}{|c|c|c|c|c|c|}
\hline \ecell & \ecell & \ecell & \ecell & \ecell & \ecell \\
\hline \ecell & 1 & \ecell & 3 & \ecell & 5 \\
\hline \ecell & \ecell & \ecell & \ecell & 1 & 2 \\
\hline \ecell & \ecell & \ecell & 2 & 3 & 1 \\
\hline 5 & 6 & \ecell & \ecell & \ecell & \ecell \\
\hline 6 & \ecell & \ecell & \ecell & 2 & \ecell \\
\hline 
\end{tabular} &
\begin{tabular}{|c|c|c|c|c|c|}
\hline \ecell & \ecell & \ecell & \ecell & \ecell & \ecell \\
\hline \ecell & 1 & \ecell & 3 & \ecell & 5 \\
\hline \ecell & \ecell & \ecell & \ecell & 1 & 2 \\
\hline \ecell & \ecell & \ecell & 2 & 3 & 1 \\
\hline \ecell & \ecell & 2 & 1 & 4 & 3 \\
\hline 6 & \ecell & \ecell & \ecell & \ecell & \ecell \\
\hline 
\end{tabular} &
\begin{tabular}{|c|c|c|c|c|c|}
\hline \ecell & \ecell & \ecell & \ecell & \ecell & \ecell \\
\hline \ecell & 1 & \ecell & 3 & \ecell & 5 \\
\hline \ecell & \ecell & \ecell & \ecell & 1 & 2 \\
\hline \ecell & \ecell & \ecell & 2 & 3 & 1 \\
\hline 5 & \ecell & \ecell & \ecell & 4 & 3 \\
\hline 6 & 3 & \ecell & \ecell & 2 & \ecell \\
\hline 
\end{tabular} &
\begin{tabular}{|c|c|c|c|c|c|}
\hline \ecell & \ecell & \ecell & \ecell & \ecell & \ecell \\
\hline \ecell & 1 & \ecell & 3 & \ecell & 5 \\
\hline \ecell & \ecell & \ecell & \ecell & 1 & 2 \\
\hline \ecell & \ecell & \ecell & 2 & 3 & 1 \\
\hline \ecell & 6 & 2 & \ecell & 4 & \ecell \\
\hline \ecell & \ecell & 1 & 5 & 2 & 4 \\
\hline 
\end{tabular}\\ 
6.5, size 12
 & 6.5, size 13
 & 6.5, size 14
 & 6.5, size 15
\\
\end{tabular} 
\end{center} 
\begin{center}
\begin{tabular}{c@{\hspace{1.0mm}}c@{\hspace{1.0mm}}c@{\hspace{1.0mm}}c@{\hspace{1.0mm}}}
\begin{tabular}{|c|c|c|c|c|c|}
\hline \ecell & \ecell & \ecell & \ecell & \ecell & \ecell \\
\hline \ecell & 1 & \ecell & 3 & \ecell & 5 \\
\hline \ecell & \ecell & \ecell & \ecell & 1 & 2 \\
\hline \ecell & \ecell & \ecell & 2 & 3 & 1 \\
\hline \ecell & 6 & 2 & 1 & 4 & \ecell \\
\hline \ecell & 3 & 1 & \ecell & 2 & 4 \\
\hline 
\end{tabular} &
\begin{tabular}{|c|c|c|c|c|c|}
\hline \ecell & \ecell & \ecell & \ecell & \ecell & \ecell \\
\hline \ecell & 1 & \ecell & 3 & \ecell & 5 \\
\hline \ecell & \ecell & \ecell & \ecell & 1 & 2 \\
\hline \ecell & \ecell & \ecell & 2 & 3 & 1 \\
\hline \ecell & \ecell & 2 & 1 & 4 & 3 \\
\hline \ecell & 3 & 1 & 5 & 2 & 4 \\
\hline 
\end{tabular} &
\begin{tabular}{|c|c|c|c|c|c|}
\hline \ecell & \ecell & \ecell & \ecell & \ecell & 6 \\
\hline \ecell & 1 & \ecell & 3 & \ecell & \ecell \\
\hline \ecell & \ecell & \ecell & \ecell & 1 & \ecell \\
\hline 4 & 5 & \ecell & \ecell & \ecell & \ecell \\
\hline 5 & 6 & \ecell & \ecell & \ecell & 4 \\
\hline \ecell & \ecell & 2 & \ecell & 4 & \ecell \\
\hline 
\end{tabular} &
\begin{tabular}{|c|c|c|c|c|c|}
\hline \ecell & \ecell & \ecell & \ecell & \ecell & 6 \\
\hline \ecell & 1 & \ecell & 3 & \ecell & \ecell \\
\hline \ecell & \ecell & \ecell & \ecell & 1 & \ecell \\
\hline \ecell & \ecell & \ecell & 1 & 2 & 3 \\
\hline 5 & \ecell & \ecell & 2 & \ecell & \ecell \\
\hline 6 & \ecell & 2 & \ecell & 4 & \ecell \\
\hline 
\end{tabular}\\ 
6.5, size 16
 & 6.5, size 17
 & 6.6, size 11
 & 6.6, size 12
\\
\end{tabular} 
\end{center} 
\begin{center}
\begin{tabular}{c@{\hspace{1.0mm}}c@{\hspace{1.0mm}}c@{\hspace{1.0mm}}c@{\hspace{1.0mm}}}
\begin{tabular}{|c|c|c|c|c|c|}
\hline \ecell & \ecell & \ecell & \ecell & \ecell & \ecell \\
\hline \ecell & 1 & \ecell & 3 & \ecell & 5 \\
\hline \ecell & \ecell & \ecell & \ecell & 1 & 2 \\
\hline \ecell & \ecell & \ecell & 1 & 2 & 3 \\
\hline 5 & 6 & \ecell & \ecell & \ecell & \ecell \\
\hline 6 & \ecell & 2 & \ecell & 4 & \ecell \\
\hline 
\end{tabular} &
\begin{tabular}{|c|c|c|c|c|c|}
\hline \ecell & \ecell & \ecell & \ecell & \ecell & \ecell \\
\hline \ecell & 1 & \ecell & 3 & \ecell & 5 \\
\hline \ecell & \ecell & \ecell & \ecell & 1 & 2 \\
\hline \ecell & \ecell & \ecell & 1 & 2 & 3 \\
\hline \ecell & 6 & 1 & \ecell & 3 & \ecell \\
\hline 6 & \ecell & 2 & \ecell & 4 & \ecell \\
\hline 
\end{tabular} &
\begin{tabular}{|c|c|c|c|c|c|}
\hline \ecell & \ecell & \ecell & \ecell & \ecell & \ecell \\
\hline \ecell & 1 & \ecell & 3 & \ecell & 5 \\
\hline \ecell & \ecell & \ecell & \ecell & 1 & 2 \\
\hline \ecell & \ecell & \ecell & 1 & 2 & 3 \\
\hline \ecell & \ecell & 1 & 2 & 3 & 4 \\
\hline 6 & \ecell & 2 & \ecell & 4 & \ecell \\
\hline 
\end{tabular} &
\begin{tabular}{|c|c|c|c|c|c|}
\hline \ecell & \ecell & \ecell & \ecell & \ecell & \ecell \\
\hline \ecell & 1 & \ecell & 3 & \ecell & 5 \\
\hline \ecell & \ecell & \ecell & \ecell & 1 & 2 \\
\hline \ecell & \ecell & 6 & 1 & 2 & \ecell \\
\hline \ecell & \ecell & 1 & 2 & 3 & 4 \\
\hline \ecell & 3 & \ecell & 5 & 4 & 1 \\
\hline 
\end{tabular}\\ 
6.6, size 13
 & 6.6, size 14
 & 6.6, size 15
 & 6.6, size 16
\\
\end{tabular} 
\end{center} 
\begin{center}
\begin{tabular}{c@{\hspace{1.0mm}}c@{\hspace{1.0mm}}c@{\hspace{1.0mm}}c@{\hspace{1.0mm}}}
\begin{tabular}{|c|c|c|c|c|c|}
\hline \ecell & \ecell & \ecell & \ecell & \ecell & \ecell \\
\hline \ecell & 1 & \ecell & 3 & \ecell & 5 \\
\hline \ecell & \ecell & \ecell & \ecell & 1 & 2 \\
\hline \ecell & \ecell & \ecell & 1 & 2 & 3 \\
\hline \ecell & \ecell & 1 & 2 & 3 & 4 \\
\hline \ecell & 3 & 2 & 5 & 4 & 1 \\
\hline 
\end{tabular} &
\begin{tabular}{|c|c|c|c|c|c|}
\hline \ecell & \ecell & \ecell & \ecell & \ecell & 6 \\
\hline \ecell & \ecell & 1 & \ecell & \ecell & \ecell \\
\hline \ecell & 1 & 2 & 5 & \ecell & \ecell \\
\hline 4 & 5 & \ecell & \ecell & \ecell & \ecell \\
\hline 5 & \ecell & \ecell & 3 & \ecell & \ecell \\
\hline \ecell & \ecell & \ecell & 2 & 3 & 1 \\
\hline 
\end{tabular} &
\begin{tabular}{|c|c|c|c|c|c|}
\hline \ecell & \ecell & \ecell & \ecell & \ecell & \ecell \\
\hline \ecell & \ecell & 1 & \ecell & \ecell & 5 \\
\hline 3 & \ecell & \ecell & \ecell & 6 & \ecell \\
\hline \ecell & \ecell & \ecell & 1 & 2 & 3 \\
\hline \ecell & 6 & 4 & \ecell & 1 & \ecell \\
\hline 6 & 4 & 5 & \ecell & \ecell & \ecell \\
\hline 
\end{tabular} &
\begin{tabular}{|c|c|c|c|c|c|}
\hline \ecell & \ecell & \ecell & \ecell & \ecell & \ecell \\
\hline \ecell & \ecell & 1 & \ecell & \ecell & 5 \\
\hline \ecell & 1 & 2 & \ecell & 6 & \ecell \\
\hline \ecell & \ecell & \ecell & 1 & 2 & 3 \\
\hline \ecell & 6 & \ecell & 3 & 1 & \ecell \\
\hline 6 & 4 & \ecell & \ecell & 3 & \ecell \\
\hline 
\end{tabular}\\ 
6.6, size 17
 & 6.7, size 12
 & 6.7, size 13
 & 6.7, size 14
\\
\end{tabular} 
\end{center} 
\begin{center}
\begin{tabular}{c@{\hspace{1.0mm}}c@{\hspace{1.0mm}}c@{\hspace{1.0mm}}c@{\hspace{1.0mm}}}
\begin{tabular}{|c|c|c|c|c|c|}
\hline \ecell & \ecell & \ecell & \ecell & \ecell & \ecell \\
\hline \ecell & \ecell & 1 & \ecell & \ecell & 5 \\
\hline \ecell & 1 & 2 & \ecell & 6 & \ecell \\
\hline \ecell & \ecell & \ecell & 1 & 2 & 3 \\
\hline \ecell & \ecell & 4 & 3 & 1 & 2 \\
\hline 6 & \ecell & \ecell & 2 & 3 & \ecell \\
\hline 
\end{tabular} &
\begin{tabular}{|c|c|c|c|c|c|}
\hline \ecell & \ecell & \ecell & \ecell & \ecell & \ecell \\
\hline \ecell & \ecell & 1 & \ecell & \ecell & 5 \\
\hline \ecell & 1 & 2 & \ecell & 6 & \ecell \\
\hline \ecell & \ecell & \ecell & 1 & 2 & 3 \\
\hline \ecell & \ecell & 4 & 3 & 1 & 2 \\
\hline \ecell & 4 & 5 & 2 & \ecell & 1 \\
\hline 
\end{tabular} &
\begin{tabular}{|c|c|c|c|c|c|}
\hline \ecell & \ecell & \ecell & \ecell & \ecell & \ecell \\
\hline \ecell & \ecell & 1 & \ecell & \ecell & 5 \\
\hline \ecell & 1 & 2 & \ecell & 6 & \ecell \\
\hline \ecell & 5 & 6 & 1 & 2 & 3 \\
\hline \ecell & \ecell & \ecell & 3 & 1 & 2 \\
\hline \ecell & \ecell & 5 & 2 & 3 & 1 \\
\hline 
\end{tabular} &
\begin{tabular}{|c|c|c|c|c|c|}
\hline \ecell & \ecell & \ecell & \ecell & \ecell & \ecell \\
\hline \ecell & \ecell & 1 & \ecell & \ecell & 5 \\
\hline \ecell & 1 & 2 & 5 & \ecell & 4 \\
\hline \ecell & \ecell & \ecell & 1 & 2 & 3 \\
\hline \ecell & \ecell & 4 & 3 & 1 & 2 \\
\hline \ecell & 4 & 5 & 2 & 3 & 1 \\
\hline 
\end{tabular}\\ 
6.7, size 15
 & 6.7, size 16
 & 6.7, size 17
 & 6.7, size 18
\\
\end{tabular} 
\end{center} 
\begin{center}
\begin{tabular}{c@{\hspace{1.0mm}}c@{\hspace{1.0mm}}c@{\hspace{1.0mm}}c@{\hspace{1.0mm}}}
\begin{tabular}{|c|c|c|c|c|c|}
\hline \ecell & \ecell & \ecell & \ecell & \ecell & \ecell \\
\hline \ecell & 1 & \ecell & 3 & 6 & \ecell \\
\hline \ecell & 5 & 1 & \ecell & \ecell & 4 \\
\hline 4 & \ecell & \ecell & \ecell & \ecell & 3 \\
\hline \ecell & \ecell & 6 & \ecell & \ecell & \ecell \\
\hline \ecell & \ecell & \ecell & \ecell & 3 & 2 \\
\hline 
\end{tabular} &
\begin{tabular}{|c|c|c|c|c|c|}
\hline \ecell & \ecell & \ecell & \ecell & \ecell & \ecell \\
\hline \ecell & 1 & \ecell & 3 & \ecell & 5 \\
\hline \ecell & \ecell & 1 & \ecell & 2 & 4 \\
\hline 4 & 6 & \ecell & \ecell & \ecell & \ecell \\
\hline \ecell & 3 & 6 & \ecell & \ecell & \ecell \\
\hline \ecell & \ecell & \ecell & \ecell & 3 & 2 \\
\hline 
\end{tabular} &
\begin{tabular}{|c|c|c|c|c|c|}
\hline \ecell & \ecell & \ecell & \ecell & \ecell & \ecell \\
\hline \ecell & 1 & \ecell & 3 & \ecell & 5 \\
\hline \ecell & \ecell & 1 & \ecell & 2 & 4 \\
\hline \ecell & \ecell & 2 & 5 & \ecell & \ecell \\
\hline \ecell & 3 & \ecell & \ecell & 4 & \ecell \\
\hline \ecell & 4 & 5 & \ecell & \ecell & 2 \\
\hline 
\end{tabular} &
\begin{tabular}{|c|c|c|c|c|c|}
\hline \ecell & \ecell & \ecell & \ecell & \ecell & \ecell \\
\hline \ecell & 1 & \ecell & 3 & \ecell & 5 \\
\hline \ecell & \ecell & 1 & \ecell & 2 & 4 \\
\hline \ecell & \ecell & 2 & 5 & \ecell & \ecell \\
\hline \ecell & 3 & 6 & \ecell & 4 & \ecell \\
\hline \ecell & 4 & \ecell & \ecell & 3 & 2 \\
\hline 
\end{tabular}\\ 
6.8, size 11
 & 6.8, size 12
 & 6.8, size 13
 & 6.8, size 14
\\
\end{tabular} 
\end{center} 
\begin{center}
\begin{tabular}{c@{\hspace{1.0mm}}c@{\hspace{1.0mm}}c@{\hspace{1.0mm}}c@{\hspace{1.0mm}}}
\begin{tabular}{|c|c|c|c|c|c|}
\hline \ecell & \ecell & \ecell & \ecell & \ecell & \ecell \\
\hline \ecell & 1 & \ecell & 3 & \ecell & 5 \\
\hline \ecell & \ecell & 1 & \ecell & 2 & 4 \\
\hline \ecell & \ecell & 2 & 5 & \ecell & \ecell \\
\hline 5 & \ecell & 6 & \ecell & 4 & 1 \\
\hline 6 & \ecell & 5 & \ecell & \ecell & 2 \\
\hline 
\end{tabular} &
\begin{tabular}{|c|c|c|c|c|c|}
\hline \ecell & \ecell & \ecell & \ecell & \ecell & \ecell \\
\hline \ecell & 1 & \ecell & 3 & \ecell & 5 \\
\hline \ecell & \ecell & 1 & 6 & 2 & \ecell \\
\hline \ecell & 6 & 2 & \ecell & 1 & 3 \\
\hline \ecell & 3 & 6 & \ecell & 4 & 1 \\
\hline \ecell & \ecell & \ecell & 1 & 3 & \ecell \\
\hline 
\end{tabular} &
\begin{tabular}{|c|c|c|c|c|c|}
\hline \ecell & \ecell & \ecell & \ecell & \ecell & \ecell \\
\hline \ecell & 1 & \ecell & 3 & \ecell & 5 \\
\hline 3 & 5 & \ecell & \ecell & \ecell & 4 \\
\hline 4 & \ecell & \ecell & 5 & 1 & 3 \\
\hline 5 & 3 & \ecell & \ecell & 4 & 1 \\
\hline \ecell & 4 & \ecell & \ecell & 3 & 2 \\
\hline 
\end{tabular} &
\begin{tabular}{|c|c|c|c|c|c|}
\hline \ecell & \ecell & \ecell & \ecell & \ecell & 6 \\
\hline \ecell & 1 & 4 & 3 & \ecell & \ecell \\
\hline \ecell & 5 & 1 & \ecell & \ecell & \ecell \\
\hline \ecell & \ecell & \ecell & \ecell & \ecell & \ecell \\
\hline \ecell & 4 & \ecell & \ecell & 1 & \ecell \\
\hline 6 & \ecell & \ecell & \ecell & \ecell & 2 \\
\hline 
\end{tabular}\\ 
6.8, size 15
 & 6.8, size 16
 & 6.8, size 17
 & 6.9, size 10
\\
\end{tabular} 
\end{center} 
\begin{center}
\begin{tabular}{c@{\hspace{1.0mm}}c@{\hspace{1.0mm}}c@{\hspace{1.0mm}}c@{\hspace{1.0mm}}}
\begin{tabular}{|c|c|c|c|c|c|}
\hline \ecell & \ecell & \ecell & \ecell & \ecell & \ecell \\
\hline \ecell & 1 & \ecell & 3 & \ecell & 5 \\
\hline \ecell & \ecell & 1 & 6 & 2 & \ecell \\
\hline \ecell & \ecell & 2 & \ecell & \ecell & \ecell \\
\hline 5 & 4 & \ecell & \ecell & \ecell & 3 \\
\hline \ecell & 3 & \ecell & \ecell & \ecell & \ecell \\
\hline 
\end{tabular} &
\begin{tabular}{|c|c|c|c|c|c|}
\hline \ecell & \ecell & \ecell & \ecell & \ecell & \ecell \\
\hline \ecell & 1 & \ecell & 3 & \ecell & 5 \\
\hline \ecell & \ecell & 1 & \ecell & 2 & 4 \\
\hline \ecell & \ecell & 2 & 5 & \ecell & \ecell \\
\hline \ecell & \ecell & 6 & 2 & 1 & \ecell \\
\hline \ecell & 3 & \ecell & \ecell & \ecell & \ecell \\
\hline 
\end{tabular} &
\begin{tabular}{|c|c|c|c|c|c|}
\hline \ecell & \ecell & \ecell & \ecell & \ecell & \ecell \\
\hline \ecell & 1 & \ecell & 3 & \ecell & 5 \\
\hline \ecell & \ecell & 1 & \ecell & 2 & 4 \\
\hline \ecell & \ecell & 2 & 5 & \ecell & \ecell \\
\hline \ecell & \ecell & 6 & 2 & 1 & \ecell \\
\hline 6 & \ecell & \ecell & \ecell & 4 & \ecell \\
\hline 
\end{tabular} &
\begin{tabular}{|c|c|c|c|c|c|}
\hline \ecell & \ecell & \ecell & \ecell & \ecell & \ecell \\
\hline \ecell & 1 & \ecell & 3 & \ecell & 5 \\
\hline \ecell & \ecell & 1 & \ecell & 2 & 4 \\
\hline \ecell & \ecell & 2 & 5 & \ecell & \ecell \\
\hline \ecell & \ecell & \ecell & 2 & 1 & 3 \\
\hline \ecell & 3 & 5 & \ecell & 4 & \ecell \\
\hline 
\end{tabular}\\ 
6.9, size 11
 & 6.9, size 12
 & 6.9, size 13
 & 6.9, size 14
\\
\end{tabular} 
\end{center} 
\begin{center}
\begin{tabular}{c@{\hspace{1.0mm}}c@{\hspace{1.0mm}}c@{\hspace{1.0mm}}c@{\hspace{1.0mm}}}
\begin{tabular}{|c|c|c|c|c|c|}
\hline \ecell & \ecell & \ecell & \ecell & \ecell & \ecell \\
\hline \ecell & 1 & \ecell & 3 & \ecell & 5 \\
\hline \ecell & \ecell & 1 & \ecell & 2 & 4 \\
\hline 4 & 6 & \ecell & \ecell & \ecell & \ecell \\
\hline 5 & 4 & 6 & \ecell & \ecell & 3 \\
\hline 6 & 3 & 5 & \ecell & \ecell & \ecell \\
\hline 
\end{tabular} &
\begin{tabular}{|c|c|c|c|c|c|}
\hline \ecell & \ecell & \ecell & \ecell & \ecell & \ecell \\
\hline \ecell & 1 & \ecell & 3 & \ecell & 5 \\
\hline 3 & \ecell & \ecell & \ecell & 2 & 4 \\
\hline 4 & \ecell & \ecell & 5 & 3 & 1 \\
\hline 5 & 4 & \ecell & 2 & 1 & 3 \\
\hline \ecell & \ecell & \ecell & \ecell & \ecell & 2 \\
\hline 
\end{tabular} &
\begin{tabular}{|c|c|c|c|c|c|}
\hline \ecell & 2 & 3 & 4 & 5 & 6 \\
\hline \ecell & \ecell & 4 & 3 & \ecell & 5 \\
\hline \ecell & 5 & \ecell & \ecell & \ecell & 4 \\
\hline \ecell & \ecell & \ecell & \ecell & \ecell & \ecell \\
\hline \ecell & 4 & 6 & \ecell & \ecell & 3 \\
\hline \ecell & 3 & 5 & \ecell & 4 & 2 \\
\hline 
\end{tabular} &
\begin{tabular}{|c|c|c|c|c|c|}
\hline 1 & \ecell & \ecell & \ecell & 5 & 6 \\
\hline \ecell & \ecell & \ecell & 3 & \ecell & \ecell \\
\hline \ecell & \ecell & 1 & \ecell & \ecell & \ecell \\
\hline 4 & \ecell & \ecell & \ecell & \ecell & \ecell \\
\hline \ecell & 3 & \ecell & 2 & \ecell & \ecell \\
\hline 6 & \ecell & \ecell & \ecell & \ecell & 1 \\
\hline 
\end{tabular}\\ 
6.9, size 15
 & 6.9, size 16
 & 6.9, size 17
 & 6.10, size 10
\\
\end{tabular} 
\end{center} 
\begin{center}
\begin{tabular}{c@{\hspace{1.0mm}}c@{\hspace{1.0mm}}c@{\hspace{1.0mm}}c@{\hspace{1.0mm}}}
\begin{tabular}{|c|c|c|c|c|c|}
\hline \ecell & \ecell & \ecell & \ecell & \ecell & 6 \\
\hline \ecell & 1 & \ecell & 3 & \ecell & \ecell \\
\hline \ecell & \ecell & 1 & \ecell & \ecell & \ecell \\
\hline 4 & \ecell & \ecell & \ecell & \ecell & \ecell \\
\hline \ecell & 3 & \ecell & 2 & 1 & \ecell \\
\hline 6 & \ecell & \ecell & 5 & 3 & \ecell \\
\hline 
\end{tabular} &
\begin{tabular}{|c|c|c|c|c|c|}
\hline \ecell & \ecell & \ecell & \ecell & \ecell & \ecell \\
\hline \ecell & 1 & \ecell & 3 & \ecell & 5 \\
\hline \ecell & \ecell & 1 & \ecell & \ecell & 2 \\
\hline \ecell & \ecell & \ecell & 1 & \ecell & 3 \\
\hline 5 & \ecell & 6 & \ecell & \ecell & \ecell \\
\hline 6 & 4 & \ecell & \ecell & 3 & \ecell \\
\hline 
\end{tabular} &
\begin{tabular}{|c|c|c|c|c|c|}
\hline \ecell & \ecell & \ecell & \ecell & \ecell & \ecell \\
\hline \ecell & 1 & \ecell & 3 & \ecell & 5 \\
\hline \ecell & \ecell & 1 & \ecell & \ecell & 2 \\
\hline \ecell & 6 & \ecell & 1 & \ecell & \ecell \\
\hline \ecell & 3 & 6 & \ecell & 1 & \ecell \\
\hline 6 & \ecell & 2 & 5 & \ecell & \ecell \\
\hline 
\end{tabular} &
\begin{tabular}{|c|c|c|c|c|c|}
\hline \ecell & \ecell & \ecell & \ecell & \ecell & \ecell \\
\hline \ecell & 1 & \ecell & 3 & \ecell & 5 \\
\hline \ecell & \ecell & 1 & \ecell & \ecell & 2 \\
\hline \ecell & \ecell & \ecell & 1 & \ecell & 3 \\
\hline \ecell & 3 & \ecell & 2 & 1 & 4 \\
\hline 6 & 4 & \ecell & \ecell & 3 & \ecell \\
\hline 
\end{tabular}\\ 
6.10, size 11
 & 6.10, size 12
 & 6.10, size 13
 & 6.10, size 14
\\
\end{tabular} 
\end{center} 
\begin{center}
\begin{tabular}{c@{\hspace{1.0mm}}c@{\hspace{1.0mm}}c@{\hspace{1.0mm}}c@{\hspace{1.0mm}}}
\begin{tabular}{|c|c|c|c|c|c|}
\hline \ecell & \ecell & \ecell & \ecell & \ecell & \ecell \\
\hline \ecell & 1 & \ecell & 3 & \ecell & 5 \\
\hline \ecell & \ecell & 1 & \ecell & \ecell & 2 \\
\hline \ecell & \ecell & \ecell & 1 & \ecell & 3 \\
\hline \ecell & 3 & 6 & 2 & 1 & \ecell \\
\hline \ecell & 4 & 2 & \ecell & 3 & 1 \\
\hline 
\end{tabular} &
\begin{tabular}{|c|c|c|c|c|c|}
\hline \ecell & \ecell & \ecell & \ecell & \ecell & \ecell \\
\hline \ecell & 1 & \ecell & 3 & \ecell & 5 \\
\hline \ecell & \ecell & 1 & \ecell & \ecell & 2 \\
\hline \ecell & \ecell & \ecell & 1 & \ecell & 3 \\
\hline \ecell & 3 & \ecell & 2 & 1 & 4 \\
\hline \ecell & 4 & 2 & 5 & 3 & 1 \\
\hline 
\end{tabular} &
\begin{tabular}{|c|c|c|c|c|c|}
\hline \ecell & \ecell & \ecell & \ecell & \ecell & \ecell \\
\hline \ecell & 1 & \ecell & 3 & 6 & \ecell \\
\hline \ecell & \ecell & 1 & 6 & \ecell & \ecell \\
\hline \ecell & 6 & 5 & 1 & 2 & 3 \\
\hline \ecell & 3 & 6 & 2 & 1 & \ecell \\
\hline \ecell & \ecell & \ecell & 5 & 3 & 1 \\
\hline 
\end{tabular} &
\begin{tabular}{|c|c|c|c|c|c|}
\hline \ecell & \ecell & \ecell & \ecell & \ecell & 6 \\
\hline 2 & 1 & \ecell & \ecell & \ecell & \ecell \\
\hline \ecell & 4 & \ecell & \ecell & 1 & \ecell \\
\hline \ecell & \ecell & \ecell & \ecell & 3 & \ecell \\
\hline \ecell & 3 & \ecell & 1 & \ecell & \ecell \\
\hline 6 & \ecell & \ecell & \ecell & \ecell & 2 \\
\hline 
\end{tabular}\\ 
6.10, size 15
 & 6.10, size 16
 & 6.10, size 17
 & 6.11, size 10
\\
\end{tabular} 
\end{center} 
\begin{center}
\begin{tabular}{c@{\hspace{1.0mm}}c@{\hspace{1.0mm}}c@{\hspace{1.0mm}}c@{\hspace{1.0mm}}}
\begin{tabular}{|c|c|c|c|c|c|}
\hline \ecell & \ecell & \ecell & \ecell & \ecell & \ecell \\
\hline \ecell & 1 & \ecell & \ecell & \ecell & 3 \\
\hline \ecell & \ecell & \ecell & \ecell & 1 & 5 \\
\hline \ecell & \ecell & \ecell & 2 & 3 & 1 \\
\hline 5 & \ecell & 6 & \ecell & \ecell & \ecell \\
\hline 6 & \ecell & \ecell & \ecell & \ecell & 2 \\
\hline 
\end{tabular} &
\begin{tabular}{|c|c|c|c|c|c|}
\hline \ecell & \ecell & \ecell & \ecell & \ecell & \ecell \\
\hline \ecell & 1 & \ecell & \ecell & \ecell & 3 \\
\hline \ecell & \ecell & \ecell & \ecell & 1 & 5 \\
\hline \ecell & \ecell & \ecell & 2 & 3 & 1 \\
\hline \ecell & \ecell & 6 & 1 & 2 & \ecell \\
\hline 6 & 5 & \ecell & \ecell & \ecell & \ecell \\
\hline 
\end{tabular} &
\begin{tabular}{|c|c|c|c|c|c|}
\hline \ecell & \ecell & \ecell & \ecell & \ecell & \ecell \\
\hline \ecell & 1 & \ecell & \ecell & \ecell & 3 \\
\hline \ecell & \ecell & \ecell & \ecell & 1 & 5 \\
\hline \ecell & \ecell & \ecell & 2 & 3 & 1 \\
\hline 5 & \ecell & \ecell & \ecell & 2 & 4 \\
\hline 6 & \ecell & \ecell & 3 & 4 & \ecell \\
\hline 
\end{tabular} &
\begin{tabular}{|c|c|c|c|c|c|}
\hline \ecell & \ecell & \ecell & \ecell & \ecell & \ecell \\
\hline \ecell & 1 & \ecell & \ecell & \ecell & 3 \\
\hline \ecell & \ecell & \ecell & \ecell & 1 & 5 \\
\hline \ecell & \ecell & \ecell & 2 & 3 & 1 \\
\hline 5 & \ecell & \ecell & \ecell & 2 & 4 \\
\hline \ecell & \ecell & 1 & 3 & 4 & 2 \\
\hline 
\end{tabular}\\ 
6.11, size 11
 & 6.11, size 12
 & 6.11, size 13
 & 6.11, size 14
\\
\end{tabular} 
\end{center} 
\begin{center}
\begin{tabular}{c@{\hspace{1.0mm}}c@{\hspace{1.0mm}}c@{\hspace{1.0mm}}c@{\hspace{1.0mm}}}
\begin{tabular}{|c|c|c|c|c|c|}
\hline \ecell & \ecell & \ecell & \ecell & \ecell & \ecell \\
\hline \ecell & 1 & \ecell & \ecell & \ecell & 3 \\
\hline \ecell & \ecell & \ecell & \ecell & 1 & 5 \\
\hline \ecell & \ecell & \ecell & 2 & 3 & 1 \\
\hline \ecell & 3 & 6 & 1 & 2 & \ecell \\
\hline \ecell & \ecell & 1 & 3 & 4 & 2 \\
\hline 
\end{tabular} &
\begin{tabular}{|c|c|c|c|c|c|}
\hline \ecell & \ecell & \ecell & \ecell & \ecell & \ecell \\
\hline \ecell & 1 & \ecell & \ecell & \ecell & 3 \\
\hline \ecell & \ecell & \ecell & \ecell & 1 & 5 \\
\hline \ecell & \ecell & \ecell & 2 & 3 & 1 \\
\hline \ecell & 3 & \ecell & 1 & 2 & 4 \\
\hline \ecell & 5 & 1 & 3 & 4 & 2 \\
\hline 
\end{tabular} &
\begin{tabular}{|c|c|c|c|c|c|}
\hline \ecell & \ecell & \ecell & \ecell & \ecell & \ecell \\
\hline \ecell & 1 & \ecell & 5 & 6 & \ecell \\
\hline 3 & 4 & \ecell & 6 & 1 & \ecell \\
\hline \ecell & 6 & \ecell & \ecell & 3 & \ecell \\
\hline \ecell & 3 & 6 & 1 & \ecell & \ecell \\
\hline 6 & 5 & 1 & 3 & 4 & \ecell \\
\hline 
\end{tabular} &
\begin{tabular}{|c|c|c|c|c|c|}
\hline \ecell & \ecell & \ecell & \ecell & \ecell & \ecell \\
\hline \ecell & \ecell & 1 & \ecell & \ecell & 5 \\
\hline \ecell & 1 & 2 & \ecell & 6 & \ecell \\
\hline \ecell & \ecell & \ecell & \ecell & 3 & \ecell \\
\hline \ecell & \ecell & 4 & \ecell & 2 & 3 \\
\hline 6 & \ecell & \ecell & 3 & \ecell & \ecell \\
\hline 
\end{tabular}\\ 
6.11, size 15
 & 6.11, size 16
 & 6.11, size 17
 & 6.12, size 11
\\
\end{tabular} 
\end{center} 
\begin{center}
\begin{tabular}{c@{\hspace{1.0mm}}c@{\hspace{1.0mm}}c@{\hspace{1.0mm}}c@{\hspace{1.0mm}}}
\begin{tabular}{|c|c|c|c|c|c|}
\hline \ecell & \ecell & \ecell & \ecell & \ecell & \ecell \\
\hline \ecell & \ecell & 1 & \ecell & \ecell & 5 \\
\hline \ecell & 1 & 2 & \ecell & 6 & \ecell \\
\hline \ecell & \ecell & \ecell & \ecell & 3 & \ecell \\
\hline \ecell & \ecell & 4 & \ecell & 2 & 3 \\
\hline 6 & 4 & 5 & \ecell & \ecell & \ecell \\
\hline 
\end{tabular} &
\begin{tabular}{|c|c|c|c|c|c|}
\hline \ecell & \ecell & \ecell & \ecell & \ecell & \ecell \\
\hline \ecell & \ecell & 1 & \ecell & \ecell & 5 \\
\hline \ecell & 1 & 2 & \ecell & 6 & \ecell \\
\hline \ecell & \ecell & \ecell & \ecell & 3 & \ecell \\
\hline \ecell & \ecell & 4 & 1 & 2 & 3 \\
\hline \ecell & 4 & 5 & \ecell & \ecell & 2 \\
\hline 
\end{tabular} &
\begin{tabular}{|c|c|c|c|c|c|}
\hline \ecell & \ecell & \ecell & \ecell & \ecell & \ecell \\
\hline \ecell & \ecell & 1 & \ecell & \ecell & 5 \\
\hline \ecell & 1 & 2 & \ecell & 6 & \ecell \\
\hline \ecell & \ecell & \ecell & \ecell & 3 & \ecell \\
\hline \ecell & 6 & 4 & \ecell & 2 & 3 \\
\hline \ecell & 4 & \ecell & 3 & 1 & 2 \\
\hline 
\end{tabular} &
\begin{tabular}{|c|c|c|c|c|c|}
\hline \ecell & \ecell & \ecell & \ecell & \ecell & \ecell \\
\hline \ecell & \ecell & 1 & \ecell & \ecell & 5 \\
\hline \ecell & 1 & 2 & \ecell & 6 & \ecell \\
\hline \ecell & \ecell & \ecell & 2 & \ecell & 1 \\
\hline \ecell & \ecell & 4 & 1 & 2 & 3 \\
\hline \ecell & 4 & 5 & 3 & \ecell & 2 \\
\hline 
\end{tabular}\\ 
6.12, size 12
 & 6.12, size 13
 & 6.12, size 14
 & 6.12, size 15
\\
\end{tabular} 
\end{center} 
\begin{center}
\begin{tabular}{c}
\begin{tabular}{|c|c|c|c|c|c|}
\hline \ecell & \ecell & \ecell & \ecell & \ecell & \ecell \\
\hline \ecell & \ecell & 1 & \ecell & \ecell & 5 \\
\hline \ecell & 1 & 2 & \ecell & 6 & 4 \\
\hline \ecell & 5 & 6 & 2 & \ecell & 1 \\
\hline 5 & 6 & 4 & 1 & \ecell & \ecell \\
\hline \ecell & 4 & 5 & \ecell & \ecell & \ecell \\
\hline 
\end{tabular}\\ 
6.12, size 16
\\
\end{tabular} 
\end{center} 
\end{footnotesize}

\iflatexml
  \chapter{Construction for Latin interchanges in a back-circulant array}\label{app3}
\else
  \refstepcounter{chapter}
  \chapter*{Appendix \thechapter\newline\newline Construction for Latin interchanges in a back-circulant array}\label{app3}
  \addcontentsline{toc}{chapter}{\hspace{1.5em}Appendix~\thechapter\quad Construction for Latin interchanges in a back-circulant array}
  \vspace{0.5cm}
\fi

This Appendix gives a construction for Latin interchanges referred
to in Chapter~\ref{ch6}, which are called ``Variety 3'' Latin interchanges there.

The construction given here is that of~\cite{MR1758263}, and results
in a Latin interchange $I$ in a back-circulant Latin square.  Recall that the
completion of the critical set $D$ given in Theorem~\ref{thm636} resulted in an $n
\times n$ Latin square, denoted $\mathcal{LD}$, of which the first
$\displaystyle{\frac{n}{2}}$ rows were the same as an $n \times n$ back-circulant
Latin square.  $D$ contains entries from the first $\displaystyle{\frac{n}{2}}$
rows of $LD$, and the following result gives a Latin interchange $I$
which intersects $D$ in any given cell $(x,y)$ in those rows.

Let $\mathcal{A}$ denote the Latin subrectangle in $\mathcal{LD^{T}}$ (the
transpose of $\mathcal{LD}$)
bounded by the entries
$(x, y; y+x)$, $(n-1, y; y - 1)$,
$(x, \displaystyle{\frac{n}{2}} - 1; \displaystyle{\frac{n}{2}} - 1 + x)$, and
$(n-1, \displaystyle{\frac{n}{2}} - 1; \displaystyle{\frac{n}{2}} - 2)$.
All future row and column references are relative to this subrectangle;
that is, a reference to the entry $(i, j; k)$ means the entry $(i-x,j-y;k)$
in $\mathcal{LD^{T}}$.

Let $c = \displaystyle{\frac{n}{2}} - y$, $r = n - x$, and $e = n + 1 - c$.

Define the sequence of numbers $\alpha_{1}, \alpha_{2}, ..., \alpha_{P}$ to be integers where
\begin{eqnarray*}
\alpha_{1} & = & c - 1 {\rm ~(mod~}e) {\rm ~and,~for~} i \geq 2, \\
\alpha_{i} & = & \alpha_{i-1} {\rm ~(mod~}(e-\alpha_{1}-...-\alpha_{i-1})).
\end{eqnarray*}
Let $P$ be the value such that $\alpha_{P} \neq 0$ and $\alpha_{P+i} = 0$ for all $i > 0$.
For $i = 1, 2, ..., P$, let $\delta_{i} = \alpha_{1} + \alpha_{2} + ... + \alpha_{i}$.
Define
\begin{eqnarray*}
A_{0} &=& \{ (0,0; x+y), (0,c-1; c-1+x+y) \}, {\rm ~and~if~} \alpha_{1} \neq c - 1 {\rm ~define}\\
B_{0} &=& \{ (c - 1 - ae, ae; c-1+x+y), (c - 1 - ae, (a + 1) e; x+y) \\ 
&& \quad{}\mid  0 \leq a \leq \displaystyle{\frac{c-1-\alpha_{1}}{e}} - 1 \}.
\end{eqnarray*}
If $\alpha_{1} \neq 0$, define
\begin{eqnarray*}
A_{1} &=& \{ (e, c - 1 - \alpha_{1}; c-1+e-\alpha_{1}+x+y), (e,c-1;x+y) \},
\end{eqnarray*}
and if $\alpha_{1} \neq \alpha_{2}$ define
\begin{eqnarray*}
B_{1} &=& \{ (\alpha_{1} - a(e-\alpha_{1}), c - 1 - \alpha_{1}; c-1+x+y ), \\
&& \quad{}(\alpha_{1} - a (e - \alpha_{1}), c - 1 - \alpha_{1} + (a+1)(e-\alpha_{1}); c-1+e-\alpha_{1}+x+y) \\
&& \quad{}\mid 0 \leq a \leq \displaystyle{\frac{\alpha_{1}-\alpha_{2}}{e-\alpha_{1}}} - 1 \}.
\end{eqnarray*}
If $P \geq 2$, for $2 \leq i \leq P$, define
\begin{eqnarray*}
A_{i} &=& \{ (e - \delta_{i-1}, c - 1 - \alpha_{i} ; c-1+e-\delta_{i}+x+y), \\
&& \quad{}(e - \delta_{i-1}, c-1; c-1+e-\delta_{i-1}+x+y) \} 
\end{eqnarray*}
and if $\alpha_{i} \neq \alpha_{i+1}$, define
\begin{eqnarray*}
B_{i} &=& \{ ( \alpha_{i} - (e-\delta_{i})a, c - 1 - \alpha_{i} + a (e-\delta_{i}); c-1+x+y ), \\
&& \quad{}(\alpha_{i} - a(e-\delta_{i}), c - 1 - \alpha_{i} + (a+1)(e-\delta_{i}); c-1+e-\delta_{i}+x+y) \mid  \\
&& \quad{}0 \leq a \leq \displaystyle{\frac{\alpha_{i}-\alpha_{i+1}}{e-\delta_{i}}} - 1 \}.
\end{eqnarray*}
Then the set $I = A_{0} \cup B_{0} \cup A_{1} \cup B_{1} \cup ... \cup A_{P} \cup B_{P}$ is the required Latin interchange.

\addcontentsline{toc}{chapter}{\protect\numberline{\ }{Bibliography}}


\end{document}